\documentclass[a4paper]{article}

\usepackage{amsthm}
\usepackage{amsmath}
\usepackage{amsfonts}
\usepackage{amssymb}

\usepackage[utf8]{inputenc}
\usepackage[T1]{fontenc}
\usepackage[english]{babel}
\usepackage{verbatim}
\usepackage{pgfplots}
\usepackage{xcolor}
\usepackage[hidelinks]{hyperref}
\usepackage{enumerate}
\usepackage{todonotes}
\usepackage{tikz-cd}

\usepackage[lmargin = 2cm,rmargin = 3cm]{geometry}

\newcounter{theo}
\numberwithin{equation}{section}

\theoremstyle{plain}
\newtheorem{theorem}[theo]{Theorem}
\newtheorem{lemma}[theo]{Lemma}
\newtheorem{corollary}[theo]{Corollary}

\theoremstyle{remark}
\newtheorem{remark}[theo]{Remark}
\newtheorem{definition}[theo]{Definition}

\makeatletter
\renewcommand*{\eqref}[1]{%
  \hyperref[{#1}]{\textup{\tagform@{\ref*{#1}}}}%
}
\makeatother

\newcommand{\R}{\mathbb{R}}

\newcommand{\norm}[1]{\left\vert\left\vert #1\right\vert\right\vert}

\newcommand{\abs}[1]{\left\vert #1 \right\vert}
\newcommand{\scalar}[1]{\langle #1 \rangle}
\newcommand{\mySpan}{\mathrm{span}}
\newcommand{\supp}{\mathrm{supp}}

\usepackage[backend = biber,style = alphabetic]{biblatex}
\begin{document}

\title{A stochastic reconstruction theorem}
\author{Hannes Kern \\ Technische Universität Berlin \\ Email: hake@tu-berlin.de}

\maketitle

\begin{abstract}
In a recent landmark paper, Khoa Lê (2020) established a stochastic sewing lemma which since has found many applications in stochastic analysis.
He further conjectured that a similar result may hold in the context of the reconstruction theorem within Hairer's regularity structures. 
The purpose of this article is to provide such a stochastic reconstruction theorem. We also discuss two variations of this theorem, motivated by different constructions of stochastic integration against white noise. Our formulation makes use of the distributional viewpoint of Caravenna--Zambotti (2021).
\end{abstract}

\tableofcontents

\section{Introduction}

\subsection{Stochastic reconstruction in a nutshell}

In recent years, the reconstruction theorem from Martin Hairer's regularity structure theory \cite{Hairer} has gained popularity in the mathematical community, not just within the said theory, but also as a standalone theorem (see \cite{caravenna}, \cite{reconProof}).

It answers the following question: Given a family of distributions $(F_x)_{x\in\mathbb R^d}$, is it possible to find a single distribution $f$, which around any point $x\in\R^d$ locally looks like $F_x$? The most simple example for such a situation is given by the Young product: Given an $0<\alpha$-Hölder continuous function $g$ and a distribution $h$, which is  $0>\beta$-Hölder continuous, with $\alpha+\beta>0$, it is possible to define a canonical product $g\cdot h$ even for non-smooth $g$. E.g. \cite{BCD} and \cite{RS} have used paraproducts to construct this product. 

As it turns out, there is a clear canonical local approximation of $h\cdot g$: We can approximate $g$ locally with its Taylor polynomial $T_x g$ around any $x\in\R^d$, and set $F_x = h\cdot T_x g$, as $T_x g$ is a smooth function. The only ingredient missing is a way to reconstruct the distribution from its local approximations.

To return to the general setting, let us consider a family of distributions $(F_x)_{x\in\R^d}$ and let us formalize what it means to ``locally look like $F_x$'' around $x$: For a smooth, compactly supported mollifier $\psi$, the rescaled function $\psi_x^\epsilon$ approximates a Dirac distribution $\delta_x$ as $\epsilon\to 0$. Since we formally want ``$F_x(x) = f(x)$'' to hold, we will require that $(f-F_x)(\psi_x^\epsilon)$ vanishes as $\epsilon\to 0$. To be precise, the rigorous condition to ``locally look like $f$'' is given by

\begin{equation}\label{caravenna_result}
\abs{(f-F_x)(\psi_x^\epsilon)}\lesssim \epsilon^\lambda
\end{equation}

\noindent for some $\lambda>0$. For this to be possible, $(F_x)_{x\in\R^d}$ will need to fulfill some kind of continuity assumption: For example, it suffices that $\abs{(F_x-F_y)(\psi_x^\epsilon)}\lesssim \epsilon^\alpha\abs{x-y}^{\gamma-\alpha}$ for some $\alpha<0$, which is the condition the germs in regularity structure fulfill \cite{Hairer}. Caravenna and Zambotti managed to find the optimal condition called coherence in \cite{caravenna}, which ensures \eqref{caravenna_result} and is given as follows: $(F_x)_{x\in\R^d}$ is called coherent, if

\begin{equation}\label{ineq:coherenceIntro}
\abs{(F_x-F_y)(\psi_x^\epsilon)}\lesssim \epsilon^\alpha(\abs{x-y}+\epsilon)^{\gamma-\alpha}
\end{equation}

\noindent for some $\alpha\in\R$ (typically negative) and a $\gamma>0$. Under this assumption, they proved the following theorem:

\begin{theorem}[\cite{caravenna}]
Let $(F_x)_{x\in\R^d}$ be a coherent germ for some $\gamma>0$. Then there exists a unique distribution $f$ with

\begin{equation*}
\abs{(f-F_x)(\psi_x^\epsilon)}\lesssim \epsilon^\gamma.
\end{equation*}
\end{theorem}

\begin{remark}
They also showed that it is possible to find such a distribution $f$ for $\gamma<0$. However, this distribution is no longer unique. Also, note that for $\gamma<0$, the bound $\epsilon^\gamma$ does explode, so the condition is much weaker. In this article, we will concentrate on the $\gamma>0$ case.
\end{remark}

\noindent At this point, the reconstruction theorem is a purely analytic tool. This is surprising, as it is often applied in SPDEs, see for example  \cite{roughVol}, \cite{brault}, \cite{gPam}, \cite{Hairer}, and many more. In all those examples, the reconstruction theorem only uses the analytic properties of the stochastic processes, neglecting their stochastic properties.

This naturally leads to the question, of whether there exists a stochastic version of the reconstruction theorem, especially if one considers its close connection to the sewing lemma \cite{pradelle}, \cite{gubinelli}: In rough paths theory, rough integrals can be constructed with the sewing lemma (see \cite{roughPaths},\cite{lyons}), whereas in regularity structure theory the reconstruction theorem is used to construct the corresponding products. And in \cite{le}, Khoa Lê showed the existence of a stochastic version of the sewing lemma.

It would be interesting to generalize his techniques to the multidimensional setting of the reconstruction theorem. Let us give two possible applications such a stochastic reconstruction theorem could have:

\begin{itemize}

\item In the paper \cite{roughSDE}, Peter Fritz and Khoa Lê use the stochastic sewing lemma to generate a mixed theory of rough paths and stochastic differential equations to solve equations of the form

\begin{equation*}
dY_t = b(t,Y_t)dt + \sigma(t,Y_t)dB_t + f(t,Y_t)d\mathbb X_t\,,
\end{equation*}

\noindent where $B_t$ is a Brownian motion and $\mathbb X_t$ is a deterministic rough path. This is outside of the scope of both Itô-calculus as well as rough path theory since the construction of the integral $\int f(t,Y_t)d\mathbb X_t$ requires one to understand the integration of a random process against a rough path. One can imagine similar problems for SPDEs, for example, one can look at the multiplicative heat equation with common noise

\begin{equation*}
du-\Delta udt = f(u)d\bar W + g(u) dW\,,
\end{equation*}

\noindent for a two independent Brownian sheets $\bar W$ and $W$ with spatial dimension $d=1$. One usually wants to treat the common noise $dW$ as a deterministic distribution, so it would make sense to treat it with the help of regularity structure theory. However, if one wants to keep $d\bar W$ as a stochastic integral, one would need a mixed theory of SPDEs and regularity structures. While we are not yet able to present such a theory, we hope that our stochastic reconstruction theorem provides the first step toward it.

\item The stochastic sewing lemma has shown tremendously useful in analyzing regularization by noise phenomena \cite{regByNoise}. In recent years, there has been a growing interest in studying such phenomenons in multi-parameter settings, e.g. \cite{waveEquation}. This is currently done with a deterministic multi-parameter sewing lemma, and the authors themselves point out that many more interesting noises like fractional Brownian sheet \cite{fBrownianSheet} will most likely require a stochastic multi-parameter technique. We conjecture that using the stochastic reconstruction theorem instead of the deterministic multi-parameter sewing lemma would allow one to extend the results of \cite{waveEquation} to more general noises.
\end{itemize}

\begin{remark}
We need to point out that since the inception of this paper's ideas, the regularization by noise community has moved towards favoring sewing techniques over reconstruction techniques. A stochastic multi-parameter sewing lemma is being researched as we are writing this paper, and while we believe that there is plenty enough reason to research both reconstruction and sewing techniques, it is at the moment not clear if stochastic reconstruction will hold a general advantage over stochastic multi-parameter sewing in the regularization by noise field, aside from being first.
\end{remark}

In this paper we present such a generalization of Lê's results to the multidimensional setting of the reconstruction theorem: Consider a family of random distributions $(F_x)_{x\in\R^d}$ with $2\le p<\infty$ (see Definition \ref{def:randomGerm}), which is adapted to a filtration $(\mathcal F_y)_{y\in\mathbb R}$ (see Definition \ref{stochDimension}). The main idea of stochastic reconstruction is to enhance the coherence property \eqref{ineq:coherenceIntro} by an additional condition on the conditional expectation of $F_x$ as follows:

\begin{align*}
\norm{(F_x-F_y)(\psi_x^\epsilon)}_{L_p}\lesssim \epsilon^\alpha(\abs{x-y}+\epsilon)^{\gamma-\alpha} \\
\norm{E^{\mathcal F_{y_1}}(F_x-F_y)(\psi_x^\epsilon)}_{L_p}\lesssim \epsilon^\alpha(\abs{x-y}+\epsilon)^{\tilde \gamma-\alpha},
\end{align*}

\noindent where $y = (y_1,\dots,y_d), x = (x_1,\dots,x_d)$ such that $y_1\le x_1$ and $\psi$ has a positive support (i.e. $\supp(\psi)\in [0,\infty)^d$) and for some $\tilde \gamma > \gamma$, so the conditional expectation is better behaved than the germ itself. We then hope, that $(F_x)_{x\in\R^d}$ can ``borrow regularity'' from its conditional expectation.

In the case of stochastic sewing, this was achieved based on a Doob-Meyer decomposition. Khoa Lê managed to borrow up to $\frac 12$ regularity from the conditional expectation of the terms corresponding to $(F_x-F_y)(\psi_x^\epsilon)$, which indicates that we should assume $\tilde\gamma = \gamma+\frac 12$. 

The interesting result of our analysis is that in the multidimensional setting, this technique can be used repeatedly for several directions to gain a higher improvement in regularity. We thus introduce the \emph{stochastic dimension} $e\le d$ of $(F_x)$ as the number of directions, such that for suitable filtrations $\mathcal F_{y_i}^{(i)}$ for $i=1,\dots,e\le d$,

\begin{equation*}
\norm{E^{\mathcal F^{(i)}_{y_i}}(F_x-F_y)(\psi_x^\epsilon)}_{L_p}\lesssim \epsilon^\alpha(\abs{x-y}+\epsilon)^{\tilde \gamma-\alpha}\,.
\end{equation*}

\noindent In this case, we can borrow $\frac 12$ regularity from all of those directions, resulting in $\gamma = \tilde{\gamma}-\frac e2$, an improvement of $\frac e2$ regularity. Thus, we say that $(F_x)$ is stochastically $\gamma$-coherent, if it fulfills certain adaptedness properties (see Definition \ref{stochDimension} for details), and for $i=1,\dots,e\le d$

\begin{align}
\norm{(F_x-F_y)(\psi_x^\epsilon)}_{L_p}\lesssim \epsilon^\alpha(\abs{x-y}+\epsilon)^{\gamma-\frac e2-\alpha} \label{ineq:coherence1Intro}\\
\norm{E^{\mathcal F^{(i)}_{y_i}}(F_x-F_y)(\psi_x^\epsilon)}_{L_p}\lesssim \epsilon^\alpha(\abs{x-y}+\epsilon)^{\gamma-\alpha}.\label{ineq:coherence2Intro}
\end{align}

\noindent Under this condition, we can formulate our main result:

\begin{theorem}[Stochastic reconstruction]\label{theo:StochReconIntro}
Assume the adaptedness condition of Definition \ref{stochDimension} and that $(F_x)_{x\in\R^d}$ is stochastically $\gamma$-coherent for some $\gamma>0$ with stochastic dimension $e\le d$. Then, there exists a unique random distribution $f$, which fulfills

\begin{align*}
\norm{(f-F_x)(\psi_x^\epsilon)}_{L_p}&\lesssim \epsilon^{\gamma-\frac e2}\\
\norm{E^{\mathcal F^{(i)}_{x_i}}(f-F_x)(\psi_x^\epsilon)}_{L_p}&\lesssim \epsilon^{\gamma}
\end{align*}

\noindent for $i=1,\dots,e$.
\end{theorem}

\noindent
We further discuss two more variations of Theorem \ref{theo:StochReconIntro}. In the above version of the reconstruction theorem, we assume $\psi$ to have positive support and $x_i\le y_i$ for $i=1,\dots,e$ in the stochastic coherence property. This is motivated by Itô calculus, where one needs to carefully separate between information ``in the future'' and ``in the past''.  In the deterministic theory however, this is not necessary and the classical reconstruction theorem \cite{caravenna}, \cite{Hairer} do not have similar assumptions. We show that as long as one carefully chooses the point $y_i$ in the conditional expectation $E^{\mathcal F^{(i)}_{y_i}}$, one can get rid of the assumptions $x_i\le y_i$ and that $\psi$ has a positive support in Theorem \ref{two-sided stochastic_reconstruction}.

The second variation we are discussing is reconstruction in the covariance setting: If $p=2$, we can replace the additional condition in the coherence property \eqref{ineq:coherence2Intro} with the following property:

\begin{equation}\label{ineq:covCoherenceIntro}
\abs{E\left(\prod_{i=1}^2 (F_{x_i}-F_{y_i})((\psi_i)_{y_i}^{\epsilon_i})\right)} \lesssim \prod_{i=1}^2\epsilon_i^\alpha(\abs{x_i-y_i}+\epsilon_i)^{\gamma-\alpha}\,,
\end{equation}

\noindent
for points $x_1, y_1, x_2, y_2\in\R^d$ such that $\abs{x_1-x_2}, \abs{y_1-y_2}$ are big enough (see Section \ref{subSectionCov} for details), two test functions $\psi_1,\psi_2$ and two parameters $\epsilon_1,\epsilon_2$. While this condition on the covariance of $(F_x-F_y)(\psi_y^\epsilon)$ looks rather technical, it allows one to entirely get rid of the filtrations $\mathcal F^{(i)}_{y_i}$! Indeed, \eqref{ineq:coherence1Intro} and \eqref{ineq:covCoherenceIntro} together suffice to construct a reconstruction of the germ $F_x$ without any adaptability assumptions. To our knowledge, this is a new idea, even in the one-dimensional setting of stochastic sewing. Let us stress that this is no technical detail: The filtrations we use in this paper need to have an orthogonality condition (also often called commuting condition in the literature, see \cite{multiparameterProcesses}, chapter 11 or \cite{walshMultiparameter}). While this is a standard assumption in the literature of multi-parameter martingales, in practice, it can be hard to verify: For example for anisotropic Gaussian fields $X(s)$ \cite{anisotropicGF}, the covariance $E(X(s)X(t))$ is easy to calculate, but the filtrations generated by $(X(t))_{t\in\R^d}$ are not well understood and hard to work with. We conjecture that for noises in this class, the covariance setting will prove to be a useful tool.

Perhaps an even easier example can be found in Section \ref{sectionWN}: In this section, we use stochastic reconstruction to construct the Walsh integral $\int X(s)\xi(ds)$ against white noise. To use Theorem \ref{theo:StochReconIntro}, we would need to make an unusually strong assumption on the adaptedness on $X$: We would need $X(s)$ to be $\mathcal F_{s_i}^{(i)}$-measurable for all $i=1\dots,d$, while it is well known that adaptedness in time, i.e. $X(s)$ is $\mathcal F^{(1)}_{s_1}$-measurable suffices to construct the integral. The stochastic reconstruction in the covariance setting is to our knowledge the easiest way to construct this integral with reconstruction techniques.

\begin{remark}
It should be noted that the stochastic reconstruction theorem in the covariance setting is not strictly stronger than the classical one: In the covariance setting, we make no distinction whether $(x_1)_i\le (x_2)_i$ or $(x_1)_i\ge (x_2)_i$, i.e. we do not differentiate between the ``past'' and ``future''. For white noise, this does not matter since $\xi(1_{[s,t]})$ is both independent of the past before $s$ and the future after $t$. However, if one only expects an improved regularity for conditioning on the past and no improvement for the future, it is highly likely that Theorem \ref{theo:StochReconIntro} is better suited for that situation.
\end{remark}

\noindent Our last main result reads as follows:

\begin{theorem}
Let $p=2$. Assume that $(F_x)_{x\in\R^d}$ is stochastically $\gamma$-coherent for some $\gamma>0$ in the covariance setting, in the sense that \eqref{ineq:coherence1Intro} and \eqref{ineq:covCoherenceIntro} hold. Then there exists a unique distribution, which fulfills

\begin{align*}
\norm{(f-F_{x_1})(\psi_{x_1}^\lambda)}_{L_2}&\lesssim \lambda^{\gamma-\frac E2} \\
\abs{E\left(\prod_{i=1}^2(f-F_{x_i})(\psi_{x_i}^\lambda)\right)}&\lesssim\lambda^{2\gamma}\,.
\end{align*} 
\end{theorem}

\subsection{The structure of this paper}

This paper is structured as follows: In Section \ref{prelim}, we will formally introduce distribution and scalings. The proofs of our main results make heavy use of wavelet techniques as well as a BDG-type inequality, but since they are not necessary for the presentation of the results themselves, we moved them to the appendix. Both of these parts mainly follow \cite{Hairer} and \cite{meyer}.

Section \ref{mainResult} will introduce the stochastic reconstruction theorem in Theorem \ref{stochastic_reconstruction} as well as the two advertised modifications in Theorem \ref{two-sided stochastic_reconstruction} and Theorem \ref{stochastic_recon_covariance}. The section focuses on presenting the results themselves, while the proofs can be found in Section \ref{sec:proofs}.

In Section \ref{sectionSewing}, we show that in one dimension the stochastic reconstruction theorem becomes the stochastic sewing lemma under weak additional assumptions needed for the distributional framework.

The last two sections of this paper are dedicated to showing simple examples, of how the stochastic reconstruction theorem could be applied. In Section \ref{sectionGMM}, we consider a Gaussian martingale measure $W_t(A)$ and a stochastic process $X(t,x)$, and show how the Walsh type integral \cite{walshOriginal}

\begin{equation*}
\int_0^\infty \int_{\R^d} X(s,x) W(ds,dx)
\end{equation*}

\noindent can be reconstructed as a product between the process $X$ and the distribution $W(ds, dx)$. This closely resembles the construction at the very start of this paper and shows how stochastic integration can be seen as a stochastic version of the Young product.

Since the above integral only assumes martingale properties in one dimension, this case only yields stochastic dimension $e=1$. To give an example with $e>1$, we show a similar reconstruction for integration against white noise in Section \ref{sectionWN} and reconstruct the integral

\begin{equation*}
\int_{\R^d} X(z)\xi(dz),
\end{equation*}

\noindent which can be found in \cite{KPZ}.

\section{Preliminaries}\label{prelim}

\subsection{Distributions and Hölder continuity}

A property that is typical for the theory of SPDEs is that most solutions are too irregular to be viewed as functions and need to be seen as distributions. Let us quickly introduce those: Let $C_c^\infty$ be the space of smooth, compactly supported test functions, and let $C_c^r$ be the space of compactly supported, $r$ times continuously differentiable functions. We equip those spaces with the following $C_c^r$ norms: For any multiindex $k=(k_1,\dots,k_n)$ and $k_1+\dots+k_n$ times continuously differentiable function $\psi$, let $\psi^{(k)} := \frac{\partial^{k_1+\dots+k_n}}{\partial x_1^{k_1}\dots\partial x_d^{k_d}}\psi$ be a partial derivative of $\psi$. Then, for some $r\in\mathbb N$, we set

\begin{equation*}
\norm{\psi}_{C_c^r} = \sum_{k_1+\dots+k_d\le r} \norm{\psi^{(k)}}_\infty.
\end{equation*}

\noindent
A distribution on this space is defined as follows:

\begin{definition}

For any $r\in\mathbb N$, a \emph{distribution} (of order at most $r$) is a linear map $f:C_c^r \to\R$, such that for any compact set $K$ there is a constant $C(K)$, such that for any $\psi\in C_c^r$ with support in $K$, it holds that

\begin{equation*}
\abs{f(\psi)}\le C(K)\norm{\psi}_{C_c^r}.
\end{equation*}

\noindent Further, for any $p<\infty$, we call a linear map $f:C_c^r\to L_p(\Omega)$ a \emph{random distribution}, if it fulfills 

\begin{equation*}
\norm{f(\psi)}_{L_p}\le C(K)\norm{\psi}_{C_c^r}.
\end{equation*}
\end{definition}

\noindent We use the notation $\abs{f(\psi)}\lesssim \norm{\psi}_{C_c^r}$, where we allow constants hidden in $\lesssim$ to depend on underlying compact sets.

\noindent
We need to extend the notion of Hölder continuity to distributions. This requires us to measure the local behavior of distributions, which can be achieved with the help of localizations of test functions. To this end, we call the family of maps $(S^\lambda)_{\lambda\ge 0}$ a \emph{scaling} associated to the vector $s\in\mathbb R^d$ with $s_i>0$ for all $i=1,\dots,d$, if it is given by

\begin{align*}
S^\lambda :\mathbb R^d&\to\R^d \\
(x_1,\dots,x_d)&\mapsto (\lambda^{s_1} x_1,\dots, \lambda^{s_d} x_d).
\end{align*}

\noindent The rank of $S$ is given by $\abs{S} = \abs{s} := s_1+\dots+s_d$. If $\R^d$ has a scaling $S$ assigned to it, we also equip the space with the seminorm

\begin{equation*}
\abs{x} = \abs{x_1}^{\frac 1{s_1}}+\dots+\abs{x_d}^{\frac 1{s_d}},
\end{equation*}

\noindent which is often referred to as a homogeneous norm, although it is  not a norm in the strict sense. Note that it has the property, that $\abs{S^\lambda x} = \lambda\abs{x}$.

To give two short examples, the canonical scaling $x\mapsto \lambda x$ is simply given by the map $S$ associated to $s=(1,\dots,1)$, while the parabolic scaling is associated to $s=(2,1,\dots,1)$.

With this in mind, we can now take a look at the localization of a test function: Let $\psi\in C_c^r$ be a test function, $x\in\R^d$ and $\lambda\in(0,1)$. Then, the \emph{localized} test function $\psi_x^\lambda$ (under the scaling $S$) is given by

\begin{equation*}
\psi_x^\lambda(y) := \frac 1{\lambda^{\abs{S}}} \psi\left(S^{\frac 1\lambda}(y-x)\right).
\end{equation*}

\noindent In the case where $S$ is the canonical scaling associated to $(1,\dots,1)$, this is simply given by the classical notion of localized test functions:

\begin{equation*}
\psi_x^\lambda(y) :=\frac 1{\lambda^d}\psi\left(\frac{y-x}{\lambda}\right).
\end{equation*}

\noindent The localization has the following properties:

\begin{itemize}
\item Assume $\supp(\psi)\subset[-C,C]^d$. Then the support of $\psi^\lambda_x$ lies in $[x-\lambda^{s_1} C,x+\lambda^{s_1} C]\times\dots\times [x-\lambda^{s_d} C,x+\lambda^{s_d} C]$, and
\item it holds that $\int_{\R^d}\abs{\psi(y)}dy = \int_{\R^d}\abs{\psi_x^\lambda(y)} dy$.
\end{itemize}

\noindent These properties justify thinking of $\psi_x^\lambda$ as the density of a measure on $\R^d$, which is highly concentrated around $x$ for small $\lambda$. Indeed, as $\lambda\to 0$, this converges to the Dirac measure $\delta_x$: For all continuous functions $f:\R^d\to \R$

\begin{equation}\label{dirac}
\scalar{f,\psi_x^\lambda}\xrightarrow{\lambda\to 0} f(x),
\end{equation}

\noindent where $\scalar{f,g} = \int_{\R^d} f(x) g(x) dx$ is simply the $L_2(\R^d)$ scalar product.

For a true distribution $g$ (i.e. there is no function $\tilde g$, such that $g(\psi) = \scalar{\tilde g,\psi}$ for all $\psi\in C_c^r$), $g(\psi_x^\lambda)$ is bound to diverge as $\lambda\to 0$. (If it were to converge, we could set $\tilde g(x) = g(\delta_x) = \lim_{\lambda\to 0}g(\psi_x^\lambda)$) The speed of divergence is going to give us a way to define the regularity of $g$, which will help us to define Hölder spaces of negative regularity. 

Let us first introduce Hölder continuity for positive regularities: For an $\alpha\in (0,1]$, and a compact set $K\subset \R^d$, let

\begin{equation*}
C^{\alpha}(K) := \{f\in C^0(\R^d)~|~\exists C>0\forall x,y\in K: \abs{f(x)-f(y)}\le C\abs{x-y}^\alpha\},
\end{equation*}

\noindent and we call $C^\alpha := \bigcap_{K\subset \R^d} C^{\alpha}(K)$ the set of (locally) $\alpha$-Hölder continuous functions, where the intersection is taken over all compact sets $K$ in $\R^d$. We also refer to these spaces simply as the space of Hölder-continuous functions (and drop the ``locally''), as every function and distribution in this paper will be only locally Hölder continuous. In consistence with the $\lesssim$ notation, we write $\abs{f(x)-f(y)}\lesssim\abs{x-y}^\alpha$ for $f\in C^\alpha$.

For random processes, we define the space of $\alpha$-Hölder continuous processes by

\begin{equation*}
C^\alpha(L_p(\Omega),K) := \{X:K\to L_P(\Omega)~\vert~ \norm{X(x)-X(y)}_{L_p}\lesssim \abs{x-y}^\alpha\}\,,
\end{equation*}

\noindent where $C^\alpha(L_p)$ is defined similarly to before.

As mentioned above, for $\alpha<0$, we measure the Hölder regularity of a distribution $f$ by measuring the speed of divergence of $f(\psi_x^\lambda)$ for $\lambda\to 0$. More specifically, we define the space $C^\alpha$ for negative $\alpha$ as follows:

\begin{definition}

Let $\alpha<0$, and $r\in\mathbb N$ with $r-1 \le \abs\alpha < r$. A distribution $f:C_c^r\rightarrow\R$ is in the space $C^\alpha$, if for any compact set $K\subset\R^d$, there is some constant $C(K)$, such that for every distribution $\psi$ with $\norm{\psi}_{C_c^r}=1$ and compact support in $[-1,1]^d$ and for all $x\in K$ the following holds:

\begin{equation}\label{ineq:Calpha}
\abs{f(\psi_x^\lambda)}\le C(K) \lambda^\alpha.
\end{equation}

\noindent For a random distribution, we say that $f\in C^\alpha(L_p(\Omega))$, if it holds that

\begin{equation}\label{ineq:CalphaRand}
\norm{f(\psi_x^\lambda)}_{L_p}\le C(K) \lambda^\alpha.
\end{equation}
\end{definition}

\begin{remark}\label{rem:allDistAreCalpha}
Note that any random $C_c^r$-distribution is in $C^{-\abs S-r}(L_p)$, since $\norm{\psi_x^\lambda}_{C_c^r}\lesssim \lambda^{-\abs S-r}\norm{\psi}_{C_c^r}$.
\end{remark}

\begin{remark}
In general, for a random distribution $f$ and any $\omega\in\Omega$, $\psi\mapsto f(\psi)(\omega)$ is not a distribution. However, it is possible to show a version of Kolmogorov's continuity theorem for negative Hölder spaces: Given an $f\in C^{\alpha}(L_p)$, it holds that for almost all $\omega\in\Omega$, $f(\cdot,\omega)\in C^{\hat\alpha}(L_p)$ for a small enough $\hat\alpha<\alpha$. By the previous remark, it follows that any random distribution $f$ is a.s. a deterministic distribution on some space $C_c^{\hat r}$ for a large enough $\hat r$. A detailed statement of this can be found in the appendix, Lemma \ref{lem:Kolmogorov}.
\end{remark}

\section{Main results}\label{mainResult}

In this section, we are going to present our main result, the stochastic reconstruction theorem, as well as two variations of it. Since the proofs are rather technical, we postpone them to a separate section.

\subsection{Stochastic Reconstruction}

Let us start by rigorously defining the notion of random germ:

\begin{definition}\label{def:randomGerm}
We call $(F_z)_{z\in\R^d}$ a \emph{random germ}, if for all $z\in\R^d$, $F_z:C_c^r\to L_p(\Omega)$ is a continuous, linear map for some fixed number $1\le p <\infty$.
\end{definition}

\noindent Recall that Martin Hairer showed in \cite{Hairer} that the reconstructing sequence

\begin{equation}\label{sequence1}
f_n(\psi) := \sum_{y\in\Delta_n}F_y(\phi_y^n)\scalar{\phi_y^n,\psi}
\end{equation}

\noindent for a wavelet function $\phi$ (see the appendix for details) converges to a limiting distribution, whenever the germ $(F_z)_{z\in\mathbb R^d}$ fulfills a property called $\gamma$-coherence for a $\gamma > 0$. (Note that Hairer did not namely mention the coherence property. It was introduced by Caravenna and Zambotti in \cite{caravenna}. However, Hairer's proof works for \eqref{sequence1} under the coherence assumption of Caravenna-Zambotti.)

\noindent If the germ $(F_x)_{x\in\R^d}$ is a random germ, it is very reasonable to assume random cancellations in the above sum and thus convergence for certain $\gamma<0$, if the germ has some form of martingale properties. As a motivation, we want to start with a short example: Consider space-time white noise $\xi(t,x)$ (see Section \ref{sectionWN} for a formal introduction) in $1+1$ dimension and an $\alpha$-Hölder continuous deterministic function $X\in C^\alpha(\R^2)$ for some $\alpha>0$. We show in Section \ref{sectionWN} that, for the germ $(F_{(t,x)})_{(t,x)\in\R^2}$ defined by

\begin{equation}\label{whiteNoise}
F_{(t,x)}(\psi) := X(t,x)\xi(\psi) = X(t,x) \int\int \psi(s,y)\xi(s,y) dsdy,
\end{equation}

\noindent the reconstructing sequence $f_n(\psi)$ converges in $L_p(\Omega)$ for $2\le p<\infty$ to

\begin{equation*}
f_n(\psi) \xrightarrow{n\to\infty} \int\int X(s,y)\psi(s,y)\xi(s,y) dsdy,
\end{equation*}

\noindent where the right-hand side is a Walsh-type integral against white noise, see \cite{walshOriginal} for reference. It is also not hard to see, that $F_{(t,x)}$ is only $(-1+\alpha)^-$-coherent (which reflects $\xi$ having regularity $(-1)^-$, if we use the canonical scaling $s = (1,1)$, and $X$ having regularity $\alpha$). So this example shows a reconstruction with a $\gamma>-1$, an improvement of $+1$ over classic reconstruction. If one compares this to the stochastic sewing lemma in which one usually gets an improvement of $+\frac 12$, it seems as if one can apply the sewing lemma twice in the above construction. This makes a lot of sense, as the white noise has martingale properties in both space and time directions. 

However, since we also wish to cover more general noises, which might not have these properties in all directions (e.g. white noise in time, colored noise in space), we will split our dimensions $1,\dots,d$ into two types of directions: Those, in which we have certain stochastic properties, and those, in which we lack such properties. We call the number of directions, which have such properties the \emph{stochastic dimension} of the germ.

Before we formally introduce this concept, let us make two quick remarks:

\begin{remark}
The above construction also works for a stochastic process $X(t,x)$, which is adapted to the sigma-algebra $\mathcal F_t$ defined in \eqref{filtrationExample}, i.e. $X(t,x)$ is $\mathcal F_t$ measurable for all $(x,t)\in\R^2$, which is $\alpha$-Hölder continuous in the following sense:

\begin{equation*}
\norm{X(t,x)-X(s,y)}_{L_p}\lesssim \abs{(t
,x)-(s,y)}^\alpha.
\end{equation*}
\end{remark}

\begin{remark}

It should be noted that for the above germ, Caravenna-Zambotti's reconstructing sequence given by

\begin{equation}\label{eq:ZorinKranichSequence}
f_n(\psi) = \int_{\R^d} F_z(\rho_z^{\epsilon_n})\psi(z)dz
\end{equation}

\noindent converges to the same limit for deterministic functions $X$. To see this, note that dominated convergence and the Itô-isometry (\cite{KPZ}, equation (121)) imply $f_n(\psi) = \int_{\R^d} (\rho^{\epsilon_n}(-\cdot)*(X\psi))(y)\xi(y) dy$. One then shows $\rho^{\epsilon_n}(-\cdot)*(X\psi)\xrightarrow{\epsilon_n\to 0}X\psi$ in $L_2(\R^d)$. Using again the Itô-isometry, we get that $f_n(\psi)$ converges to $\int_{\R^d} X(y)\psi(y)\xi(y) dy$ in $L_2$. 

Because of this, we suspect that the main result of this paper can also be achieved using their mollification techniques.
\end{remark}

Let us fix the stochastic setting for reconstruction: Using our example above as a motivation, we see that $(F_{(t,x)})_{(t,x)\in\R^2}$ is a random germ over the probability space $(\Omega,\mathcal F,P)$, where $\mathcal F = \sigma(\xi(\psi)|\psi\in L_2(\R^2))$ is the sigma-algebra generated by white noise. This sigma-algebra comes with a natural filtration, which one can write informally as ``$\mathcal F_t = \{\xi(s,x)|s\le t\}$'', but is more formally given by

\begin{equation}\label{filtrationExample}
\mathcal F_t = \sigma(\xi(\psi)\vert \supp(\psi)\subset (-\infty,t]\times\R).
\end{equation}

\noindent Let now $t\in\mathbb R$ and $\psi$ with a support in $(-\infty,s]\times\R$ be fixed. It then holds that $\xi(\psi)$ is $\mathcal F_s$-measurable and we assume that $X(\cdot,x)$ is adapted to $(\mathcal F_t)_{t\in\R}$, i.e. $X(t,x)$ is $\mathcal F_t$-measurable for all $x\in\R$. It immediately follows, that $F_{t,x}(\psi) = X(t,x)\xi(\psi)$ is $\mathcal F_{\max(s,t)}$-measurable.

Of course, a similar construction can be made in the spatial direction $x$ with respect to the filtration $\mathcal F_x^{\text{spatial}} := \sigma(\xi(\psi)\vert \supp(\psi)\subset\R\times(-\infty,x])$. This gives us two different filtrations, and our germ is adapted to both in the sense that for $t,x\in\R$ and $\psi$ with support in $(-\infty,s]\times (-\infty,y]$, $F_{t,x}(\psi)$ is $\mathcal F_{\max(s,t)}$ and $\mathcal F_{\max(x,y)}^{\text{spatial}}$-measurable, as long as $X$ is adapted to both filtrations.

We want to generalize this adaptedness property: Let $(\Omega,\mathcal F,P)$ be a complete probability space. As in the example above, consider $e$ many filtrations $(\mathcal F_t^{(i)})_{t\in\R}$, $i=1,\dots,e$. We make the following assumptions about these filtrations:

\begin{itemize}
\item {\bf Completeness:} We assume, that for each $i=1,\dots, e$, $t\in\R$, $\mathcal F^{(i)}_t$ is complete.

\item {\bf Right-continuity:} We assume, that for $i=1,\dots, e$, $(\mathcal F_t^{(i)})_{t\in\R}$ is a right continuous filtration.

\item {\bf Orthogonality:} We assume that conditioning in one direction does not change measurability along the other directions. More precisely, let $X$ be a $\mathcal F_{z_i}^{(i)}$-measurable random variable for some $1\le i\le e$ and $z_i\in\R$. We assume that $E^{\mathcal F_{z_j}^{(j)}}X$ is still $\mathcal F_{z_i}^{(i)}$-measurable for all $1\le j\le e, j\neq i$ and all $z_j\in\R$.
\end{itemize}

\noindent Let us give a short intuition, of what we need each of these properties for. The reconstructing sequence \eqref{sequence1} will be a limit in $L_p$, so completeness is necessary to ensure that the limit, as well as modifications of the limit, are still $\mathcal F_t^{(i)}$ measurable for all $i$ and large enough $t$. A closer look at the reconstructing sequence reveals that $f_n(\psi)$ is only $\mathcal F^{(i)}_{y_i+\epsilon_n}$-measurable for $\psi$ with support in $(-\infty,y_1]\times\dots\times(-\infty,y_d]$ and a certain sequence $\epsilon_n\xrightarrow{n\to\infty} 0$. Right continuity then assures that the limit is $\mathcal F^{(i)}_{y_i}$-measurable.

Orthogonality is used more subtly: The proof requires us to use the BDG-type inequality \eqref{BDG} inductively over all stochastic directions. However, without any assumption on how the filtrations $(\mathcal F_{x_i}^{(i)})_{x_i}$, $i=1,\dots,e$ interact with each other, $E^{\mathcal F_{x_i}^{(i)}}F_{x}(\psi)$ will not have any adaptedness property with regards to the filtrations $(\mathcal F_{x_j}^{(j)})_{x_j}$ for $i\neq j$. So orthogonality ensures that we can use \eqref{BDG} several times with respect to different filtrations.

It should also be noted that if $e=1$, orthogonality is automatically fulfilled as one only has one filtration.

\begin{remark}
Another way of formulating orthogonality is requiring that conditioning on $\mathcal F^{(i)}$ and $\mathcal F^{(j)}$ commutes, i.e. for all $1\le i,j\le e$ with $i\neq j$, $z_i,z_j\in \R$ and for a random variable $X$, it holds that $E^{\mathcal F_{s_j}^{(j)}} E^{\mathcal F_{s_i}^{(i)}} X = E^{\mathcal F_{s_i}^{(i)}}E^{\mathcal F_{s_j}^{(j)}} X$.
\end{remark}

\begin{remark}\label{rem:condIndependence}
This setting can be seen as a generalization of the setting used in  \cite{ImkellerNParameterMart}, \cite{walshMultiparameter} to construct multiparameter martingales: The authors start from a multiparameter filtration $(\mathcal F_x)_{x\in\R^d}$ and constructs \[\mathcal F_y^{(i)} := \sigma\left(\bigcup_{x_1,\dots,x_{i-1},x_{i+1},\dots,x_d\in\R} \mathcal F_{(x_1,\dots,x_{i-1},y,x_{i+1},\dots,x_d)}\right)\,.\] In this case, orthogonality is equivalent to conditional independence of $\mathcal F_y^{(i)}$, $i=1\dots,d$, see hypothesis (F4) of \cite{walshMultiparameter} on page 348. We prefer to directly work with the filtrations $(\mathcal F_x^{(i)})$, as our proofs do not use any additional structure of the multiparameter setting.
\end{remark}

\begin{remark}\label{rem:notFree}
We want to mention that the condition of orthogonality does not come for free: While it is a standard assumption in multiparameter martingale theory, e.g. the noises constructed in \cite{anisotropicGF} do not generate such filtrations. For a rule of thumb, Brownian sheet and deterministic kernel integrated over Brownian sheet $B^H_t := \int_{(-\infty,t]^d} K(t,s) dB_s$ like fractional Brownian sheet (see \cite{fBrownianSheet} for the definition of $B^H$ in dimension $d>1$ and \cite{representationFormulasFBS} for the kernel $K$) will be adapted to orthogonal filtrations, while more general Gaussian noises like the spectral density construction of \cite{anisotropicGF} will not generate such filtrations. Theorem \ref{stochastic_recon_covariance} is better suited to deal with those cases.
\end{remark}

\noindent We denote by $\pi_i$ the projection onto the $i$-th coordinate. With this notation, we can now define the term stochastic dimension: 

\begin{definition}\label{stochDimension}
Let $(F_x)_{x\in\R^d}$ be a random germ in $d$ dimensions, and let $e\le d$. Let $ (\mathcal F^{(i)}_z)_{z\in\R}, i=1,\dots,e$ be filtrations of some $\sigma$-Algebra $\mathcal F$, which fulfill the above three properties of completeness, right-continuity, and orthogonality.

\noindent
We say that $(F_x)$ is of \emph{stochastic dimension} $e$, if for $i=1,\dots,e$:

\noindent
For any test function $\psi$ with support $\supp(\psi)\subset \R^{i-1}\times(-\infty,y]\times\R^{d-i}$, it holds that 

\begin{equation}\label{measurability}
F_x(\psi) \text{ is } \mathcal F^{(i)}_{\max(\pi_i(x),y)}\text{-measurable}.
\end{equation}

\noindent We further say that a random distribution $f$ has stochastic dimension $e$, if for any test function $\psi$ with support $\supp(\psi)\subset \R^{i-1}\times(-\infty,y]\times\R^{d-i}$

\begin{equation*}
f(\psi)\text{ is } \mathcal F^{(i)}_{y}\text{-measurable}.
\end{equation*}
\end{definition}

\begin{remark}
As most readers have undoubtedly noticed, we use a two-directional setting in the sense that we always use $t\in\R$ instead of $t\in[0,\infty)$. In one dimension, our example \eqref{whiteNoise} becomes a two-sided Brownian motion, not a classical one.

Since this whole theory highly depends on recentering and scaling operations, the origin point does not have any special role, which makes this two-directional setting more natural. However, one can use the locality of this theory to achieve a one-sided setting, by restricting oneself to compact sets $K \subset [0,\infty)\times\R^{d-1}$ in Definition \ref{defCoherence} and Theorem \ref{stochastic_reconstruction}. %it is also a local theory: Our main statement only holds over all compact sets $K\subset\R^d$. So by making the restriction $K \subset [0,\infty)\times\R^{d-1}$, one can easily get the same result for a one-sided setting.
\end{remark}

\noindent Since we work with a scaling $S$ associated to some $s\in\R^d_+$, we do not consider $e$ to be the number of stochastic dimensions, but rather consider

\begin{equation*}
E := \sum_{i=1}^e s_i\le \abs S.
\end{equation*}

\noindent With the definition of the stochastic dimension in mind, we can now formulate a stochastic version of the coherence property. Since we aim to get $L_p$-convergence, let us fix now some $2\le p <\infty$ and define stochastic coherence as follows:

\begin{definition}\label{defCoherence}
Let $(F_x)_{x\in\R^d}$ be a random germ with stochastic dimension $e\le d$ and let $E$ be as above. We call $(F_x)$ \emph{stochastically $\gamma$-coherent}, if there is an $\alpha \le \gamma$, such that the following holds for all test functions $\psi\in C_c^r(\R^d)$ with $\norm{\psi}_{C_c^r} = 1$ and $\supp(\psi)\subset [0,\tilde R]^e\times [-\tilde R,\tilde R]^{d-e}$ for some $\tilde R>0$, for the stochastic directions $i=1,\dots,e$, for any compact set $K\subset\R^d$ and for any $x,y\in K$ with $\pi_i(x)\le \pi_i(y)$:

\begin{align}
\norm{E^{\mathcal F^{(i)}_{\pi_i(x)}}(F_x-F_y)(\psi_y^\epsilon)}_{L_p} &\lesssim \epsilon^{\alpha}(\abs{x-y}+\epsilon)^{\gamma-\alpha} \label{coherence2}\\
\norm{(F_x-F_y)(\psi_y^\epsilon)}_{L_p} &\lesssim \epsilon^{\alpha}(\abs{x-y}+\epsilon)^{\gamma-\frac E2-\alpha},\label{coherence}
\end{align}

\noindent where the constant appearing in $\lesssim$ is allowed to depend on the compact set $K$ and the radius $\tilde R$.

\end{definition}

\noindent We sometimes write $(\gamma,\alpha)-coherent$, if we need the explicit value of $\alpha$. Let us stress that the constant appearing in $\lesssim$ (both in the above definition and the results throughout this paper) should not depend on $\psi$, as long as $\tilde R$ and $\norm{\psi}_{C_c^r}$ are fixed.

\begin{remark}
In the usual construction of stochastic integrals, it is far more natural to use left-side approximations than two-sided ones. Thus, \eqref{whiteNoise} is only a useful approximation, if the support of $\psi$ is in $[t,\infty)\times[x,\infty)$. We therefore only consider test functions with positive support in the stochastic directions in the definition above and Theorem \ref{stochastic_reconstruction}. It is however possible to get a two-sided stochastic reconstruction, which will be discussed in Section \ref{twosided}. 
\end{remark}

\begin{remark}
It is possible to allow $\alpha = \alpha(K)$ to depend on the compact set $K$. In this case, one simply has to replace $\alpha$ with $\alpha(\tilde K)$ in the proofs, for some fattening $\tilde K$ of $K$ which depends on the support of the wavelet.
\end{remark}

\noindent Note that each stochastic dimension reduces the coherence required by $(F_x)$ in \eqref{coherence} by $\frac 12$, resulting in a total reduction of $\frac E2$, under the corresponding scaling. In particular, if the scaling is just the canonical scaling $s=(1,\dots,1)$, we get a reduction by $\frac e2$. Further note that the support of $\psi_y^\epsilon$ lies completely on the right-hand side of the time we condition on, $\pi_i(x)$, with respect to the direction $i$.

\noindent Recall that a random distribution $\tilde f$ is a modification of $f$, if for all test functions $\psi\in C_c^r$, $f(\psi) = \tilde f(\psi)$ almost surely. With this in mind, we can formulate our main result:

\begin{theorem}[Stochastic Reconstruction]\label{stochastic_reconstruction}
Let $(F_x)_{x\in\R^d}$ be a random germ for some $2\le p<\infty$ with stochastic dimension $e\le d$, which is stochastically $(\gamma,\alpha)$-coherent for some $\gamma > 0$ and $\alpha<0$. Set $\tilde r := r\cdot\min_{i=1,\dots,d} s_i$. Assume that our parameters fulfill $\alpha+\tilde r > 0$ and $\gamma-\frac E2+\tilde r > 0$.

Then, there is a unique random distribution $f$ (up to modifications) with stochastic dimension $e$ with regard to the same filtrations $\mathcal F^{(i)}$ as our germ, such that the following holds for $i=1,\dots,e$, any test function $\psi\in C_c^r$ with $\norm{\psi}_{C_c^r}=1$ and whose support is in $\supp(\psi)\subset[\tilde C,\tilde R]^e\times[-\tilde R,\tilde R]^{d-e}$ for some $0\le\tilde C\le \tilde R$ and $\lambda\in(0,1]$:

\begin{align}
\norm{(f-F_x)(\psi_x^\lambda)}_{L_p}&\lesssim \lambda^{\gamma-\frac E2} \label{unique1}\\
\norm{E^{\mathcal F^{(i)}_{\pi_i(x)}}(f-F_x)(\psi_x^\lambda)}_{L_p}&\lesssim \lambda^\gamma,\label{unique2}
\end{align} 

\noindent where the constant in $\lesssim$ again depends on the compact set $K\subset \R^d$ with $x\in K$ and the constants $\tilde C, \tilde R$.
\end{theorem}

\noindent Note that the support of $\psi$ does have to be strictly positive. This is necessary to ``fit in'' a wavelet between the zero and the support of $\psi$. If we set $\tilde C = 2R$ (recall that the support of our wavelet is in $[C,R]^d$), the constant in $\lesssim$ generated by $\tilde C$ disappears.

\begin{remark}
The restrictions to $\psi$ might seem strange to the reader. Indeed, it seems that this theorem only makes a statement about test functions with (strictly) positive supports in the stochastic directions. This is not true since the localizations $\psi_y^\epsilon$ can have negative supports in all directions depending on $y$. The support does only serve to separate the point of localization of $\psi$ and the index of the germ $F_y$. Another way to formulate the result is to use a $\psi$ with any compact support, and take a look at $F_{(y-S^\lambda C)}(\psi_y^\lambda)$ for a certain vector $C$ depending on the support of $\psi$ and which directions are supposed to be stochastic.
\end{remark}

\begin{remark}
If one is given a germ $(F_x)_{x\in\R^d}$ such that $\alpha+\tilde r$ or $\gamma-\frac E2+\tilde r$ is smaller than $0$, one can always restrict $F_x$ to a space $C_c^{\hat r}$ for a large enough $\hat r$ to fulfill the conditions on the parameters.
\end{remark}

\noindent
Before we prove this theorem, let us discuss two modifications of it:

\subsection{Two-sided Reconstruction}\label{twosided}

The theorem we presented above is modeled after left-sided approximations of stochastic integrals. However, it is also possible (and closer to the original reconstruction Theorem) to have a two-sided reconstruction, i.e. to use test functions $\psi$ with support in $[-\tilde R,\tilde R]^d$ for some $\tilde R>0$. In this setting, it no longer makes sense to condition $F_x(\psi_x^\lambda)$ onto $\mathcal F^{(i)}_{\pi_i(x)}$, since the conditioning would not have any effect on test functions with purely negative support. Instead, we are going to condition it on the left boundary of the support of $\psi_x^\lambda$, i.e. onto $\mathcal F^{(i)}_{\pi_i(x)-\lambda^{s_i}\tilde R}$.

\begin{definition}[two-sided stochastic coherence]

Let $(F_x)_{x\in\R^d}$ be a stochastic germ with stochastic dimension $e\le d$. We call $(F_x)$ \emph{two-sided stochastically $\gamma$-coherent}, if there is an $\alpha<\gamma$, such that the following holds for all $\psi\in C_c^r$ with $\norm{\psi}_{C_c^r}=1$ and support in $[-\tilde R,\tilde R]^d$ for some $\tilde R>0$, for the stochastic directions $i=1,\dots,e$, for any compact set $K\subset \R^d$ and for any $x,y\in K$:

%Let $\psi\in C_c^r(\R^d)$ be a test function with compact support on $\R^d$. Assume that $\supp(\psi)\subset [-\tilde R,\tilde R]^d$. Let further $(F_x)_{x\in\R^d}$ be a stochastic germ with stochastic dimension $e\le d$. Then, we call $(F_x)$ \emph{two-sided stochastically $\gamma$-coherent}, if there is an $\alpha\in\mathbb R$, such that the following holds for the stochastic directions $i=1,\dots,e$:

\begin{align*}
\norm{E^{\mathcal F^{(i)}_{\pi_i(y)-\lambda^{s_i} \tilde R}}(F_x-F_y)(\psi_y^\lambda)}_{L_p} &\lesssim \lambda^{\alpha}(\abs{x-y}+\lambda)^{\gamma-\alpha}\\
\norm{(F_x-F_y)(\psi_y^\lambda)}_{L_p} &\lesssim \lambda^{\alpha}(\abs{x-y}+\lambda)^{\gamma-\frac E2-\alpha},
\end{align*}

\noindent where the constant in $\lesssim$ is allowed to depend on the compact set $K$ and $\tilde R$.
\end{definition}

\noindent Note that this definition also does not have the condition, that $\pi_i(x)\le \pi_i(y)$, but rather uses arbitrary $x,y$. We want to add that there is no cheap way to get from the one-sided coherence property to the two-sided, or vice versa: If one replaces $\psi$ with a shifted function $\psi_{\vec r}$ to get a test function with a purely positive support, and substitute $y$ with $\tilde y = y-S^\lambda\vec r$, such that $\psi_y^\lambda = (\psi_{\vec r})_{\tilde y}^\lambda$, one gets $F_{\tilde y + S^\lambda \vec r}$ instead of $F_{\tilde y}$. This highlights the main difference between the two versions: In the one-sided reconstruction, $F_y$ only gets tested against the future, while in the two-sided version, $F_y$ sits ``in the middle'' of the support of the test function. This makes the two-sided version closely connected to Stratonovich integration. Since we are interested in the martingale properties, the one-sided version (which is closely connected to Itô integration) is more natural.

Using this property, one can show the following theorem:

\begin{theorem}[two sided stochastic reconstruction]\label{two-sided stochastic_reconstruction}
Let $(F_x)_{x\in\R^d}$ be a random germ for some $2\le p <\infty$ with stochastic dimension $e\le d$, which is two-sided stochastically $(\gamma,\alpha)$-coherent for some $\gamma > 0$, $\alpha<0$. Let $\tilde r := r\cdot\min_{i=1,\dots, d}(s_i)$ and assume $\alpha+\tilde r > 0$ and $\gamma-\frac E2+\tilde r > 0$.

Then, where is a unique (up to modifications) random distribution $f$ with stochastic dimension $e$ with respect to the same filtrations as the germ $(F_x)$, such that the following holds for $i=1,\dots,e$ and any test function $\psi$ with support in $[-\tilde R,\tilde R]^d$ for some $\tilde R>0$ and $\norm{\psi}_{C_c^r}=1$ and $\lambda\in(0,1]$:

\begin{align*}
\norm{(f-F_x)(\psi_x^\lambda)}_{L_2}&\lesssim \lambda^{\gamma-\frac E2} \\
\norm{E^{\mathcal F^{(i)}_{\check S_i(x)}}(f-F_x)(\psi_x^\lambda)}_{L_2}&\lesssim \lambda^\gamma,
\end{align*} 

\noindent where $\check S_i(x) := \pi_i(x)-\lambda^{s_i}(\tilde R+\tilde C)$ for some $\tilde C>0$, and the constant in $\lesssim$ is allowed to depend on the compact set $K\subset\R^d$ with $x\in K$ and on $\tilde C, \tilde R$.
\end{theorem}

\begin{proof}
The proof is essentially the same as for the one-sided version. However, one can use the more usual wavelets with supports fulfilling $\supp(\phi),\supp(\hat\phi)\subset [- R, R]^d$.
\end{proof}

\begin{remark}
Note that this theorem has a similar fattening happening in the uniqueness-property: We do not condition on the left boundary $\pi_i(x)-\lambda^{s_i}\tilde R$, but we leave a ``gap'', in which enough wavelets fit. If one sets $\tilde C=4R$, then $\tilde C$ does not contribute a constant in the above inequalities.
\end{remark}

\subsection{The covariance setting in $L_2$}\label{subSectionCov}

Theorem \ref{stochastic_reconstruction} contains an assumption, which might be stronger than one would expect at first: Namely that $F_x(\psi)$ is $\mathcal F_{\max(\pi_i(x),y)}^{(i)}$ measurable for $i=1,\dots,d$ for test functions $\psi$ with the respective support, for a set of filtrations which have the orthogonality property. As mentioned in Remark \ref{rem:notFree}, this adaptedness condition does not come for free. Furthermore, one can find several interesting examples, in which stochastic integrals have been constructed, even if one does violate this property:

\begin{itemize}
\item In the example \eqref{whiteNoise}, we have to assume that $X$ is adapted to all $\mathcal F^{(i)}_x$ in the above mentioned way. However, it is well known that the construction of the Walsh integral should work for all processes $X$, which are adapted to the time-filtration $\mathcal F^{(1)}_t$ \cite{walshOriginal}. A more throughout discussion of this can be found in Section \ref{sectionWN}.

\item More generally, Malliavin calculus \cite{nualart} successfully constructs the Skorohod integral over white noise, even for non-adapted processes $X$.

\item The technique used in \cite{regByNoise} to analyze regularization by noise only uses local non-determinism properties of the underlying noise for their construction, which only need the conditional covariances $\mathrm{cov}(X(t)~\vert~ X(s_1),\dots, X(s_n))$ to be known for a Gaussian noise $X$. These techniques should be possible for certain noises defined through their spectral densities \cite{anisotropicGF}, even if they do not generate orthogonal filtrations.
\end{itemize}

\noindent All of the above examples can circumvent the use of conditional expectation by using the covariances of the underlying Gaussian processes in some way. This motivated us to consider the case, in which the germ $F_x(\psi)$ is not adapted to orthogonal filtrations, but we instead know its covariance $E(F_x(\psi_1)F_y(\psi_2))$.

Let us take a look at how the conditional expectation is used in the proof of Theorem \ref{stochastic_reconstruction}: We use it to establish an inequality of the form $\norm{\sum_{x\in\Delta_n} g_x}_{L_p}\lesssim \left(\sum_{x\in\Delta_n}\norm{g_x}^2_{L_p}\right)^{\frac 12}$ plus some remainder (see estimate \eqref{iteration}) for a technical random term $g_x$ defined in \eqref{def:gx}. The remainder gets controlled by the additional property \eqref{coherence2} of the stochastic coherence property, which allows us to get a better bound for $\norm{\sum_{x\in\Delta_n} g_x}_{L_p}$ than simply pulling the norm into the sum.

If $p=2$, there is no need for such a bound, as one can explicitly calculate $\norm{\sum_{x\in\Delta_n} g_x}_{L_2}^2 = \sum_{x\in\Delta_n}\norm{g_x}_{L_2}^2 + \sum_{x\neq y} E(g_x g_y)$. So what we need is an upper bound on $\sum_{x\neq y} E(g_x g_y)$ to replace the stochastic machinery used in Theorem \ref{stochastic_reconstruction}. So we need an additional property that allows us to bound $E(g_xg_y)$ for $x\neq y$.

In practice, this should be a bound on the term $\Pi_{i=1}^2(F_{x_i}-F_{y_i})((\psi_i)_{y_i}^{\epsilon_i})$ for $(x_1,y_1)$ and $(x_2,y_2)$ ``far enough'' from each other, and two test functions $\psi_1,\psi_2$ in $C_c^r$. To quantify how far apart $(x_1,y_1)$ and $(x_2,y_2)$ need to be, we introduce the following concept:

\begin{definition}
We call the following intervals effective supports of the corresponding expressions:

\begin{itemize}
\item For $(F_x-F_y)(\psi)$, the effective support is the smallest possible intervall $[a,b] := [\pi_1(a),\pi_1(b)]\times\dots\times[\pi_d(a),\pi_d(b)]\subset\R^d$, such that $x,y\in[a,b]$ and $\supp(\psi)\subset[a,b]$.

\item Given the reconstruction $f$ of Theorem \eqref{stochastic_recon_covariance}, and a test function $\psi$ with support in $[0,\tilde R]^d$, we say that the effective support of $(f-F_x)(\psi_x^\lambda)$ is given by $[\pi_1(x),\pi_1(x)+2\lambda^{s_1}\tilde R]\times\dots\times[\pi_d(x),\pi_d(x)+2\lambda^{s_d}\tilde R]$.
\end{itemize}

\noindent Furthermore, given an $e\le d$, we call $[a,b]_e := [\pi_1(a),\pi_1(b)]\times\dots\times[\pi_e(a),\pi_e(b)]$ as well as $[\pi_1(x),\pi_1(x)+2\lambda^{s_1}\tilde R]\times\dots\times[\pi_e(x),\pi_e(x)+2\lambda^{s_e}\tilde R]$ the stochastic effective support of the respective expressions.
\end{definition}

\noindent
Note that the effective support of $(F_x-F_y)(\psi)$ is precisely the smallest interval that the expression sees. For $(f-F_x)(\psi_x^\lambda)$ however, $x$ and $\supp(\psi_x^\lambda)$ are contained in the smaller interval $[\pi_1(x),\pi_1(x)+\lambda^{s_1}\tilde R]\times\dots\times[\pi_d(x),\pi_d(x)+\lambda^{s_d}\tilde R]$ without the factor $2$. The reason we use this factor is that the reconstruction theorem causes a ``fattening'' effect, where we need some extra distance between two test functions to put in another wavelet. The easiest way to do this is by extending the effective support to be slightly bigger. Instead of the factor $2$, one can use any factor $\eta\ge 1$, at the cost that the constants appearing in $\lesssim$ will depend on $\eta$.

With these definitions, we can now give a precise meaning to points $(x_1,y_1)$ $(x_2,y_2)$ being ``far enough'' from each other. We want them to be so far apart, that the (stochastic) effective supports of $(F_{x_i}-F_{y_i})(\psi_{y_i}^{\epsilon_i})$ are disjoint. This is motivated by two observations: 

\begin{enumerate}
\item If we think of the example of reconstructing the Walsh integral against white noise \eqref{whiteNoise}, we have that white noise $\xi(\psi)$ is independent to $\mathcal F^{(i)}_{x_i}$ if $\supp(\psi)\subset \R^{i-1}\times[x_i,\infty)\times\R^{d-i+1}$, and has expected value $E(\xi(\psi)) = 0$. This implies that $\abs{E\left(\prod_{i=1}^2 (F_{x_i}-F_{y_i})(\psi_i)\right)} = 0$, where $F_x(\psi) = X(x)\xi(\psi)$ and whenever the effective stochastic supports of this terms are disjoint. So in this example, disjoint effective stochastic supports give us a true improvement in the regularity over simply applying Hölder's inequality to the product.

\item Let $x_i,y_i,\psi_i, i=1,2$ be such that the effective stochastic supports of $(F_{x_i}-F_{y_i})(\psi_i)$ are disjoint and assume w.l.o.g. $\pi_1 (x_1)\le \pi_1 (x_2)$. Then the disjoint effective stochastic supports imply that $\pi_1(y_1)\le \pi_1(x_2)$ and $\supp(\psi_1)\subset (-\infty,\pi_1(x_2)]$ which shows that $(F_{x_1}-F_{y_1})(\psi_1)$ is $\mathcal F^{(1)}_{\pi_1(x_2)}$-measurable. It follows that

\begin{align*}
\abs{E\left(\prod_{i=1}^2 (F_{x_i}-F_{y_i})(\psi_i)\right)} &= \abs{E\left((F_{x_1}-F_{y_1})(\psi_1)E^{\mathcal F_{\pi_1(x_2)}^{(1)}} (F_{x_2}-F_{y_2})(\psi_2)\right)} \\
&\le \norm{(F_{x_1}-F_{y_1})(\psi_1)}_{L_2}\norm{E^{\mathcal F_{\pi_1(x_2)}^{(1)}} (F_{x_2}-F_{y_2})(\psi_2)}_{L_2}\,.
\end{align*}

\noindent So if we assume that inequality \eqref{coherence2} from the classical setting holds, disjoint effective stochastic supports also give us better bounds on $E\left(\abs{\prod_{i=1}^2 (F_{x_i}-F_{y_i})(\psi_i)}\right)$ than simply applying Hölder's inequality.
\end{enumerate}

\noindent
Let us formally write down the coherence property in this setting:

\begin{definition}[stochastic coherence in the covariance setting]\label{defCoherenceCov}
Let $(F_x)_{x\in\R^d}$ be a stochastic germ with $p=2$ and let $e\le d$. We call $(F_x)$ stochastically $\gamma$-coherent in the covariance setting, if there is an $\alpha <\gamma$, such that the following holds for all test functions $\psi_1,\psi_2\in C_c^r(\R^d)$ with $\norm{\psi_i}_{C_c^r} = 1$ and $\supp(\psi_i)\subset[0,\tilde R]^e\times[-\tilde R,\tilde R]^{d-e}$ for some $\tilde R>0$ for $i=1,2$, for any compact set $K\subset\R^d$ and $x_1,x_2,y_1,y_2\in K$ with $\pi_i(x_j)\le \pi_i(y_j)$ for $i=1,\dots, e, j=1,2$, such that the stochastic effective supports of $(F_{x_i}-F_{y_i})((\psi_i)_{y_i}^{\epsilon_i})$, $i=1,2$ are disjoint and $\epsilon_1,\epsilon_2>0$:

\begin{align}
\norm{(F_{x_1}-F_{y_1})((\psi_1)_{y_1}^{\epsilon_1})}_{L_2} &\lesssim \epsilon_1^\alpha(\abs{x-y}+\epsilon)^{\gamma-\frac E2-\alpha} \label{coherenceCov1}\\
\abs{E\left(\prod_{i=1}^2 (F_{x_i}-F_{y_i})((\psi_i)_{y_i}^{\epsilon_i})\right)} &\lesssim \prod_{i=1}^2\epsilon_i^\alpha(\abs{x_i-y_i}+\epsilon_i)^{\gamma-\alpha}\label{coherenceCov2}\,,
\end{align}

\noindent where the constant appearing in $\lesssim$ is allowed to depend on $K$ and $\tilde R$.
\end{definition}

\noindent Note that the stochastic dimension $e$ of our problem is no longer an intrinsic property of the germ $(F_x)_{x\in\R^d}$, but only appears in the coherence property. If there is any room for confusion, we will write that $(F_x)_{x\in\R^d}$ is stochastically $(\gamma,\alpha,e)$ coherent in the covariance setting.

We want to stress that this condition is not a generalization of \eqref{coherence2}, not even in $p=2$: While one can show a bound on $\sum_{x\neq y}E(g_x g_y)$ under the conditions of Definition \ref{defCoherence}, it is in general not true, that $\abs{E(g_x g_y)}\lesssim \norm{E^{\mathcal F^{(i)}_{\pi_i(x)}}g_x}_{L_2}\norm{E^{\mathcal F^{(i)}_{\pi_i(y)}}g_y}_{L_2}$ for $\pi_i(x)<\pi_i(y)$. So \eqref{coherence2} does not imply \eqref{coherenceCov2}.

Let us now formulate the main result of this section:

\begin{theorem}[stochastic reconstruction in the covariance setting]\label{stochastic_recon_covariance}
Let $(F_x)_{x\in\R^d}$ be a random germ for some $p\ge 2$ which is stochastically $(\gamma,\alpha, e)$-coherent in the covariance setting for some $\gamma > 0, \alpha<0$ and $e\le d$. Let $\tilde r := r\cdot\min_{i=1,\dots, d}(s_i)$ and assume $\alpha+\tilde r > 0$ and $\gamma-\frac E2+\tilde r > 0$.

Then, where is a unique (up to modifications) random distribution $f$, such that the following holds for any test function $\psi\in C_c^r$ with $\norm{\psi}_{C_c^r}=1$ and with support $\supp(\psi)\subset[\tilde C,\tilde R]^e\times[-\tilde R,\tilde R]^{d-e}$ for some $0<\tilde C<\tilde R$, and for any compact set $K\subset\R^d$, $x_1,x_2\in K$ and $\lambda\in(0,1]$ such that the stochastic effective supports of $(f-F_{x_i})(\psi)$ are disjoint for $i=1,2$:

\begin{align}
\norm{(f-F_{x_1})(\psi_{x_1}^\lambda)}_{L_2}&\lesssim \lambda^{\gamma-\frac E2} \label{uniqueCov1}\\
\abs{E\left(\prod_{i=1}^2(f-F_{x_i})(\psi_{x_i}^\lambda)\right)}&\lesssim\lambda^{2\gamma}\label{uniqueCov2}
\end{align} 

\noindent where the constant in $\lesssim$ again depends on $K$ and the constants $\tilde C, \tilde R$.
\end{theorem}

\noindent Condition \eqref{uniqueCov2} is the weakest additional condition we could find, which still uniquely characterizes the reconstruction. However, it would be closer to the coherence condition to analyze the expression $\prod_{i=1}^2 (f-F_x)((\psi_i)_{x_i}^{\lambda_i})$ for different test functions $\psi_1,\psi_2$ and different $\lambda_1,\lambda_2\in(0,1]$. A simple corollary of the proof of Theorem \ref{stochastic_recon_covariance} shows that \eqref{uniqueCov2} still holds in this setting:

\begin{corollary}\label{cor:uniqueCovMult}
Let $(F_x)$ be as in Theorem \ref{stochastic_recon_covariance} and let $f$ be the reconstructed random distribution. Let $\psi_1,\psi_2\in C_c^r$ be two test functions with $\norm{\psi_i}_{C_c^r}=1$ and with support $\supp(\psi_i)\subset[\tilde C,\tilde R]^e\times[-\tilde R,\tilde R]^{d-e}$ for some $0<\tilde C<\tilde R$, $i=1,2$. For any compact set $K\subset\R^d$, $x_1,x_2\in K$ and $\lambda_1,\lambda_2\in(0,1]$ such that the stochastic effective supports of $(f-F_{x_i})((\psi_i)_{x_i}^{\lambda_i})$ are disjoint, it holds that:

\begin{equation*}
\abs{E\left(\prod_{i=1}^2(f-F_{x_i})((\psi_i)_{x_i}^{\lambda_i})\right)}\lesssim\lambda_1^\gamma\lambda_2^\gamma\,,
\end{equation*}

\noindent
where the constant in $\lesssim$ depends on $K, \tilde C$ and $\tilde R$.
\end{corollary}

\noindent Since the proof of this result relies on the proof of Theorem \ref{stochastic_recon_covariance}, we will present it directly after said proof in Section \ref{sec:proofs}.

\section{Proofs of Theorem \ref{stochastic_reconstruction} and \ref{stochastic_recon_covariance}}\label{sec:proofs}

The proofs in this section make use of technical tools, which we did not yet introduce, namely the BDG-type inequality \eqref{BDG} and wavelet techniques. 
As we will make heavy use of wavelet techniques especially, let us fix a wavelet bases $(\phi,\Phi)$ here, as well as two constants $0<C<R$ such that $\supp(\phi),\supp(\hat\phi)\subset[C,R]^d$ holds for all $\hat\phi\in\Phi$. A proper introduction to wavelets can be found in the appendix.

\subsection{Proof of Theorem \ref{stochastic_reconstruction}}

The main idea of the proof is to take the reconstructing sequence

\begin{equation}\label{eq:reconstructingSequence}
f^{(n)}(\psi) := \sum_{x\in\Delta_n} F_x(\phi_x^n)\scalar{\phi_x^n,\psi}\,,
\end{equation}

\noindent and show the following three properties under the assumptions of Theorem \ref{stochastic_reconstruction}

\begin{enumerate}
\item $f^{(n)}$ converges to some distribution $f$ for $n\to\infty$.
\item $f$ fulfills inequalities \eqref{unique1} and \eqref{unique2}.
\item There is at most one distribution $f$ fulfilling inequalities \eqref{unique1} and \eqref{unique2}.
\end{enumerate}

%\begin{remark}
%We show property 1 only under an additional assumption, namely that $F_x$ is a regular enough random distribution to invoke the second part of Lemma \ref{lem:PnPsiConverges}. We will first show properties 1 and 2 under this additional assumptions, and use this to construct an $f$ which fulfills property 2 without this assumption.
%\end{remark}

\noindent Our strategy for the proof is to decompose $f^{(l)} = f^{(n)} + \sum_{k=n}^{l-1}(f^{(k+1)}-f^{(k)})$ and split a test function $\psi$ into $\psi = P_n\psi + \sum_{m=n}^\infty \hat P_m\psi$, so that it suffices to test $f^{(n+1)}-f^{(n)}$ against $P_n\psi$ and $\hat P_n\psi$. We first show bounds for these terms in Lemma \ref{lem:technical1} and \ref{lem:technical2}, before proving properties 1 and 2 in Lemma \ref{lem:prop1+2} and Lemma \ref{lem:CbetaUnnessecary}. Property 3 will be shown in Lemma \ref{lem:prop3}.

Before we start with the proofs, let us recall the assumptions of Theorem \ref{stochastic_reconstruction}:

\begin{enumerate}[i)]
\item $(F_x)$ is a random germ with stochastic dimension $e\le d$.
\item $(F_x)$ is stochastically $(\gamma,\alpha)$-coherent for some $\gamma>0,\alpha<0$.
\item $r$ is big enough, such that $\alpha+\tilde r>0, \gamma-\frac E2+\tilde r>0$, for $\tilde r$ as in Theorem \ref{stochastic_reconstruction}.
\end{enumerate}

\noindent With this in mind, we can state our first technical lemma. Since $\norm{(f^{(n+1)}-f^{(n)})(P_n \psi)}_{L_p}$ is rather technical to bound, we do so in a separate lemma and start with the other necessary bounds. As it turns out, instead of estimating $\norm{(f^{(n+1)}-f^{(n)})(\hat P_n\psi)}_{L_p}$, estimating $\norm{(f^{(n+1)}-f^{(n)}-F_x)(\hat P_n\psi)}_{L_p}$ allows us to combine the proofs of properties 1 and 2, so this is the term we want to bound from above:

\begin{lemma}\label{lem:technical1}
Let $(F_x)$ be as in Theorem \ref{stochastic_reconstruction}, choose some $\tilde C>0$ and let $B := \min_{i=1,\dots,e} \left(\frac {\tilde C}{2R}\right)^{-\frac 1{s_i}}$. Then it holds that for any $x$ in some compact $K\subset\R^d$, $B2^{-n}\le\lambda\le 1$ and any test function $\psi\in C_c^r$ with support in $[\tilde C,\infty)^e\times\R^{d-e}$, we have:

\begin{align}
\norm{(f^{(n+1)}-f^{(n)}-F_x)(\hat P_n\psi_x^\lambda)}_{L_p} &\lesssim 2^{-n(\alpha+\tilde r)}\lambda^{\gamma-\frac E2-\alpha-\tilde r}\norm{\psi}_{C_c^r}\label{eq:techIneq1}\\
\norm{E^{\mathcal F^{(i)}_{\pi_i(x)}}(f^{(n+1)}-f^{(n)}-F_x)(\hat P_n\psi_x^\lambda)}_{L_p}&\lesssim 2^{-n(\alpha+\tilde r)}\lambda^{\gamma-\alpha-\tilde r}\norm{\psi}_{C_c^r}\label{eq:techIneq2}\\
\norm{E^{\mathcal F_{\pi_i(x)}^{(i)}}(f^{(n+1)}-f^{(n)})(P_n\psi_x^\lambda)}_{L_p}&\lesssim 2^{-n\gamma}\norm{\psi}_{C_c^r}\,,\label{eq:techIneq3}
\end{align}

\noindent for $i=1,\dots,e$.
\end{lemma}

\begin{proof}
Let us start with inequality \eqref{eq:techIneq1}. Since $\hat P_n\psi\in (V^n)^{\perp V_{n+1}}$, it follows that $f^{(n)}(\hat P_n\psi) = 0$. Using the definition of $f^{(n)}$ as well as $F_x(\hat\phi_y^n) = \sum_{z\in\Delta_{n+1}}F_x(\phi_z^{n+1})\scalar{\phi_z^{n+1},\hat\phi_y^m}$, we get

\begin{align*}
(f^{(n+1)}-f^{(n)}-F_x)(\hat P_n\psi_x^\lambda) &= \sum_{y\in\Delta_n}\sum_{\hat\phi\in\Phi}(f^{(n+1)}-F_x)(\hat\phi_y^n)\scalar{\hat\phi_y^n,\psi_x^\lambda}\\
&= \sum_{y\in\Delta_n}\sum_{z\in\Delta_{n+1}}\sum_{\hat\phi\in\Phi}(F_z-F_x)(\phi_z^{n+1})\scalar{\phi_z^{n+1},\hat\phi_y^n}\scalar{\hat\phi_y^n,\psi_x^\lambda}\,.
\end{align*}

\noindent For all non-zero terms in this sum, the supports of $\hat\phi_y^n$ and the support of $\psi_x^\lambda$ need to overlap, as well as the supports of $\phi_z^{n+1}$ and $\hat\phi_z^n$. Together with $\supp(\psi)\in [\tilde C,\infty)^d$ and $\supp(\phi), \supp(\hat\phi)\in[C,R]$, this implies that all non-zero terms fulfill $\pi_i(z)\ge \pi_i(y)-2^{-(n+1)s_i}R \ge \pi_i(x)+\lambda^{s_i}\cdot \tilde C- 2^{-ns_i}\cdot 2R\ge \pi_i(x)$ by our choice of $\lambda$ for $i=1\dots,e$, so we can use stochastic coherence. Note that the number of $z$, such that $\scalar{\phi_z^{n+1},\hat\phi_y^{n}}\neq 0$ for any fixed $y$ is of order $1$, and the number of $y$, such that $\scalar{\hat\phi_y^n,\psi_x^\lambda}\neq 0$ is of order $2^{n\abs S}\lambda^{\abs S}$. Furthermore, all non-zero terms will have $\abs{z-x}\lesssim\lambda+2^{-n}\lesssim\lambda$. Applying \eqref{LemIneq1} to $\abs{\scalar{\phi_z^{n+1},\hat\phi_y^n}}$ as well as \eqref{LemIneq2} to $\abs{\scalar{\hat\phi_y^n,\psi_x^\lambda}}$ and using stochastic coherence, we get:

\begin{align*}
\norm{\hat P_n(f^{(n+1)}-f^{(n)}-F_x)(\psi_x^\lambda)}_{L_p} &\le \sum_{y\in\Delta_n}\sum_{z\in\Delta_{n+1}}\sum_{\hat\phi\in\Phi}\underbrace{\norm{(F_z-F_x)(\phi_z^{n+1})}_{L_p}}_{\lesssim 2^{-n\frac {\abs S}2 -n\alpha}\lambda^{\gamma-\frac E2-\alpha}}\\&\qquad\qquad\qquad\qquad\times\underbrace{\abs{\scalar{\phi_z^{n+1},\hat\phi_y^n}}}_{\lesssim 1}\underbrace{\abs{\scalar{\hat\phi_y^n,\psi_x^\lambda}}}_{\lesssim 2^{-n\frac {\abs S}2-n\tilde r}\lambda^{-{\abs S}-\tilde r}\norm{\psi}_{C_c^r}} \\&\lesssim 2^{-n\alpha-n\tilde r}\lambda^{\gamma-\frac E2-\alpha-\tilde r}\norm{\psi}_{C_c^r}\,,
\end{align*}

\noindent which shows \eqref{eq:techIneq1}. \eqref{eq:techIneq2} follows analogously. To show \eqref{eq:techIneq3}, we decompose 

\begin{equation}\label{eq:techEqual1}
(f^{(n+1)}-f^{(n)})(P_n\psi_x^\lambda) = \sum_{y\in\Delta_n}(f^{(n+1)}-f^{(n)})(\phi_y^n)\scalar{\phi_y^n,\psi_x^\lambda}
\end{equation}

\noindent and calculate

\begin{align}
(f^{(n+1)}-f^{(n)})(\phi_x^n) &= \sum_{z\in\Delta_{n+1}}F_z(\phi_z^{n+1})\scalar{\phi_z^{n+1},\phi_x^{n}}-\sum_{y\in\Delta_n}F_y(\phi_y^n)\scalar{\phi_y^n,\phi_x^n}\nonumber\\ 
&=\sum_{z\in\Delta_{n+1}}\sum_{k\in\Delta_{n+1}}F_z(\phi_z^{n+1})\scalar{\phi_z^{n+1},a_k^{(n+1)}\phi_{x+k}^{n+1}} -F_x(\phi_x^n) \nonumber\\
&= \sum_{k\in\Delta_{n+1}} a_k^{(n+1)} F_{x+k}(\phi_{x+k}^{(n+1)}) - \sum_{k\in\Delta_{n+1}} a_k^{(n+1)} F_{x}(\phi_{x+k}^{(n+1)})\nonumber\\
&=\sum_{k\in\Delta_{n+1}} a_k^{(n+1)}(F_{x+k}-F_x)(\phi_{k+x}^{(n+1)}).\label{eq:(III)Sum}
\end{align}

\noindent Note that $a_k^{(n+1)} \neq 0$ only for $\abs k\le dR2^{-n}$, so the sum is finite. Furthermore, $\pi_i(k)\ge 0$ for $i=1,\dots,n$ and all non-zero terms and the support of $\phi$ is positive, so we can use stochastic reconstruction to show:

\begin{align}
\norm{E^{\mathcal F^{(i)}_{\pi_i(x)}}(f^{(n+1)}-f^{(n)})(\phi_x^n)}_{L_p}&\le \sum_{k\in\Delta_{n+1}} \abs{a_k^{(n+1)}}\norm{E^{\mathcal F_{\pi_i(x)}^{(i)}}(F_{x+k}-F_x)(\phi_{k+x}^{(n+1)})}_{L_p}\nonumber\\
&\lesssim 2^{-\frac{(n+1)\abs S}{2}-\alpha(n+1)}(\abs k+2^{-n-1})^{\gamma-\alpha}\nonumber\\
&\lesssim 2^{-n\gamma-n\frac{\abs S}{2}}\,.\label{ineq2}
\end{align}

\noindent Putting this into \eqref{eq:techEqual1} gives us

\begin{align*}
\norm{E^{\mathcal F^{(i)}_{\pi_i(x)}}(f^{(n+1)}-f^{(n)})(P_n\psi_x^\lambda)}_{L_p} &\le \sum_{y\in\Delta_n}\norm{E^{\mathcal F^{(i)}_{\pi_i(x)}}(f^{(n+1)}-f^{(n)})(\phi_x^n)}_{L_p}\underbrace{\abs{\scalar{\phi_y^n,\psi_x^\lambda}}}_{\le2^{-n\frac{\abs S}2}\lambda^{-\abs S}\norm{\psi}_{C_c^r}}\\
&\lesssim 2^{-n\gamma}\norm{\psi}_{C_c^r}\,,
\end{align*}

\noindent where we used \eqref{LemIneq1} as well as the fact that the order of non-zero terms is given by $2^{n\abs S}\lambda^{\abs S}$.
\end{proof}

\noindent The last inequality we need for the main result is a bound on $\norm{(f^{(n+1)}-f^{(n)})(P_n\psi)}$. If one uses the techniques used in Lemma \ref{lem:technical1}, this ends up being $\lesssim 2^{-n(\gamma-\frac E2)}$, which is not good enough for our result. One can improve this bound with the use of the BDG type inequality \eqref{BDG}.

Since we will this technique both in the proof of existence and uniqueness of the reconstruction, let us formulate the technique itself in its own lemma:

\begin{lemma}\label{lem:stochTechnique}
Let $h$ be a random distribution, and assume there exists a constant $\tilde R>0$ such that for $i=1,\dots, e, n\in\mathbb N$ and $x\in\Delta_n$, $h(\phi_x^n)$ is $\mathcal F_{\hat S_i(x)}^{(i)}$-measurable for $\hat S_i(x) = \pi_i(x)+2^{-n s_i}\tilde R$. Further, assume that $h$ fulfills the inequalities

\begin{align}
\norm{h(\phi_x^n)}_{L_p}&\lesssim 2^{-n\gamma-n\frac{\abs S-E}{2}}\label{eq:ineq1}\\
\norm{E^{\mathcal F_{\pi_i(x)}^{(i)}}h(\phi_x^n)}_{L_p}&\lesssim 2^{-n\gamma-n\frac{\abs S}{2}}\,,\label{eq:ineq2}
\end{align}

\noindent for $i=1,\dots, e$. Then it holds that for each $x_0$ in any constant $K\subset\R^d$, $1\le\lambda\le 2^{-n}$ and $\psi\in C_c^r$ with support in $[-\hat R,\hat R]^d$ for some $\hat R>0$, we have

\begin{equation*}
\norm{h(P_n\psi_{x_0}^\lambda)}_{L_p}\lesssim 2^{-n\gamma}\lambda^{-\frac E2}\norm{\psi}_{C_c^r}\,,
\end{equation*}

\noindent where the constant in $\lesssim$ is allowed to depend on $K,\tilde R$ and $\hat R$.
\end{lemma}

\begin{remark}
Note that any $h$ with stochastic dimension $e$ will automatically fulfill the measurability condition given above. The reason we want to work with this weaker condition is that we want to apply the above lemma to $h=f^{(n+1)}-f^{(n)}$, which is in general not adapted but has a ``fattening of the support'' happening. More rigorously, given a test function $\psi$ with support in $(-\infty,y]\times\R^{d-1}$, $(f^{(n+1)}-f^{(n)})(\psi)$ is not $\mathcal F_{\pi_1(y)}^{(1)}$ measurable, but $\mathcal F_{\pi_1(y)+R2^{-ns_1}}^{(1)}$.
\end{remark}

\begin{remark}
We want to stress that this lemma is the only place, where we need \eqref{BDG}. Thus, this can be seen as the technique that gives us the improvement in regularity over classical reconstruction.
\end{remark}

\begin{proof}
We decompose

\begin{equation}\label{def:gx}
h(P_n\psi_{x_0}^\lambda) = \sum_{x\in \Delta_n}\underbrace{h(\phi_x^n)\scalar{\phi_x^n,\psi_{x_0}^\lambda}}_{=:g_x}.
\end{equation}

\noindent Our strategy will be to use \eqref{BDG} on the outermost sum and then iterate this step over all stochastic directions. To this end, we specify the definition 

\begin{equation*}
\Delta_n^{a,b} := \left\{\sum_{i=a}^b 2^{-ns_i}l_i\vec e_i ~\bigg\vert~ l_i\in\mathbb Z\right\}
\end{equation*} 

\noindent to be the rescaled mesh in the variables $a$ to $b$. Unfortunately, naively applying \eqref{BDG} does not work for our sum: If we set $Z_y := \sum_{x\in\Delta_n^{2,d}} g_{(y,x)}$, so that $h(P_n\psi_{x_0}^\lambda) = \sum_{y\in\Delta_n^{1,1}} Z_y$, we recall that $Z_y$ is $\mathcal F^{(1)}_{\hat S_1(y)}$-measurable. \eqref{BDG} would now allow us to condition $Z_y$ onto $\mathcal F_{\hat S_1(y)-2^{-ns_1}}^{(1)}$, which results in a random variable we do not control.

Instead, we are going to split the sum into $\lceil\tilde R\rceil^e$ many sums in which the nets have a bigger mesh, to get the right conditioning. To this end, let $0\le r_i\le \lceil \tilde R\rceil-1, i=1,\dots, e$ be natural numbers and set

\begin{equation*}
r_n^{(k)} := S^{2^{-n}}(0,\dots,0,r_{k+1},\dots,r_{e},0,\dots,0) \in\Delta_n^{k+1,d}
\end{equation*}  

\noindent for $k = 0,\dots, e$. We define the rescaled and shifted net for each such $r_n^{(k)}$ by

\begin{equation*}
\Delta_n^{a,b}(r^{(k)}_n) = \left\{\sum_{i=a}^b \lceil C_i\rceil\cdot 2^{-ns_i}l_i\vec e_i+r_n^{(k)}~\bigg\vert~ l_i\in\mathbb Z\right\}\,.
\end{equation*} 

\begin{figure}
\begin{tikzpicture}
\node(Delta12) at (0, 0) {\includegraphics[scale=1.2]{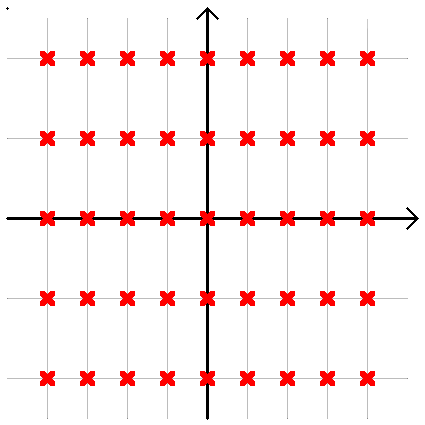}};
\node(BoxDelta12) at (1.5, 1.5) {\colorbox{white}{\fbox{\color{red} $\Delta_n^{1,2}$}}};

\node(Delta11) at (5, 0) {\includegraphics[scale=1.2]{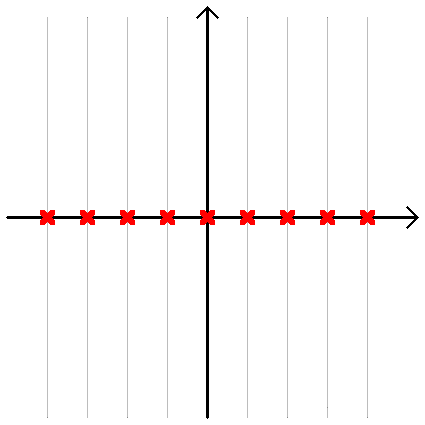}};
\node(BoxDelta11) at (6.5, 1.5) {\colorbox{white}{\fbox{\color{red} $\Delta_n^{1,1}$}}};

\node(Delta22) at (10, 0) {\includegraphics[scale=1.2]{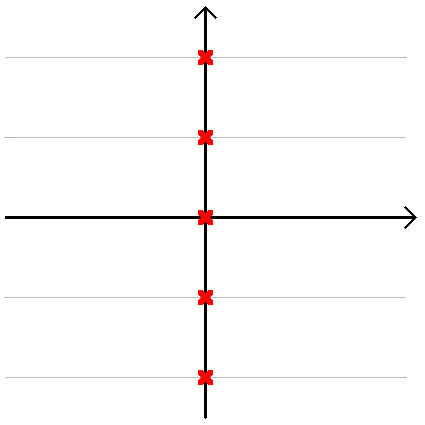}};
\node(BoxDelta22) at (11.5, 1.5) {\colorbox{white}{\fbox{\color{red} $\Delta_n^{2,2}$}}};
\end{tikzpicture}
\end{figure}

\begin{figure}
\begin{tikzpicture}
\node(DeltaR) at (0,0) {\includegraphics[scale=1.2]{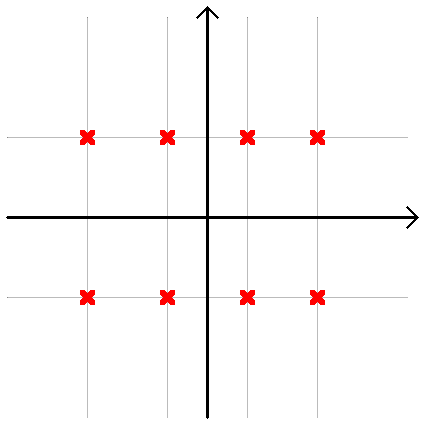}};
\node(BoxDeltaR) at (1.5, 1.5) {\colorbox{white}{\fbox{\color{red} $\Delta_n^{1,2}(r_n^{(0)})$}}};
\end{tikzpicture}
\end{figure}

\noindent As these definitions are rather technical, we have illustrated $\Delta_n^{1,2}, \Delta_n^{1,1}, \Delta_n^{2,2}$ as well as $\Delta_n^{1,2}(r_n^{(0)})$ in the case $e=d=2$, $\lceil \tilde R \rceil = 2$ and $r=(1,1)$. We split the expression

\begin{equation*}
h(P_n\psi_{x_0}^\lambda) = \sum_{r_n^{(0)}} \sum_{x\in\Delta_n^{1,d}(r_n^{(0)})}g_x
\end{equation*}

\noindent into $\lceil\tilde R\rceil^e$ many sums, where the first sum runs over all possible $r_n^{(0)}$. It suffices to find a bound for the sum over $x$ as the sum over $r_n^{(0)}$ will merely add the constant factor $\lceil \tilde R\rceil^e$. Using \eqref{BDG}, we now do indeed get

\begin{align}
\norm{\sum_{x\in\Delta_n^{1,d}(r_n^{(0)})}g_x}_{L_p} \lesssim  &\sum_{y\in\Delta_n^{1,1}(r_1)}\norm{\sum_{x\in\Delta_n^{2,d}(r_n^{(1)})}E^{\mathcal F^{(1)}_y} g_{(y,x)}}_{L_p}\nonumber\\&+\left(\sum_{y\in\Delta_n^{1,1}(r_1)}\norm{\sum_{x\in\Delta_n^{2,d}(r_n^{(1)})}g_{(y,x)}-E^{\mathcal F^{(1)}_y}g_{(y,x)}}_{L_p}^2\right)^{\frac 12},\label{iteration}
\end{align}

\noindent We want to iterate this step over all stochastic directions. To this end, we use the simplified notation $(y,x):= y+x$ for $y\in\Delta_n^{a,b}, x\in\Delta_n^{c,d}$ with $[a,b]\cup[c,d]=\emptyset$ and define

\begin{align*}
g_x &:= h(\phi_x^n)\scalar{\phi_x^n,\psi_{x_0}^\lambda} &&x\in\Delta_n^{1,d}\\
g_{y,x} &:= g_{(y,x)}-E^{\mathcal F^{(1)}_y}g_{(y,x)} &&x\in\Delta_n^{2,d}, y\in\Delta_n^{1,1} \\
g_{y_1,y_2,x} &= g_{y_1,(y_2,x)}-E^{\mathcal F^{(2)}_{y_2}}g_{y_1,(y_2,x)} &&x\in\Delta_n^{3,d}, y_1\in\Delta^{1,1}_n, y_2\in\Delta_n^{2,2}\\ &\vdots \\
g_{y_1,\dots,y_e,x} &:=  g_{y_1,\dots,y_{e-1}(y_e,x)}-E^{\mathcal F^{(e)}_{y_e}}g_{y_1,\dots,y_{e-1}(y_e,x)} &&x\in\Delta_n^{e+1,d}, y_1\in\Delta_n^{1,1},\dots,y_e\in\Delta_n^{e,e},
\end{align*}

\noindent with the goal of finding bounds for $\norm{\sum_{x\in\Delta_n^{i+1,d}(r_n^{(i)})}g_{y_1,\dots,y_i,x}}_{L_p}$ using a backward induction over $i=0,\dots,e$. Note that since $(\mathcal F_{t}^{(i)})_t$, $i=1,\dots,e$ are orthogonal, $g_{y_1,\dots,y_i,x}$ has the same adaptedness properties as $g_x$ for $i=1,\dots,e$, so we can apply \eqref{BDG} to $\norm{\sum_{x\in\Delta_n^{i+1,d}(r_n^{(i)})}g_{y_1,\dots,y_i,x}}_{L_p}$. 

Recall, that \eqref{LemIneq1} gives us

\begin{equation}
\abs{\scalar{\phi_x^n,\psi_{x_0}^\lambda}}\lesssim 2^{-\frac{n\abs S}2}\lambda^{-\abs S}\norm{\psi}_{C_c^r}.
\end{equation}  

\noindent Together with \eqref{eq:ineq1}, this implies 

\begin{align}
\norm{g_x}_{L_p} &\le \norm{h(\phi_{x}^n)}_{L_p}\abs{\scalar{\phi_x^n,\psi_{x_0}^\lambda}} \nonumber\\
&\lesssim 2^{-n\gamma-n\frac{\abs S-E}{2}}2^{-n\frac {\abs S}2}\lambda^{-\abs S}\norm{\psi}_{C_c^r} \nonumber\\
&= 2^{-n\gamma-n\abs S+n\frac E2}\lambda^{-\abs S}\norm{\psi}_{C_c^r},\label{g1}
\end{align}

\noindent and analogously, using \eqref{eq:ineq2}

\begin{equation}\label{g2}
\norm{E^{\mathcal F^{(1)}_y}g_{(y,x)}}_{L_p} \lesssim 2^{-n\gamma-n\abs S}\lambda^{-\abs S}\norm{\psi}_{C_c^r}.
\end{equation}

\noindent Observe that for any $i=1,\dots,e$, $g_{y_1,\dots,y_i,x}$ is just a finite linear combination of $g_{(y_1,\dots,y_i,x)}$ conditioned on different sigma-algebras. By the contraction properties of the conditional expectation, \eqref{g1} and \eqref{g2} thus hold for any $g_{y_1,\dots,y_i,x}, i=0,\dots, e$.

We claim, that the following holds for $i=0,\dots,e$:

\begin{equation*}
\norm{\sum_{x\in\Delta_n^{i+1,d}(r_n^{(i)})} g_{y_1,\dots,y_i,x}}_{L_p}\lesssim 2^{-n\gamma-n\frac {E_i}2}\lambda^{-\frac {E+E_i}{2}}\norm{\psi}_{C_c^r},
\end{equation*}

\noindent where $E_i := \sum_{j=1}^i s_j$. As said above, we show this claim by backwards induction. Note that the $g_x\neq 0$ implies $\scalar{\phi_x^n,\psi_{x_0}^\lambda}\neq 0$, which is only true for $\cong 2^{ns_i}\lambda^{s_i}$ many $x$ for every coordinate $i = 1,\dots,d$. Thus, using \eqref{g1}, we get the induction start

\begin{equation}\label{g3}
\sum_{x\in\Delta_n^{e+1,d}(r_n^{(e)})}\norm{g_{y_1,\dots,y_e,x}}_{L_p} \lesssim 2^{-n\gamma-n\abs S+n\frac E2}\lambda^{-\abs S}\cdot2^{n(\abs S-E)}\lambda^{\abs S-E}\norm{\psi}_{C_c^r} = 2^{-n\gamma-n\frac E2}\lambda^{-E}\norm{\psi}_{C_c^r}.
\end{equation}

\noindent For the induction step, observe that

\begin{align*}
\sum_{x\in\Delta_n^{i+1,d}(r_n^{(i)})}\norm{E^{\mathcal F^{(i)}_{y_i}}g_{y_1,\dots,y_{i-1},(y_i,x)}}_{L_p} &\lesssim 2^{-n\gamma-n\abs S}\lambda^{-\abs S}\cdot2^{n(\abs S-E_i)}\lambda^{\abs S-E_i}\norm{\psi}_{C_c^r} \\&= 2^{-n\gamma -nE_i} \lambda^{-E_i}\norm{\psi}_{C_c^r}
\end{align*}

\noindent holds, and that summing over the squares and taking the square root has the effect of multiplying a factor of $2^{\frac {s_in}2}\lambda^{\frac {s_i}2}$:

\begin{equation*}
\left(\sum_{k=1}^{\lfloor 2^{s_in}\lambda^{s_i}\rfloor} a^2\right)^{\frac 12} \le 2^{\frac {s_in}2}\lambda^{\frac{s_i}{2}}a.
\end{equation*}

\noindent Now, using \eqref{iteration}, we get

\begin{align*}
\norm{\sum_{x_\in\Delta_n^{i+1,d}(r_n^{(i)})} g_{y_1,\dots,y_i,x}}_{L_p} &\lesssim \underbrace{\sum_{y_{i+1}\in\Delta_n^{i+1,i+1}(\pi_{i+1}(r_n))}\sum_{x\in\Delta_n^{i+2,d}(r_n^{(i+1)})} \norm{E^{\mathcal F^{(i+1)}_{y_{i+1}}} g_{y_1,\dots,y_i,(y_{i+1},x)}}_{L_p}}_{\lesssim 2^{-n\gamma-nE_i}\lambda^{-E_i}\norm{\psi}_{C_c^r}} \\&\qquad+ \Bigg(\sum_{y_{i+1}\in\Delta_n^{i+1,i+1}(\pi_{i+1}(r_n))}\underbrace{\norm{\sum_{x\in\Delta_n^{i+2,d}(r_n^{(i+1)})}g_{y_1,\dots,y_{i+1},x}}_{L_p}^2}_{\lesssim \left(2^{-n\gamma-n\frac{E_{i+1}}2}\lambda^{-\frac{E+E_{i+1}}{2}}\norm{\psi}_{C_c^r}\right)^2}\Bigg)^{\frac 12} \\
&\lesssim 2^{-n\gamma-n\frac {E_i}2}\lambda^{-\frac{E+E_i}{2}}\norm{\psi}_{C_c^r}.
\end{align*}

\noindent This shows the claim. Therefore, $i = 0$ implies

\begin{equation*}
\norm{h(P_n\psi_{x_0}^\lambda)}_{L_p} = \norm{\sum_{x\in\Delta_n}g_x}_{L_p} \lesssim \sum_{r_n^{(0)}}\norm{\sum_{x\in\Delta_n^{1,d}(r_n^{(0)})}g_x}_{L_p} \lesssim 2^{-n\gamma}\lambda^{-\frac E2}\norm{\psi}_{C_c^r},
\end{equation*}

\noindent which finishes the proof of the lemma.
\end{proof}

\noindent It is now immediate to show the last bound we need:

\begin{lemma}\label{lem:technical2}
Let $(F_x)$ be as in Theorem \ref{stochastic_reconstruction}. Then it holds that for any $x_0$ in some compact $K\subset\R^d$, $1\le\lambda\le 2^{-n}$ and $\psi\in C_c^r$ with support in $[-\tilde R,\tilde R]^d$ for some $\tilde R>0$, we have:

\begin{equation}\label{eq:techIneq5}
\norm{(f^{(n+1)}-f^{(n)})(P_n\psi_{x_0}^\lambda)}_{L_p} \lesssim 2^{-n\gamma}\lambda^{-\frac E2}\norm{\psi}_{C_c^r}.
\end{equation}

\noindent There the constant in $\lesssim$ is allowed to depend on $K$ and $\tilde R$.
\end{lemma}

\begin{proof}
Analogously to \eqref{ineq2}, one shows that

\begin{equation} \label{ineq1}
\norm{(f^{(n+1)}-f^{(n)})(\phi_x^n)}_{L_p} \lesssim 2^{-n\gamma-n\frac{\abs S-E}{2}}\,.
\end{equation}

\noindent Observe that $(f^{(n+1)}-f^{(n)})(\phi_x^n)$ is $\mathcal F^{(i)}_{\hat S_i(x)}$-measurable for $\hat S_i(x) := \pi_i(x) +(2^{-ns_i}+2^{-(n+1)s_i})R$. This together with \eqref{ineq1} and \eqref{ineq2} shows that $h = f^{(n+1)}-f^{(n)}$ fulfills the conditions of Lemma \ref{lem:stochTechnique}, which immediately gives us \eqref{eq:techIneq5}.

\end{proof}

\noindent We proceed with the proof of properties 1 and 2. To do so, we first make the additional assumption that $F_x$ is a random distribution with enough regularity to invoke the second part of Lemma \ref{lem:PnPsiConverges}. Without this assumption, the reconstructing sequence \eqref{eq:reconstructingSequence} does not converge, but it is still possible to construct a distribution $f$ such that property 2 holds. This construction will be shown in Lemma \ref{lem:CbetaUnnessecary}.

\begin{lemma}\label{lem:prop1+2}
Let $(F_x)$ be as in Theorem \ref{stochastic_reconstruction} and assume that for all $x\in\R^d$, $F_x\in C^\beta(L_p)$ with $\beta+\tilde r>0$. Then for each $\psi\in C_c^r$, $f^{(n)}(\psi)$ converges against a random variable $f(\psi)$ in $L_p(\Omega)$ and $f:C_c^r\to L_p(\Omega)$ is a random distribution with stochastic dimension $e$ and fulfills \eqref{unique1} and \eqref{unique2}. 
\end{lemma}

\begin{proof}
The idea of the proof is to show that $(f^{(l)}-P_l F_x)(\psi_x^\lambda)$ converges to some random variable in $L_p$ as $l\to\infty$ and that its $L_p$-norm is bound by $\lambda^{\gamma-\frac E2}$, while the $L_p$- norm of its conditional expectation is bound by $\lambda^\gamma$. Since $P_lF_x\to F_x$ is convergent thanks to the additional assumption and Lemma \ref{lem:PnPsiConverges}, it follows that $f^{(l)}$ converges to some $f$, and that $(f-F_x)(\psi_x^\lambda) = \lim_{l\to\infty} (f^{(l)}-P_l F_x)(\psi_x^\lambda)$ fulfills \eqref{unique1} and \eqref{unique2}.

Let us fix a test function $\psi$ with support in $[\tilde C,\infty)^d$ for some $\tilde C>0$, an $x\in K$ as well as a $\lambda>0$. We choose $n$ in such a way, that $B2^{-n}\le\lambda\le B2^{-n+1}$ for $B := \min_{i=1,\dots,d}\left(\frac {\tilde C}{2R}\right)^{-\frac 1{s_i}}$. For $l\ge n$, we decompose $f^{(l)} = f^{(n)}+\sum_{m=n}^{l-1}(f^{(m+1)}-f^{(m)})$ as well as $P_l\psi = P_n\psi +\sum_{m=n}^{l-1} \hat P_m\psi$ to get the decomposition

\begin{align}
f^{(l)}(\psi_x^\lambda)-F_x(P_l\psi_x^\lambda) &= \underbrace{(f^{(n)}-P_nF_x)(\psi_x^\lambda)}_{(I)}+\sum_{m=n}^{l-1}\underbrace{\hat P_m(f^{(m+1)}-f^{(m)}-F_x)(\psi_x^\lambda)}_{(II)_m}\nonumber\\ &\qquad+\sum_{m=n}^{l-1}\underbrace{P_m(f^{(m+1}-f^{(m)})(\psi_x^\lambda)}_{(III)_m}\,,\label{eq:decomp1}
\end{align}

\noindent where we used $f^{(l)}(\psi_x^\lambda) = f^{(l)}(P_l \psi_x^\lambda)$, which follows from $f^{(l)}\in V_l$ \eqref{Vl}. $(II)_m$ and $(III)_m$ are bound by Lemma \ref{lem:technical1} and \ref{lem:technical2}, so we only need to find a bound for $(I)$. Using \eqref{LemIneq1}, \eqref{LemIneq2} as well as stochastic coherence, the following holds:

\begin{align*}
\norm{(I)}_{L_p} &\le \sum_{y\in\Delta_n}\norm{(F_y-F_x)(\phi_y^n)}_{L_p}\abs{\scalar{\phi_y^n,\psi_x^\lambda}}\nonumber\\
&\lesssim \lambda^{\gamma-\frac E2}\norm{\psi}_{C_c^r}\,.
\end{align*}

\noindent Note that the positive support of $\psi$ together with our choice of $n$ implies $\pi_i(y) \ge \pi_i(x)+\lambda^{s_i}\tilde C -2^{-ns_i}R \ge \pi_i(x)$ by our choice of $n$, which justifies our use of stochastic coherence. Analogously

\begin{equation*}
\norm{E^{\mathcal F_{\pi_i(x)}^{(i)}}(I)}_{L_p}\lesssim \lambda^\gamma\norm{\psi}_{C_c^r}\,.
\end{equation*}

\noindent We now have all inequalities we need and want to argue that $(f^{(n)}-P_nF_x)(\psi)$ is a Cauchy sequence. By Lemma \ref{lem:technical1} and \ref{lem:technical2} it holds that for any $n\le k\le l$:

\begin{align}
\sum_{m=k}^{l-1}\norm{(II)_m}_{L_p}+\sum_{m=k}^{l-1}\norm{(III)_{m}}_{L_p}&\lesssim \sum_{m=k}^{l-1}(2^{-m(\alpha+\tilde r)}\lambda^{\gamma-\frac E2-\alpha-\tilde r}+2^{-m\gamma}\lambda^{-\frac E2})\norm{\psi}_{C_c^r}\nonumber\\
&\lesssim (2^{-k(\alpha+\tilde r)}\lambda^{\gamma-\frac E2-\alpha-\tilde r}+2^{-k\gamma}\lambda^{-\frac E2})\norm{\psi}_{C_c^r}\,,\label{eq:calc2}
\end{align}

\noindent where we used $\alpha+\tilde r>0$. For an arbitrary $\tilde\psi\in C_c^r$, we choose $\lambda = 1$ and $x\in\R^d$ such that $\tilde\psi = \psi_x^\lambda$ for some $\psi$ with support in $[\tilde C,\infty)^d$. \eqref{eq:calc2} than implies that $f^{(n)}(\tilde\psi)-F_x(P_n\tilde\psi)$ is a Cauchy-sequence, and thus converges in $L_p(\Omega)$. Since $F_x(P_n\tilde\psi)\rightarrow F_x(\tilde\psi)$ in $L_p(\Omega)$, it follows that $f^{(n)}(\tilde\psi)$ converges to a limit $f(\tilde\psi)$. To show that $f$ is a random distribution, note that by construction, $f$ is linear in $\tilde \psi$. Furthermore, the above inequalities show that for $\lambda = 1\approx 2^{-n}$:

\begin{equation*}
\norm{f(\psi)-F_x(\psi)}_{L_p} = \norm{(I)}_{L_p} +\sum_{m=n}^\infty\norm{(II)_m}_{L_p}+\sum_{m=n}^\infty\norm{(III)_m}_{L_p}\lesssim\norm{\psi}_{C_c^r}. 
\end{equation*}

\noindent Since $F_x$ is already a random distribution, it follows that $\norm{f(\psi)}_{L_p}\lesssim\norm{\psi}_{C_c^r}$, which shows that $f$ is a random distribution.

To show that $f$ has stochastic dimension $e$, simply observe that $f^{(n)}(\psi) = \sum_{x} F_x(\phi_x^n)\scalar{\phi_x^n,\psi}$ is $\mathcal F^{(i)}_{y+2^{-ns_i}R}$ measurable for any $\psi$ with support in $\R^{i-1}\times(-\infty,y]\times\R^{d-i}$. The completeness and right continuity of $\mathcal F^{(i)}$ now implies, that the limit $f(\psi)$ is $\mathcal F^{(i)}_y$-measurable for $i=1,\dots,e$.

It remains to show \eqref{unique1} and \eqref{unique2}. We use \eqref{eq:calc2} and $2^{-n}\approx\lambda$ to get for $l\to\infty$

\begin{align*}
\norm{(f-F_x)(\psi_x^\lambda)}_{L_p} &\lesssim \left(\lambda^{\gamma-\frac E2} +\sum_{m=n}^\infty 2^{-m(\alpha+\tilde r)}\lambda^{\gamma-\frac E2-\alpha\tilde r}+\sum_{m=n}^\infty 2^{-m\gamma}\lambda^{-\frac E2}\right)\norm{\psi}_{C_c^r}\\
&\lesssim \lambda^{\gamma-\frac E2}\norm{\psi}_{C_c^r}\,.
\end{align*}

\noindent
Analogously, it follows that

\begin{equation*}
\norm{E^{\mathcal F_{\pi_i(x)}^{(i)}}(f-F_x)(\psi_x^\lambda)}_{L_p}\lesssim\lambda^{\gamma}\norm{\psi}_{C_c^r}\,.
\end{equation*}

\noindent This finishes the proof.
\end{proof}

%\begin{remark}
%A surprising feature of this proof is that $\beta+\tilde r>0$ is not necessary to show the convergence of $f^{(n)}-F_x(P_n\cdot)$ to some random distribution. However, the condition is necessary to get $F_x(P_n\cdot)\to F_x$ by Lemma \ref{lem:PnPsiConverges} and thus to get the convergence of $f^{(n)}$. So as far as we can tell, the condition is necessary, if one wants to utilize wavelet techniques for the reconstruction. \cite{Zorin-Kranich} showed in the deterministic case that one can get rid of this condition by using the reconstructing sequence \eqref{eq:ZorinKranichSequence}. We conjecture that this extends to the stochastic case, but we do not have proof of this statement at the moment.
%\end{remark}

\noindent While the above proof needs $F_x\in C^\beta(L_p)$ for sufficiently large $\beta$ to get $f^{(n)}\to f$ for $n\to\infty$, a closer look at the proof reveals that the property is not needed to get $f^{(n)}-P_nF_x$ to converge to some random distribution $G_x$, which should formally be $f-F_x$. So we can get rid of the assumption $F_x\in C^\beta(L_p)$ by setting

\begin{equation*}
f(\psi) := \lim_{n\to\infty} (f^{(n)}-P_n F_x)(\psi) +F_x(\psi)\,.
\end{equation*}

\noindent As it turns out, this gives us the reconstruction if we do not have $F_x\in C^\beta(L_p)$, as the next lemma shows.

\begin{lemma}\label{lem:CbetaUnnessecary}
Under the assumptions of Theorem \ref{stochastic_reconstruction}, there is a random distribution $f$ on $C_c^r$ with stochastic dimension $e$, such that \eqref{unique1} and \eqref{unique2} hold.
\end{lemma}

\begin{proof}
By Remark \ref{rem:allDistAreCalpha}, we get that for all $x\in\R^d$, $F_x\in C^{-\abs S-r}(L_p)$. So if we restrict ourselves to the case $\psi\in C^q_c$, $q+\abs S-r>0$, we get that $F_x$ fulfills the conditions of Lemma \ref{lem:prop1+2} and can thus be reconstructed to get a random distribution $f(\psi)$ for $\psi\in C_c^q$. As stated above, the proof of Lemma \ref{lem:prop1+2} shows that for any $\psi\in C_c^r$, $(f^{(l)}-P_lF_x)(\psi)$ converges against some random $C_c^r$ distribution $\psi\mapsto G_x(\psi)$ as $l\to\infty$. 

Therefore, we set $f_x:= G_x+F_x$ and show that $f_x =: f$ is independent of $x\in\R^d$ and thus a well-defined random $C_c^r$ distribution. This is a slight abuse of notation, as we defined $f(\psi)$ twice for $\psi\in C_c^q$. However, by the proof of Lemma \ref{lem:prop1+2}, it is clear that for all $\psi\in C_c^q$, $f(\psi) = \lim_{l\to\infty} (f^{(l)}-P_l F_x)(\psi) + F_x(\psi) = (G_x+F_x)(\psi) = f_x(\psi)$, justifying the notation.

Note that this would imply that $(f-F_x)(\psi) = G_x(\psi) = \lim_{l\to\infty} (f^{(l)}-P_lF_x)(\psi)$ in $L_p$, which means that $f$ has stochastic dimension $e$ and fulfills \eqref{unique1} and \eqref{unique2}.

So it only remains to show that $f_x(\psi)$ is independent of $x$. Let $x,y\in\R^d$. Then $f_x(\psi) = f_y(\psi)$ a.s. for all $\psi\in C_c^q$, since $f_x(\psi) = f_y(\psi) = \lim_{n\to\infty} f^{(n)}(\psi)$ in $L_p$ holds on $C_c^q$. Since $f_x, f_y :C_c^r\to L_p(\Omega)$ are continuous and $C_c^q\subset C_c^r$ is dense, it follows that $f_x = f_y$ a.s.
\end{proof}

\noindent
It remains to prove the uniqueness of the reconstructed germ. We do so in the next lemma, which finishes the proof of Theorem \ref{stochastic_reconstruction}:

\begin{lemma}\label{lem:prop3}
There is at most one random distribution $f$ (up to modification) with stochastic dimension $e$, which fulfills \eqref{unique1} and \eqref{unique2}.
\end{lemma}

\begin{proof}
Let $f,g$ be two random distributions as in the above lemma, and consider $h=f-g$. By Lemma \ref{lem:PnPsiConverges}, we have $h(\psi) = \lim_{n\to\infty} h(P_n\psi)$ in $L_p$ for $\psi\in C^{q}_c$ for a sufficiently large $q$. Since $C^{q}_c$ is dense in $C_c^r$, it suffices to show that $h(P_n\psi)$ converges to $0$ in $L_p$ for $\psi\in C_c^{q}$. For a test function $\psi$ with support in $[\tilde C,\tilde R]^d$ for some $0<\tilde C<\tilde R$, we can apply \eqref{unique1} to get

\begin{align*}
\norm{h(\psi_x^\lambda)}_{L_p}&\le\norm{(f-F_x)(\psi_x^\lambda)}_{L_p}+\norm{(g-F_x)(\psi_x^\lambda)}_{L_p}\\&\lesssim \lambda^{\gamma-\frac E2}\,,
\end{align*}

\noindent and anlaogously, using \eqref{unique2} and for $i=1\dots,e$:

\begin{equation*}
\norm{E^{\mathcal F^{(i)}_{\pi_i(x)}}h(\psi_x^\lambda)}_{L_p}\lesssim\lambda^\gamma\,.
\end{equation*}

If we substitute $\psi_x^\lambda$ with $\phi_x^n$ and note that $h$ has stochastic dimension $e$, we see that $h$ fulfills the conditions of Lemma \ref{lem:stochTechnique}. It follows that $\norm{h(P_n\psi)}_{L_p}\lesssim 2^{-n\gamma}$ goes to $0$, which finishes the proof.
\end{proof}

\subsection{Proof of Theorem \ref{stochastic_recon_covariance}}

\noindent
The strategy for the prove of Theorem \ref{stochastic_recon_covariance} is the same as for Theorem \ref{stochastic_reconstruction}: For $l\ge n$, we use the decomposition \eqref{eq:decomp1} to show that $(f^{(n)}-P_nF_x)(\psi)$ is a Cauchy sequence and thus convergent in $L_2(\Omega)$. By choosing a smart $n$, this decomposition will also be used to show the bounds \eqref{uniqueCov1} and \eqref{uniqueCov2}.

The main difference in the proofs is the uniqueness part: For Theorem \ref{stochastic_reconstruction}, we took two distributions $f,g$ fulfilling \eqref{unique1} and \eqref{unique2} and analyzed $h=f-g$. However, if one follows this approach in the covariance setting, the term $h(\psi_x^\lambda)h(\psi_y^\lambda)$ contains mixed expressions $E\left((f-F_x)(\psi_x)(g-F_y)(\psi_y)\right)$, which are hard to bound from above. So we use a different tactic to show uniqueness for the covariance setting:

\begin{lemma}\label{lem:uniqueCov}
Under the assumptions of Theorem \ref{stochastic_recon_covariance}, there is at most one (up to modification) random distribution $f$, such that \eqref{uniqueCov1} and \eqref{uniqueCov2} hold.
\end{lemma}

\begin{proof}
We will show that any $f$ fulfilling \eqref{uniqueCov1} and \eqref{uniqueCov2} needs to fulfill

\begin{equation*}
f(\psi) = \lim_{n\to\infty} f^{(n)}(\psi)\,
\end{equation*}

\noindent for $\psi\in C_c^q$ for high enough $q$, and $f^{(n)}$ as in \eqref{eq:reconstructingSequence}, where the right-hand-side limit is in $L_2(\Omega)$. Thus, $f(\psi)$ is uniquely defined for $\psi\in C_c^q$ and since $C_c^q\subset C_c^r$ is dense, this shows the claim.

So assume $f$ fulfills \eqref{uniqueCov1} and \eqref{uniqueCov2}. Note that $f(\psi) = \lim_{n\to\infty}f(P_n\psi)$ in $L_2$ holds for $\psi\in C_c^q$ by Lemma \ref{lem:PnPsiConverges} for high enough $q$. Using \eqref{uniqueCov1}, \eqref{uniqueCov2} and counting the non-zero terms, we can formally calculate

\begin{align}
\norm{f(P_n\psi)-f^{(n)}(\psi)}_{L_2}^2 &= \norm{\sum_{x\in\Delta_n}(f-F_x)(\phi_x^n)\scalar{\phi_x^n,\psi}}_{L_2}^2 \nonumber\\
&\le \sum_{x\in\Delta_n^{1,e}}\left(\sum_{y\in\Delta_n^{e+1,d}}\norm{(f-F_{(x,y)})(\phi_{(x,y)}^n)\scalar{\phi_{(x,y)}^n,\psi}}_{L_2}\right)^2\nonumber\\
&\qquad +\sum_{x_1\neq x_2\in\Delta_n^{1,e}}\sum_{y_1,y_2\in\Delta_n^{e+1,d}} E\left(\prod_{i=1}^2 (f-F_{(x_i,y_i)})(\phi_{(x_i,y_i)}^n)\scalar{\phi_{(x_i,y_i)}^n,\psi}\right) \nonumber\\
&\lesssim 2^{-2n\gamma}\xrightarrow{n\to\infty} 0\,,\label{ineq:formalCalcCov}
\end{align}

\noindent where the calculation is formal at the moment since we applied \eqref{uniqueCov2} without checking that $(f-F_{(x_i,y_i)})(\phi^n_{(x_i,y_i)})$ have disjoint effective stochastic supports for $i=1,2$. Looking at said effective stochastic support, we see that it is given by $[\pi_1(x_i),\pi_1(x_i)+2\cdot 2^{-ns_1}R]\times\dots\times[\pi_e(x_i),\pi_e(x_i)+2\cdot 2^{-ns_e}R]$, so the mesh size of $\Delta_n^{1,e}$ of $2^{-ns_i}$ for $i=1\dots,e$ is not big enough to ensure that the effective stochastic supports are disjoint for all $x_1\neq x_2\in\Delta_n^{1,e}$. 

However, comparing the mesh size $2^{-n s_i}$ with the size of the effective stochastic support $2\cdot 2^{-ns_i} R$ in the direction $i=1,\dots,e$, we see that for each $x_1\in\Delta_n^{1,e}$, only a constant number of $x_2\in \Delta_n^{1,e}$ will cause an overlapping effective stochastic support. So it suffices to introduce finitely many offsets $r^{(0)}_n$ and split the mesh $\Delta_n^{1,e}$ into finitely many meshes $\Delta_n^{1,e}(r_n^{(0)})$. \eqref{ineq:formalCalcCov} holds true if we replace $\Delta_n^{1,e}$ with $\Delta_n^{1,e}(r_n^{(0)})$ and summing over all $r_n^{(0)}$ only adds a constant factor to the $\lesssim$. It follows that $\norm{f(P_n\psi)-f^{(n)}(\psi)}$ vanishes for $\psi\in C_c^q$, which shows the claim.
\end{proof}

\noindent
With this, we can tackle the proof of Theorem \ref{stochastic_recon_covariance}. As the proof is very similar to the proof of Theorem \ref{stochastic_reconstruction}, we will skip over some of the calculations and focus on highlighting the broader arguments.

\begin{proof}[Proof of Theorem \ref{stochastic_recon_covariance}]
We only show the proof under the additional assumption that for any $x\in\R^d$, $F_x\in C^{\beta}(L_2)$ for a $\beta\in\R$ such that $\beta+\tilde r>0$. The general case can be recovered from this case just as before. 

For $n\le l$, recall our decomposition \eqref{eq:decomp1} for a test function $\psi$ with support in $[\tilde C, \tilde R]^d$ for some $0<\tilde C<\tilde R$, $0<\lambda\le 1$ and $n$ such that $B2^{-n}\le \lambda\le B2^{-n+1}$ for $B=\min_{i=1,\dots,d}\left(\frac{\tilde C}{2R}\right)^{-\frac 1{s_i}}$:

\begin{align*}
f^{(l)}(\psi_x^\lambda)-F_x(P_l\psi_x^\lambda) &= \underbrace{(f^{(n)}-P_nF_x)(\psi_x^\lambda)}_{(I)_n}+\sum_{m=n}^{l-1}\underbrace{\hat P_m(f^{(m+1)}-f^{(m)}-F_x)(\psi_x^\lambda)}_{(II)_m}\\&\qquad+\sum_{m=n}^{l-1}\underbrace{P_m(f^{(m+1)}-f^{(m)})(\psi_x^\lambda)}_{(III)_m}\,.
\end{align*}

\noindent
The bound $\norm{(II)_m}_{L_2}\lesssim 2^{-m(\alpha+\tilde r)}\lambda^{\gamma-\frac E2-\alpha-\tilde r}\norm{\psi}_{C_c^r}$ holds by the same argument as in Theorem \ref{stochastic_reconstruction}, so we only need to bound $\norm{(III)_m}_{L_p}$ to show convergence. Note that \eqref{ineq1} still holds, and one can use \eqref{coherenceCov2} to replace \eqref{ineq2} with

\begin{equation}
\abs{E\left(\prod_{i=1}^2(f^{(m+1)}-f^{(m)})(\phi_{x_i}^m)\right)} \lesssim 2^{-2m\gamma-m\abs S}\,,\label{ineqCov2}
\end{equation}

\noindent for $\abs{x_1^{(e)}-x_2^{(e)}} \ge 2eR2^{-m}$ (where $x_i^{(e)} = (\pi_1(x_i),\dots,\pi_e(x_i))$), since the stochastic effective support of all terms in the sum $(f^{(m+1)}-f^{(m)})(\phi_{x_i}^m) = \sum_{k\in\Delta^{n+1}} a_k^{(n+1)}(F_{x+k}-F_x)(\phi_{k+x}^{(n+1)})$ (see \eqref{eq:(III)Sum}) are contained in $[\pi_1(x_i),\pi_1(x_i)+2R2^{-ms_1}]\times\dots\times[\pi_e(x_i),\pi_e(x_i)+2R2^{-ms_e}]$. Recall that $(III)_m$ is of the form \[(III)_m = \sum_{x\in\Delta_m^{1,e},y\in\Delta_m^{e+1,d}}g_{(x,y)}\,\] where $g_x := (f^{m+1}-f^{(m)})(\phi_x^m)\scalar{\phi_x^m,\psi_{x_0}^\lambda}$. Just as in the proof of Lemma \ref{lem:uniqueCov}, we do not have a satisfying upper bound for $\abs{E(g_{(x_1,y_1)}g_{(x_2,y_2)})}$ with $x_1\neq x_2\in\Delta_m^{1,e}, y_1,y_2\in\Delta_m^{e+1,d}$, unless $\abs{x_1^{(e)}-x_2^{(e)}}$ is big enough to ensure disjoint effective stochastic supports of the involved terms. This can be achieved by splitting the net $\Delta_n$ into different nets $\Delta_n(r_n^{(0)})$, as before. This will only add a constant factor to our calculations, so we suppress this notation and write with a slight abuse of notation:

\begin{align*}
\norm{\sum_{x\in\Delta_m^{1,e}, y\in\Delta_m^{e+1,d}} g_{(x,y)}}_{L_2}^2&\le \sum_{x\in\Delta_m^{1,e}}\left(\sum_{y\in\Delta_m^{e+1,d}}\norm{g_{(x,y)}}_{L_2}\right)^2+ \sum_{x_1\neq x_2 \in \Delta_m^{1,e}}\sum_{y_1,y_2\in\Delta_m^{e+1,d}}E(g_{(x_1,y_1)} g_{(x_2,y_2)}) \\
&\lesssim 2^{-2m\gamma}(\lambda^{-E}+1)\norm{\psi}_{C_c^r}^2\,,
\end{align*}

\noindent where one uses the same bounds \eqref{ineq1}, \eqref{ineqCov2} and \eqref{LemIneq1} while counting the order of non-zero terms. This shows that

\begin{equation*}
\norm{(III)_m}_{L_2}\lesssim 2^{-m\gamma}\lambda^{-\frac E2}\norm{\psi}_{C_c^r}\,,
\end{equation*}

\noindent
which is sufficient to show that $f^{(n)}(\psi)$ is cauchy and thus convergent. As before, the limit $f$ is a random distribution.

It remains to show \eqref{uniqueCov1} and \eqref{uniqueCov2}. Note that \eqref{uniqueCov1} follows exactly as \eqref{unique1} in Theorem \ref{stochastic_reconstruction}, so we only need to focus on \eqref{uniqueCov2}. To show this, we use again the decomposition \eqref{eq:decomp1} and write $(I)_n(\psi_{x}^\lambda) = (f^{(n)}-P_n F_{x})(\psi_{x}^\lambda)$. Before we calculate $\abs{E((I)_n(\psi_{x_1}^\lambda) (I)_n(\psi_{x_2}^\lambda))}$, we need to check that there is no problem with the effective stochastic supports of any involved term: It holds that

\begin{equation}\label{eq:(I)}
(I)_n(\psi_{x}^\lambda) = \sum_{y\in\Delta_n}(F_y-F_x)(\phi_y^n)\scalar{\phi_y^n,\psi_{x}^\lambda}\,,
\end{equation}

\noindent which has for all non-zero terms and for $j=1,\dots, e$: $\pi_j(y) \ge \pi_j(x)$ and $\pi_j(y)\le \pi_j(x)+\tilde R\lambda^{s_j}$ by our choice of $n$. Thus the effective stochastic support of $(F_x-F_y)(\phi_y^n)$ is a subset of $\prod_{j=1}^e[\pi_j(x),\pi_j(x)+\tilde R\lambda^{s_j}+2^{-ns_j}R]$. Using again our choice of $n$, we get that $2^{-ns_j}R\lesssim \lambda^{s_i} \tilde C/2 \le \lambda^{s_i}\tilde R$ and thus $\pi_j(x)+\tilde R\lambda^{s_j}+2^{-n s_j}R \le \pi_j(x)+2\tilde R\lambda^{s_i}$. Thus, the stochastic effective support of all terms in the sum $I_n(\psi_{x}^\lambda)$ are contained in the one of $(f-F_{x})(\psi_{x}^\lambda)$. So as long as $(f-F_{x_i})(\psi_{x_i}^\lambda)$ has disjoined stochastic effective supports for $i=1,2$, we can apply the coherence property \eqref{coherenceCov2} to all terms appearing in $(I)_n(\psi_{x_1}^\lambda)(I_n)(\psi_{x_2}^\lambda)$. Similarly, it holds that the stochastic effective supports of terms in $(II)_m(\psi_x^\lambda)$ and $(III)_m(\psi_x^\lambda)$ (where we define $(II)_m(\psi_x^{\lambda}), (III)_m(\psi_x^\lambda)$ analogously to $(I)_n(\psi_x^\lambda)$) are contained in the one of $(f-F_x)(\psi_x^\lambda)$, so there is never a problem applying \eqref{coherenceCov2}.

So let $x_1,x_2$ be chosen in such a way, that $(f-F_{x_i})(\psi_{x_i}^\lambda)$ has disjoined stochastic support. We use \eqref{eq:(I)} to calculate

\begin{align*}
\abs{E((I)_n(\psi_{x_1}^\lambda) (I)_n(\psi_{x_2}^\lambda))} &\le \sum_{y_1,y_2\in\Delta_n}\abs{E\left(\prod_{i=1}^2 (F_{y_i}-F_{x_i})(\phi_{y_i}^n)\scalar{\phi_{y_i}^n,\psi_{x_i}^\lambda}\right)} \\
&\lesssim 2^{-2n\gamma-n\abs{S}}\sum_{y_1,y_2\in\Delta_n}\abs{\prod_{i=1}^2 \scalar{\phi_{y_i}^n,\psi_{x_i}^\lambda}} \\
&\lesssim 2^{-2n\gamma}\approx \lambda^{2\gamma}\,,
\end{align*} 

\noindent since the number of non-zero terms is of order $2^{2n\abs{S}}\lambda^{2\abs S}$. More generally, by using the identities

\begin{align}
(II)_m(\psi_x^\lambda) &= \sum_{y\in\Delta_m}\sum_{z\in\Delta_{m+1}}\sum_{\hat\phi\in\Phi}(F_z-F_x)(\phi_z^{m+1})\scalar{\phi_z^{m+1},\hat\phi_y^m}\scalar{\hat\phi_y^m,\psi_x^\lambda} \label{eq:(II)}\\
(III)_m(\psi_x^\lambda) &= \sum_{y\in\Delta_m}(f^{(m+1)}-f^{(m)})(\phi_y^m)\scalar{\phi_y^m,\psi_x^\lambda} \nonumber\\
&= \sum_{y\in\Delta_m}\sum_{k\in\Delta^{m+1}} a_k^{(m+1)}(F_{y+k}-F_y)(\phi_{y+k}^{(m+1)})\scalar{\phi_y^m,\psi_x^\lambda}\,\label{eq:(III)}
\end{align}

\noindent
we see that $(I)_n(\psi_x^\lambda)$ interacts with the respective other terms by adding a factor

\begin{align}\label{ineq:calc1}
\abs{E((I)_n(\psi_{x_1}^\lambda)(\bigstar_1))} &= \abs{E\left(\left(\sum_{y_1\in\Delta_n}(F_{y_1}-F_{x_1})(\phi_{y_1}^n)\scalar{\phi_{y_1}^n,\psi_{x_1}^\lambda}\right)\cdot(\bigstar_2)\right)}\\
&\lesssim \lambda^\gamma(\bigstar_3)\,,\nonumber
\end{align}

\noindent where $(\bigstar_1)$ is given by $(I)_n(\psi_{x_2}^\lambda)$ $(II)_m(\psi_{x_2}^\lambda)$ or $(III)_m(\psi_{x_2}^\lambda)$, $(\bigstar_2)$ is given by \eqref{eq:(I)}, \eqref{eq:(II)} or \eqref{eq:(III)}, respectively and $(\bigstar_3)$ is given by $\lambda^{a}2^{-mb}$ for certain rates $a,b$ depending on $(\bigstar_1)$. A similar analysis leads to the bounds

\begin{align}\label{ineq:calc2}
\abs{E((II)_m(\psi_{x_1}^\lambda)(\bigstar_1))} &= \abs{E\left(\left(\sum_{y_1\in\Delta_m}\sum_{z_1\in\Delta_{m+1}}\sum_{\hat\phi\in\Phi}(F_{z_1}-F_{x_1})(\phi_{z_1}^{m+1})\scalar{\phi_{z_1}^{m+1},\hat\phi_{y_1}^m}\scalar{\hat\phi_{y_1}^m,\psi_{x_1}^\lambda}\right)\cdot(\bigstar_2)\right)} \\
&\lesssim 2^{-m(\alpha+\tilde r)}\lambda^{\gamma-\alpha-\tilde r}(\bigstar_3)\nonumber
\end{align}

\noindent as well as

\begin{align}\label{ineq:calc3}
\abs{E((III)_m(\psi_{x_1}^\lambda)(\bigstar_1))} &= \abs{E\left(\left(\sum_{y_1\in\Delta_m}\sum_{k_1\in\Delta^{m+1}} a_{k_1}^{(m+1)}(F_{y_1+k_1}-F_{y_1})(\phi_{y_1+k_1}^{(m+1)})\scalar{\phi_{y_1}^m,\psi_{x_1}^\lambda}\right)\cdot(\bigstar_2)\right)}\\
&\lesssim 2^{-m\gamma}(\bigstar_3)\,.\nonumber
\end{align}

\noindent Combining these gives us the bounds

\begin{align*}
\abs{E((II)_{l}(\psi_{x_1}^\lambda)(II)_k(\psi_{x_2}^\lambda))} &\lesssim 2^{-(k+l)(\alpha+\tilde r)}\lambda^{2\gamma-2(\alpha+\tilde r)}\\
\abs{E((III)_k(\psi_{x_1}^\lambda)(III)_l(\psi_{x_2}^\lambda))} &\lesssim 2^{-(k+l)\gamma}\\
\abs{E((I)_n(\psi_{x_1}^\lambda)(II)_m(\psi_{x_2}^\lambda))} &\lesssim 2^{-m(\alpha+\tilde r)}\lambda^{2\gamma-(\alpha+\tilde r)}\\
\abs{E((I)_n(\psi_{x_1}^\lambda)(III)_m(\psi_{x_2}^\lambda))} &\lesssim 2^{-m\gamma}\lambda^{\gamma}\\
\abs{E((II)_k(\psi_{x_1}^\lambda)(III)_l(\psi_{x_2}^\lambda))} &\lesssim 2^{-k(\alpha+\tilde r)-l\gamma}\lambda^{\gamma-(\alpha+\tilde r)}\,.
\end{align*}

\noindent Note that $\gamma>0, \alpha+\tilde r>0$ implies that all of the above terms are summable in $k,l,m$, respectively. Summing those expressions over their respective indices in the decomposition \eqref{eq:decomp1} and using $2^{-n}\cong\lambda$ shows

\begin{align*}
\abs{E\left(\prod_{i=1}^2 (f-F_{x_i})(\psi_{x_i}^\lambda)\right)}&\lesssim\lambda^{2\gamma}\,,
\end{align*}

\noindent which finishes the proof.
\end{proof} 

\begin{remark}
A surprising detail of this proof is that the existence part only makes use of \eqref{coherenceCov2} for $\psi_1=\psi_2$ and $\epsilon_1=\epsilon_2$, while the general case was used to show \eqref{uniqueCov2}. This implies that the reconstruction is possible under the weaker condition that \eqref{coherenceCov2} holds only for $\psi_1=\psi_2$ and $\epsilon_1=\epsilon_2$, but in this case, we do not know how to uniquely characterize the limit. 
\end{remark}

\noindent Last but not least, let us proof Corollary \ref{cor:uniqueCovMult}

\begin{proof}[Proof of Corollary \ref{cor:uniqueCovMult}]
This proof is essentially identical to the proof of \eqref{uniqueCov2} in Theorem \ref{stochastic_recon_covariance}. The only difference is that we use two natural numbers $n_1,n_2$ such that $B 2^{-n_i}\le\lambda_i\le B 2^{-n_i+1}$ for $i=1,2$, and decompose $(f-F_x)((\psi_i)_{x_i}^{\lambda_i})$ with \eqref{eq:decomp1} with the respective $n_i$. One observes that the calculations \eqref{ineq:calc1}, \eqref{ineq:calc2} and \eqref{ineq:calc3} still hold if we add the corresponding indexes $i$ everywhere. This directly gives us 

\begin{equation*}
\abs{E\left(\prod_{i=1}^2 (f-F_{x_i})((\psi_i)_{x_i}^{\lambda_i})\right)}\lesssim \lambda_1^\gamma\lambda_2^\gamma\,.
\end{equation*}
\end{proof}

\section{Stochastic reconstruction is stochastic sewing in 1 dimension}\label{sectionSewing}

As we discussed in the introduction, sewing and reconstruction are closely related to one another. It is well known \cite{BrouxZambotti},\cite{Zorin-Kranich}, that for a two-parameter process $(A(s,t))_{s,t\in[0,T]}$ which fulfills the conditions of the classical sewing lemma, the germ given by

\begin{equation*}
F_s(t) = \frac{\partial}{\partial t}A(s,t)
\end{equation*}

\noindent fulfills the conditions of the reconstruction theorem (as formulated in \cite{caravenna}). Here, $\frac{\partial}{\partial t}$ should be seen as a distributional derivative, i.e.

\begin{equation*}
F_s(\psi) = -\int_{-\infty}^{\infty} A(s,t)\psi'(t) dt
\end{equation*}

\noindent for any $\psi\in C_c^r$. 

\begin{remark}
We should note that the classical sewing lemma defines $A(s,t)$ only on $\Delta := \{(s,t)\in[0,T]^2~\vert~ s\le t\}$ and assumes $A(s,s) = 0$ for all $s$. In this case, we can extend $A$ anti-symmetrically by setting $A(t,s) := -A(s,t)$ for all $s\le t$.
\end{remark}

\noindent Using the above germ $(F_s)_{s\in[0,T]}$ and under the assumptions of the sewing lemma, it holds that $(F_s)_{s\in\R}$ can be reconstructed to get a distribution $f$ and that $f = \frac\partial{\partial t}I$ is the distributional derivative of the process $(I(t))_{t\in[0,T]}$ gained from the sewing lemma. 

\noindent Thus, sewing can indeed be seen as the one-dimensional case of reconstruction. As a result of \cite{brault}, Lemma 3.10, we can reverse the time derivative and regain $I(t)$ as ``$f(1_{[0,t]})$'', which is of course only rigorous as the function $z$ from \cite{brault}, Lemma 3.10. This allows one to construct the process $I$ from the reconstruction $f$.

Given all this, it should come as little surprise, that the statement stays true in the stochastic case: indeed, the stochastic reconstruction theorem is nothing else than the distributional version of Khoa Lê's stochastic sewing Lemma in one dimension. We recall said stochastic sewing lemma \cite{le}:

\begin{theorem}[stochastic sewing lemma]
Let $2\le p < \infty$, and let $(A(s,t))_{s,t\in[0,T]}$ be a two parameter process in $L_p(\Omega)$, which is adapted to some complete, right continuous filtration $\mathcal F_t$ in the sense, that $A(s,t)\in\mathcal F_t$ for all $0\le s\le t\le T$. Suppose that there is a $\gamma>0$, such that for all $0\le s\le u\le t\le T$, it holds that

\begin{align*}
\norm{A(s,t)-A(s,u)-A(u,t)}_{L_p}&\lesssim \abs{t-s}^{\frac 12+\gamma} \\
\norm{E^{\mathcal F_s}(A(s,t)-A(s,u)-A(u,t))}_{L_p}&\lesssim \abs{t-s}^{1+\gamma}.
\end{align*}

\noindent Then, there is a unique (up to modifications) process $(I(t))_{t\in[0,T]}$ in $L_p(\Omega)$, which is adapted to $\mathcal F_t$, has $I(0) = 0$ and fulfills for all $0\le s\le t\le T$

\begin{align*}
\norm{I(t)-I(s)-A(s,t)}_{L_p}&\lesssim \abs{t-s}^{\frac 12+\gamma} \\
\norm{E^{\mathcal F_s}(I(t)-I(s)-A(s,t))}_{L_p}&\lesssim \abs{t-s}^{1+\gamma}.
\end{align*}
\end{theorem}

\noindent Note that we can easily extend $A(s,t)$ to be a process over $\R^2$, by setting $A(s,t) = A(0,t)$ for $s\le 0, t\in[0,T]$ and $A(s,t) = A(s,T)$ for $t\ge T,s\in[0,T]$. Should both parameter $s,t\notin[0,T]$, we simply set $A(s,t) = A(0,T)$ if $s\le 0$ and $t\ge T$ and $A(s,t) = 0$, if both are less than $0$ or greater than $T$. This leads to a process $I(t)$, which is the same for $t\in[0,T]$ and constant outside of this interval.

We want to set $F_s(t) = \frac{\partial}{\partial t}A(s,t)$, as above. To do so, we need to assume that almost surely $A(s,\cdot)$ is locally integrable. In practice, most candidates for $A$ are almost surely continuous or at least piecewise continuous, so this additional assumption is reasonable to make. Under this assumption, we define the germ

\begin{equation*}
F_s(\psi) = -\int_{-\infty}^{\infty} A(s,t)\psi'(t) dt
\end{equation*}

\noindent as before. Note that $\mathcal F_t$ fulfills our assumptions of completeness and right-continuity by assumption, and is trivially orthogonal as we only have one filtration. It follows that the germ $(F_s)$ has stochastic dimension $e=d=1$. We can now show the following:

\begin{theorem}
Let $A(s,t)$ fulfill the assumptions of the stochastic sewing lemma, and assume $A(s,\cdot)$ is in $L_1^{loc}(\R)$ almost surely. Then, $(F_s)_{s\in\R}$ is stochastically $\gamma$-coherent with stochastic dimension $1$. Let $I(t)$ be the process we get from the sewing lemma from $A(s,t)$, and let $f$ be the distribution gained from Theorem \ref{stochastic_reconstruction} applied to $(F_x)$. It holds that

\begin{equation*}
f(\psi) = -\int_{-\infty}^\infty I(t)\psi'(t) dt,
\end{equation*}

\noindent for all $\psi\in C_c^r$ a.s.
\end{theorem}

\begin{proof}
Let us check the coherence property: Observe that

\begin{align*}
(F_s-F_u)(\psi) = \int_{-\infty}^\infty (A(s,t)-A(s,u)-A(u,t))\psi'(t) dt,
\end{align*}

\noindent where we used that $A(s,u)\int_{-\infty}^\infty \psi'(t) = 0$ since $\psi$ is only compactly supported. It follows that for all $s\le u$ and $\psi$ with compact, positive support

\begin{align*}
\norm{(F_s-F_u)(\psi_u^\lambda)}_{L_2} &\le \sup_{t\in \supp(\psi_u^\lambda)}\norm{A(s,t)-A(s,u)-A(u,t)}_{L_2}\int_{-\infty}^\infty \abs{(\psi_u^\lambda)'(t)} dt \\&\lesssim (\abs{u-s}+\lambda)^{\frac 12+\gamma}\lambda^{-1},
\end{align*}

\noindent and analogously

\begin{equation*}
\norm{E^{\mathcal F_s}(F_s-F_u)(\psi_u^\lambda)}_{L_2} \lesssim (\abs{u-s}+\lambda)^{1+\gamma}\lambda^{-1}.
\end{equation*}

\noindent Thus, $(F_s)_{s\in\mathbb R}$ is indeed stochastically $\gamma$-coherent with $\alpha = -1$ and there exists a reconstruction $f$. It remains to show that $f=\frac\partial{\partial t} I$. To this end, we calculate

\begin{align*}
\left(F_s-\frac\partial{\partial t} I\right)(\psi) = -\int_{-\infty}^\infty (A(s,t) - (I(t)-I(s)))\psi'(t) dt,
\end{align*}

\noindent where we again used that $I(s) \int\psi'(t) dt = 0$. It follows that for all test functions with strictly positive support

\begin{align*}
\norm{\left(F_s-\frac \partial{\partial t} I\right)(\psi_s^\lambda)}_{L_2} &\le \sup_{t\in \supp(\psi_s^\lambda)}\norm{A(s,t)-(I(t)-I(s))}_{L_2}\int_{-\infty}^\infty \abs{(\psi_s^\lambda)'(t)}dt \\
&\lesssim \lambda^{-\frac 12+\gamma},
\end{align*}

\noindent and analogously,

\begin{equation*}
\norm{E^{\mathcal F_s}\left(F_s-\frac\partial{\partial t} I\right)(\psi_s^\lambda)}_{L_2}\lesssim \lambda^\gamma.
\end{equation*}

\noindent By the uniqueness of the reconstruction, this concludes the proof.
\end{proof}

\section{Gaussian Martingale Measure}\label{sectionGMM}

As an application of stochastic reconstruction, we show that integration against Gaussian martingale measures can be seen as a product of a process in $C^\alpha$ and a distribution in $C^\beta$, similar to the young product between distributions presented in the introduction. The martingale properties of the measure will allow us to do this reconstruction up to $\alpha+\beta>-\frac 12$, in comparison to the classical assumption $\alpha+\beta>0$.

We begin by introducing the notation of martingale measures, which can be found in \cite{davar} or \cite{walshOriginal}.

\subsection{Gaussian martingale measures and Walsh-type Integration}

Loosely speaking, a Gaussian martingale measure is a Gaussian family $(W_t(A))_{t\ge 0, A\in \mathcal A_K}$ for some subset $\mathcal A_K\subset B(\R^d)$, such that

\begin{itemize}
\item $W_t(A)$ is a martingale in $t$, and
\item $W_t(A)$ is a measure in $A$ in the sense that $W_t(\emptyset) = 0$ and for all disjoint sets $(A_n)_{n\in\mathbb N}$, $W_t\left(\bigcup_n A_n\right) = \sum_n W_t(A_n)$ (where both equations hold almost surely).
\end{itemize}

\noindent To get a Gaussian measure (or a Gaussian family in general), we need to clarify its covariance. This is given through the notion of a covariance measure:

\begin{definition}
Let $K$ be a symmetric, positive definite and $\sigma$-finite signed measure on $(\R^d\times\R^d,B(\R^d)\otimes B(\R^d))$. Then, $K$ is called a \emph{covariance measure}, if there is a symmetric, positive definite, and $\sigma$-finite measure $\abs{K}$, such that for all $A,B\in B(\R^d)$:

\begin{equation*}
\abs{K(A\times B)}\le\abs{K}(A\times B).
\end{equation*}
\end{definition}

\noindent We set 

\begin{equation*}
\norm{f}_K^2 := \int_{\R^d\times\R^d} f(x)f(y) K(dx,dy),
\end{equation*}

\noindent and define $\norm{f}_{\abs K}$ analogously. We further set $\mathcal A_K := \{A\in B(\R^d)~\vert~K(A\times A)<\infty\}$. This allows us to formally define the Gaussian martingale measure $W$ as follows:

\begin{definition}
Let $K$ be a covariance measure. A \emph{Gaussian martingale measure} is a family of centered Gaussian variables

\begin{equation*}
(W_t(A)~\vert~t\ge 0,A\in\mathcal A_k),
\end{equation*}

\noindent such that

\begin{enumerate}[i)]
\item $W_0(A) = 0$.
\item For all $A\in \mathcal A_K$, $(W_t(A))_{t\ge 0}$ is a continuous martingale with respect to the filtration $\sigma(W_s(A)~\vert~0\le s\le t,A\in\mathcal A_K)$.
\item Almost surely, $W_t(\emptyset) = 0$ and $W_t\left(\bigcup_n A_n\right) = \sum_n W_t(A_n)$ holds for all disjoint sequences $(A_n)_{n\in\mathbb N}$.
\item Its covariance is given by

\begin{equation*}
E(W_s(A)W_t(B)) = (s\land t) K(A\times B).
\end{equation*}
\end{enumerate}
\end{definition}

\begin{remark}
One can show that iii) is redundant in the sense that it can be derived from the other three properties. Since we decided not to present this result, we included the property in the definition.
\end{remark}

\noindent We also denote the closure of the filtration of $W_t$ as

\begin{equation*}
\mathcal F_t := \overline{\sigma(W_s(A)~\vert~0\le s\le t,A\in\mathcal A_K)}\,,
\end{equation*}

\noindent which is complete, right continuous, and, since we only have one filtration, fulfills the orthogonality assumption. The most common example of such a measure is space-time white noise, which is a Gaussian martingale measure with covariance measure $K^{WN}(dx,dy) = dx \,\delta_x(dy)$, which should be read as

\begin{equation*}
\int_{\R^d\times\R^d} f(x)g(y) K^{WN}(dx,dy) = \int_{\R^d}f(x)g(x) dx\,.
\end{equation*}

\noindent The construction of the integral over such martingale measure is rather straight-forward: We call processes of the form

\begin{equation*}
H(s,x) = \sum_{k=1}^{n-1}\sum_{l=1}^{L_k} h_{k,l} 1_{(t_k,t_{k+1}]}(s) 1_{A_{k,l}}(x)
\end{equation*}

\noindent for $\mathcal F_{t_k}$ measurable random variables $h_{k,l}\in L_{\infty}(\Omega)$ and sets $A_{k,l}\in\mathcal A_K$ \emph{elementary processes}, and define the integral over such processes by

\begin{equation*}
\int_0^\infty\int_{\R^d} H(s,x) W(ds,dx) := \sum_{k=1}^{n-1}\sum_{l=1}^{L_k} h_{k,l}(W_{t_{k+1}}(A_{k,l})-W_{t_k}(A_{k,l})).
\end{equation*} 

\noindent We want to extend this definition to a certain $L_2$ space. To do this, let $\mathcal P$ be the predictable $\sigma$-algebra, which is generated by the elementary processes. We set

\begin{equation*}
L_2(W) := \left\{H \text{ predictable}, E\left(\int_0^\infty \norm{H(s,\cdot)}_{\abs K}^2ds\right)<\infty \right\}.
\end{equation*}

\noindent The following result can be found in \cite{walshOriginal}:

\begin{lemma}
The set of elementary processes is dense in $L_2(W)$.
\end{lemma}

\noindent Our notion of an integral extends to the space $L_2(W)$, as the following theorem states:

\begin{theorem}[Walsh Integral]

The above-given integral can be extended uniquely to a continuous, linear map

\begin{align*}
L_2(W)&\rightarrow M^2_0 \\
H&\mapsto \int_0^t\int_{\R^d} H(s,x) W(ds,dx),
\end{align*}

\noindent where $M^2_0$ is the set of square-integrable martingales starting from 0.

\end{theorem}

\noindent The proof can be found in \cite{walshOriginal}.

By a simple approximation argument, one can show that the \emph{Itô-isometry} holds for all $H\in L_2(W)$ and $t\ge 0$:

\begin{equation*}
E\left(\left(\int_0^t\int_{\R^d} H(s,x) W(ds,dx)\right)^2\right) = E\left(\int_0^t \norm{H(s,\cdot)}_K^2ds\right).
\end{equation*}

\noindent With this, we can tackle the main result of this section:

\subsection{Reconstruction of the Walsh-type Integral}

Let $W$ be a Gaussian martingale measure, and let $X(t,x)$ be a predictable process. We want to show, that the distribution associated with the Walsh-type integral

\begin{equation*}
\psi\mapsto \int_0^\infty\int_{\R^d} \psi(s,x)X(s,x)W(ds,dx)\,,
\end{equation*}

\noindent can be constructed with the stochastic reconstruction theorem. Note that the integral is a martingale in time $t$, making our stochastic dimension $e=1$. 

We need a local approximation for the integral, which we get by the constant approximating $X(s,x)\approx X(t,y)$ for $(s,x)$ close to $(t,y)$. This motivates using the germ $F_{t,y} = X(t,y) dW$, i.e.

\begin{equation*}
F_{t,y}(\psi) := X(t,y)\int_0^\infty \int_{\R^d} \psi(s,x)W(ds,dx).
\end{equation*}

\begin{remark}
While this allows us to formally define the germ on any test function, it is only really useful on test functions with a support on the right-hand-side of $t$, so that $\tau\mapsto X(t,y)\int_0^\tau\int_{R^d}\psi(s,x)W(ds,dx)$ is a martingale. Our left-sided version of the reconstruction theorem mirrors this property.
\end{remark}

\noindent For simplicity, we restrict ourselves to $L_2(\Omega)$, but one can always use Gaussianity to extend our result into $L_p$ spaces for more general $p$. The regularity of $X$ is governed by its Hölder-continuity, i.e. we assume $X\in C^\alpha(L_2(\Omega))$ for some $\alpha>0$, where we use the classical scaling $s=(1,1,\dots,1)$. It is however not clear a priori how to measure the regularity of $W$. To this end, let $K$ be the covariance measure of $W$. We say that $K$ has a scaling $\delta>0$ if it holds that

\begin{equation*}
K(\lambda dx,\lambda dy) = \lambda^\delta K(dx,dy).
\end{equation*}

\noindent The following lemma shows, that $W$ as a random distribution is $-d+\frac 12+\frac\delta 2$-Hölder continuous, if its covariance measure has the scaling $\delta$:

\begin{lemma}\label{lemma4}
Let $K$ have scaling $\delta$, as above. It holds that:

\begin{equation*}
\int_0^\infty\norm{\phi_{t,y}^\lambda(s,\cdot)}_K^2 ds \lesssim \lambda^{2\alpha}
\end{equation*}

\noindent for $\alpha = -d-\frac 12+\frac\delta2$. We say that $K$ is of homogeneity $\alpha.$
\end{lemma}

\begin{proof}
A straightforward calculation shows:

\begin{align*}
\int_0^\infty \int_{\R^d\times\R^d} \phi_{t,y}^\lambda(s,x_1)&\phi_{t,y}^\lambda(s,x_2)K(dx_1,dx_2)ds \\& = \lambda^{-2d-2}\int_0^\infty \int_{\R^d\times\R^d} \phi_{t,y}\left(\frac s\lambda,\frac{x_1}\lambda\right)\phi_{t,y}\left(\frac s\lambda,\frac {x_2}\lambda\right)K(dx_1,dx_2)ds \\
&= \lambda^{-2d-2}\int_0^\infty \int_{\R^d\times\R^d} \phi_{\frac t\lambda,\frac y\lambda}(r,v_1)\phi_{\frac t\lambda,\frac y\lambda}(r,v_2)K(\lambda dv_1,\lambda dv_2)\lambda dr \\
&= \lambda^{-2d-1+\delta} \int_0^\infty \int_{\R^d\times\R^d} \phi_{\frac t\lambda,\frac y\lambda}(r,v_1)\phi_{\frac t\lambda,\frac y\lambda}(r,v_2)K(dv_1,dv_2) dr \\&= \lambda^{-2d-1+\delta}\int_0^\infty\norm{\phi_{\frac t\lambda,\frac y\lambda}(r,\cdot)}_K^2 dr \\&= \lambda^{-2d-1+\delta}\int_0^\infty\norm{\phi(s,\cdot)}_K^2 ds.
\end{align*}
\end{proof}

\begin{remark}\label{remark}
Using the Itô-isometry, this implies that

\begin{equation*}
\norm{\int_{\R_+\times\R^d}\psi(s,x)W(ds,dx)}_{L_2} = \left(\int_0^\infty\norm{\phi_{t,y}^\lambda(s,\cdot)}_K^2 ds\right)^{\frac 12}\lesssim \lambda^\alpha.
\end{equation*}

\noindent Thus, the distribution $\psi\mapsto \int_{\R_+\times\R^d}\psi(s,x)W(ds,dx)$ is indeed in $C^\alpha(L_2(\Omega))$.
\end{remark}

\noindent With this result, we can reconstruct the Walsh integral:

\begin{theorem}\label{Walsh}
Let $K$ be of homogeneity $\alpha$, and let $X\in C^\beta(L_2(\Omega))$ for some $\beta\in(0,1)$. Let $(F_{t,y})$ be as above. If $\alpha+\beta>-\frac 12$, then $(F_{t,y})$ fulfills the assumption of Theorem \ref{stochastic_reconstruction}, and thus there exists a distribution $f$ fulfilling \eqref{unique1} and \eqref{unique2}. It further holds, that $f$ is given by the Walsh-type integral, i.e.

\begin{equation}\label{walshRec}
f(\psi) = \int_0^{\infty}\int_{\R^d} X(s,x)\psi(s,x)W(ds,dx),
\end{equation}

\noindent for all $\psi\in C_c^r$ a.s., $r=\max(\lfloor-\alpha\rfloor+1,1)$.
\end{theorem}

\begin{remark}
Observe that by Lemma \ref{lemma4}, Theorem \ref{Walsh} automatically applies to all $K$ with a scaling $\delta$, such that

\begin{equation*}
\beta+\frac \delta2 > d.
\end{equation*}
\end{remark}

\begin{remark}
Looking at our example of white noise, we get that $K^{WN}(\lambda dx,\lambda dy) = (\lambda dx) \delta_x(\lambda dy) = \lambda^{d} K(dx,dy)$, which implies that $K^{WN}$ has scaling $\delta = d$ and thus, white noise has stochastic Hölder regularity $-\frac{d+1}{2}$. Therefore, we would need $\beta > \frac d2$ to apply the above theorem, which is a highly sub-optimal requirement. Indeed, it is well known that $\beta > 0$ suffices to define integration against space-time white noise \cite{KPZ}. The reason our result is much worse is that we only use the martingale properties in the time direction, i.e. we use stochastic dimension $e=1$, whereas white noise has additional martingale properties in all spatial directions and should have stochastic dimension $e = d+1$. Theorem \ref{Walsh} is much better suited to deal with cases of colored noise, which only have martingale properties in the time direction. The case of space-time white noise will be discussed at length in Section \ref{sectionWN}.
\end{remark}

\begin{proof}
We first show that the assumptions of Theorem \ref{stochastic_reconstruction} are fulfilled: Let $t\le s$, and let $\psi$ be a test function with support in $[0,1]\times[-1,1]^d$. It holds that:

\begin{align*}
E^{\mathcal F_{t}} &(F_{t,y}-F_{s,x})(\psi_{s,x}^\lambda) \\
&=E^{\mathcal F_{t}} \left((X(t,y)-X(s,x)) E^{\mathcal F_s}\int_0^\infty\int_{\R^d} \psi_{s,x}^\lambda(u,v) W(du,dv)\right) = 0,
\end{align*}

\noindent since $\int_0^\tau\int_{\R^d}\psi_{s,x}^\lambda(u,v) W(du,dv)$ is a martingale in $\tau$ and $s$ is on the edge of the support of $\psi_{s,x}^\lambda$. It remains to show, that

%Do I need t\le s? Do I need a reconstruction Theorem on {t\le s}?

\begin{equation*}
\norm{(F_{t,y}-F_{s,x})(\psi_{s,x}^\lambda)}_{L_2} \lesssim \lambda^{\tilde\alpha}(\abs{(s,x)-(t,y)}+\lambda)^{\gamma-\frac 12-\tilde\alpha}
\end{equation*}

\noindent for some $\tilde\alpha\in\R, \gamma > 0$. We chose $\gamma = \alpha+\beta+\frac 12$ and observe the following:

\begin{align}
\norm{(F_{t,y}-F_{s,x})(\psi_{s,x}^\lambda)}_{L_2}^2 &= \norm{(X(t,y)-X(s,x)) \int_0^\infty\int_{\R^d} \psi_{s,x}^\lambda(u,v) W(du,dv)}_{L_2}^2 \nonumber
\nonumber\\
&\le \int_0^\infty\int_{\R^d\times\R^d} \underbrace{\norm{(X(t,y)-X(s,x))}_{L_2}^2}_{\lesssim \abs{(t,y)-(s,x)}^{2\beta}} \psi_{s,x}^\lambda(u,v)\psi_{s,x}^\lambda(u,w) K(dv,dw) du \nonumber
\\&\lesssim \abs{(t,y)-(s,x)}^{2\beta}  \int_0^\infty \norm{\psi_{s,x}^\lambda(u,\cdot)}_K^2 du \nonumber
\\&\lesssim \abs{(t,y)-(s,x)}^{2\beta}\lambda^{2\alpha}. \label{calc}
\end{align}

\noindent Thus, using $\tilde \alpha = \alpha$, it follows that

\begin{align*}
\norm{(F_{t,y}-F_{s,x})(\psi_{s,x}^\lambda)}_{L_2} \lesssim \lambda^{\tilde \alpha} (\abs{(t,y)-(s,x)}+\lambda)^{\alpha+\beta-\tilde\alpha} = \lambda^{\tilde \alpha} (\abs{(t,y)-(s,x)}+\lambda)^{\gamma-\frac 12-\tilde\alpha}.
\end{align*}

\noindent This shows that there is a unique reconstruction $f$. To find $r$, note that the proof of Theorem \ref{stochastic_reconstruction} requires $\tilde r = r$ to be strictly greater than the absolute value of $\tilde \alpha = \alpha$, and $r> -\gamma+\frac 12 > +\frac 12$. It follows that $r=\max(\lfloor-\alpha\rfloor+1,1)$.

It remains to show \eqref{walshRec}. To this end, assume that the support of $\psi$ is in $[\tilde C,\tilde R]\times[-1,1]^d$ for some $0<\tilde C<\tilde R$. We need to show that \[\tilde f(\psi) := \int_0^{\infty}\int_{\R^d} X(s,x)\psi(s,x)W(ds,dx)\] fulfills \eqref{unique1} and \eqref{unique2}. Note that \eqref{unique2} is again obvious due to the martingale properties of \newline $\int_0^t\int_{\R^d} X(s,x)\psi(s,x)W(ds,dx)$, so it remains to show \eqref{unique1}. Using the same calculation as in \eqref{calc}, we see that

\begin{align*}
\norm{(\tilde f- F_{t,y})(\psi_{t,y}^\lambda)}_{L_2} &= \norm{\int_0^\infty\int_{\R^d} (X(s,x)-X(t,y))\psi_{t,y}^\lambda(s,x)W(ds,dx)}_{L_2} \\
&\lesssim \sup_{(s,x)\in \supp(\psi_{t,y}^\lambda)} \underbrace{\norm{X(s,x)-X(t,y)}_{L_2}}_{\lesssim \lambda^\beta} \int_0^\infty \norm{\psi_{s,x}^\lambda}_K^2 ds \\&\lesssim \lambda^{\alpha+\beta} = \lambda^{\gamma-\frac 12}.
\end{align*}

\noindent Thus, $\tilde f = f$ up to modifications, which finishes the proof.
\end{proof}

\section{Integration against white noise}\label{sectionWN}

Let us take a closer look at space-time white noise, which is given by the Gaussian martingale measure $W^{WN}$ characterized by its covariance measure $K^{WN}(dx,dy) = dx\,\delta_x(dy)$. As mentioned above, it has scaling $\delta = d$ and does not fall into the setting of Theorem \ref{Walsh}. This is not at all surprising, as the previous section assumes a stochastic dimension of $e=1$, whereas space-time white noise has martingale properties in all of its directions and should be treated as $e= d+1$. This makes integration against white noise an easy example of a stochastic reconstruction of a non-trivial stochastic dimension, which is our main motivation for discussing it in this paper.

Let us first discuss the construction of space-time white noise: Treating it as a Gaussian martingale measure comes with the downside that we treat time and space differently in the construction, whereas white noise treats both the same. If we look at the space $\R^{d+1}$ containing vectors of the form $(t,x_1,\dots, x_d)$ and look at a Gaussian family $(\tilde W(A))_{A\in\mathcal A_k}$ for $\mathcal A_k\subset B(\R^{d+1})$ with covariance $E(\tilde W(A)\tilde W(B)) = \tilde K(A,B)$ and covariance measure $\tilde K(dx,dy) = dx\,\delta_x(dy)$ in $\R^{d+1}$, it is easy to see that $W^{WN}_t(A) = \tilde W([0,t]\times A)$ in distribution. So it is far more natural to define space-time white noise as a Gaussian family over sets in $\R^{d+1}$, and forget that the first variable is given by time.

The most common notion of white noise takes this idea a bit further, and directly defines white noise as a random distribution: As shown in the literature (e.g. \cite{janson},\cite{nualart}) there exists a linear map $\xi:L_2(\R^d)\to L_2(\Omega)$ such that for each $\psi\in L_2(\R^d)$, $\xi(\psi)$ is a centered Gaussian random variable and the family $(\xi(\psi))_{\psi\in L_2(\R^d)}$ has the following covariance:

\begin{equation}\label{covarianceNew}
E(\xi(\psi)\xi(\phi)) = \int_{\R^d}\psi(z)\phi(z) dz\,,
\end{equation}

\noindent which is also often referred to as $E(\xi(x)\xi(y)) = \delta_x(y)$.  This makes $\xi$ a random $L_2(\R^d)$ distribution (continuity follows from $\norm{\xi(\phi)}_{L_2(\Omega)}^2 = \int_{\R^d} \phi(x)^2 dx = \norm{\phi}_{L_2(\R^d)}$ by \eqref{covarianceNew}). The connection to $\tilde W$ is given by $\tilde W(A) = \xi(1_A)$ for all compact sets $A\in B(\R^d)$, and since $\xi: L_2(\R^d)\mapsto L_2(\Omega)$ is linear and continuous, defining it on the indicator functions uniquely determines it on all of $L_2(\R^d)$. One can think of $\xi(\psi)$ as the integration of white noise against deterministic test functions, i.e. $``\xi(\psi) = \int_{\R^d}\psi(x)\tilde W(dx)$''.

In the remainder of this section, we will call $\xi$ white noise and use the notation $\xi(\psi) = \int_{\R^d}\psi(x)\xi(dx)$.

Our goal in this section is to use stochastic reconstruction to construct the higher-order iterated integrals

\begin{equation}\label{whiteNoiseInt}
\int_{\R^d\times\dots\times\R^d} \psi(z_1,\dots,z_n)\xi(dz_1)\dots\xi(dz_n)\,,
\end{equation}

\noindent
for $\psi\in L_2(\R^{d\times n})$, which make up the Wiener chaos \cite{janson}, \cite{nualart}. To do so, we use the approach of \cite{KPZ}: Set $X_0(z_1,\dots, z_n) := \psi(z_1,\dots,z_n)$ and for $k=1,\dots, n$, we integrate one variable out: $X_{k+1}(z_1,\dots,z_{n-k-1}) = \int X_k(z_1,\dots, z_{n-k})\xi(dz_{n-k})$. This way, $X_n\in L_2(\Omega)$ is the desired integral.

So the problem comes down to defining the integral 

\begin{equation*}
\int X(z)\xi(dz)
\end{equation*}

\noindent for a random process $X$, which we will do via stochastic reconstruction. Since $\xi$ is defined by its covariance, Theorem \ref{stochastic_recon_covariance} will be perfectly suited to deal with this, but let us shortly highlight what would happen if one wants to use Theorem \ref{stochastic_reconstruction}: To do so, one needs to fix filtrations 

\begin{equation*}
\mathcal F^{(i)}_{t} = \overline{\sigma\left(\xi(\psi)~\big|~ \supp(\psi)\subset\R^{i-1}\times(-\infty,t]\times\R^{d-i}\right)}.
\end{equation*}

\noindent These filtrations are the same as the filtrations $\underline{\underline{F_t^i}}$ generated by a Brownian sheet in \cite{walshMultiparameter}. As John B. Walsh shows in this paper, they fulfill our assumptions of completeness, right continuity, and orthogonality. (Orthogonality is equivalent to conditional independence, see Remark \ref{rem:condIndependence}.)

Note that by \eqref{covarianceNew}, for $\psi,\phi\in L_2(\R^d)$ with disjoint support, $\xi(\psi)$ and $\xi(\phi)$ are uncorrelated. Since $\xi$ is a Gaussian family, this implies that $\xi(\psi)$ and $\xi(\phi)$ are independent. Thus, for a test function with positive support and any $i=1,\dots,d$, it follows that $\xi(\psi_z^\lambda)$ is independent of $\mathcal F^{(i)}_{\pi_i(z)}$.

For a stochastic process $X$, which is adapted to all directions $i=1,\dots,d$ in the sense that $X(z)\in\mathcal F^{(i)}_{\pi_i(z)}$ and under the assumption that $X\in C^\gamma(L_2(\Omega))$ for some $\gamma>0$, one can now show that the germ $F_z(\psi) := X(z)\xi(\psi)$ is indeed stochastically $\gamma$-coherent with stochastic dimension $e=d$ and can be reconstructed. Moreover, it will turn out that the reconstruction is nothing else than the classical integral $f(\psi) = \int_{\R^d} X(z)\psi(z) \xi(dz)$, which we define in \eqref{classicalWNInt}.

However, there is one subtlety hidden in this construction: If we want to use this to construct \eqref{whiteNoiseInt}, one notices that $X_k(z_1,\dots, z_{n-k})$ is in general not $\mathcal F^{(i)}_{\pi_i(z_{n-k})}$ adapted. We can save this by setting

\begin{equation*}
X_{k+1}(z_1,\dots,z_{n-k-1}) = \int_{z_{n-k}\le z_{n-k-1}} X_k(z_1,\dots,z_{n-k}) \xi(d z_{n-k})\,,
\end{equation*}

\noindent where we say $x \le y$ if and only if $\pi_i(x)\le \pi_i(y)$ for $i=1,\dots, d$. This ensures that $X_{k}(z_1,\dots,z_{n-k})$ is $\mathcal F^{(i)}_{\pi_i(z_{n-k})}$ measurable. However, if one does this, one only succeeds in defining the iterated integral

\begin{equation*}
\int_{z_1\le\dots\le z_n\in\R^d} \psi(z_1,\dots,z_n)\xi(dz_1)\dots\xi(dz_n)\,.
\end{equation*}

\noindent This somehow resembles the construction in \cite{KPZ}, where one only manages to directly construct the integral on the set $t_1\le t_2\le\dots\le t_n$, where $t_i = \pi_1(z_i)$ is the time variable of the vector $z_i$. In this case, the construction can be saved by using the symmetry of white noise: By assumption, it holds that

\begin{equation*}
\int_{\R^d} \psi(z_1,\dots,z_n)\xi(dz_1)\dots\xi(dz_n) = \int_{\R^d} \psi(z_{\sigma(1)},\dots,z_{\sigma(n)})\xi(dz_1)\dots\xi(dz_n)
\end{equation*}

\noindent for all permutations $\sigma$. Thus, if one can define the integral for $t_1\le\dots\le t_n$, one can use permutations to extend the integral to all of $\R^d$.

Unfortunately, this does no longer work if one is only able to intrinsically define \eqref{whiteNoiseInt} over the area $z_1\le\dots\le z_n$. For example for two points $z_1,z_2$ with $\pi_1(z_1)<\pi_1(z_2)$ and $\pi_2(z_1)>\pi_2(z_2)$, there is no permutation $\sigma:\{1,2\}\to\{1,2\}$ such that $z_{\sigma(1)}\le z_{\sigma(2)}$. Thus, we can not extend the integral to the whole space by symmetry alone. We know of three ways, how one can get around this subtlety:

\begin{itemize}
\item By using time-swapped filtrations

\begin{equation*}
\tilde{\mathcal F}^{(i)}_{t} = \overline{\sigma\left(\xi(\psi)\big| \supp(\psi)\subset\R^{i-1}\times[t,\infty)\times\R^{d-i}\right)}\,,
\end{equation*}

\noindent one can construct integrals of the form $ \int_{z_{n-k}\ge z_{n-k-1}} X_k(z_1,\dots,z_{n-k}) \xi(d z_{n-k})$. By cleverly using combinations of time-swapped filtrations in some directions $\eta\subset\{1,\dots,d\}$, and the normal filtrations in the other directions $\{1,\dots,d\}\setminus\eta$, one can construct the integral over all areas individually. To our knowledge, this approach would work, although it is quite tedious.

\item One can construct sigma-algebras which more accurately describe the measurability of $F_x(\psi)$. Right now, we only consider sigma-algebras depending on $\max(\pi_i(x),\pi_i(y))$, where the support of $\psi$ is contained in $(-\infty,\pi_1(y)]\times\dots\times(-\infty,\pi_d(y)]$. If one constructs sigma-algebras depending on $x$ and the support of $\psi$ in a more sophisticated manner, it is possible to show that the reconstructing sequence $f^{(n)}$ \eqref{eq:reconstructingSequence} converges to a limit $f$ for all $X$ such that $X(z)\in\mathcal F^{(1)}_{\pi_1(z)}$, reducing the adaptedness assumptions on $X$ to only requiring adaptedness in time. However, the problem with this approach is that it is highly unclear how the adaptedness assumption on $F_x(\psi)$ should translate to some assumption on $f$ since it no longer has an index $x$. So in this setting, it is not clear under which condition $f$ is unique. 

In Conclusion, this approach would allow one to show that the reconstructing sequence for the germ $F_z(\psi)$ defined in Theorem \ref{theoWN} converges against the correct limit \eqref{eq:WNInt}, but we are currently unable to generalize this observation to a new version of the stochastic reconstruction theorem.

\item The easiest way to deal with this problem is to use Theorem \ref{stochastic_recon_covariance}. It allows us to directly use the covariance of our germ and thus allows us to easily reconstruct the integral, even if $X$ is only adapted in time. As mentioned above, one can use this together with the symmetry of white noise to construct \eqref{whiteNoiseInt} over all of $\R^d$. In the remainder of this section, we will present this approach.

\end{itemize}

\noindent Let $(\mathcal F^{(1)}_t)_{t\in\R}$ be as above, let $X\in C^\gamma(L_2(\Omega))$ for some $\gamma > 0$ and assume that $X(z)\in\mathcal F^{(1)}_{\pi_1(z)}$ for all $z\in\R^d$. We can then show the main result of this section:

\begin{theorem}\label{theoWN}
Let $F_z(\psi) = X(z)\xi(\psi)$. Then $(F_z)_{z\in\R^d}$ is stochastically $(\gamma,\alpha)$-coherent in the covariance setting with stochastic dimension $e=d$ and $\alpha = -\frac d2$. Thus, it can be reconstructed on $C_c^{r}$ with $r=\lfloor d/2 \rfloor +1$.
\end{theorem}

\begin{proof}
We need to show \eqref{coherenceCov1} and \eqref{coherenceCov2}. To see \eqref{coherenceCov1}, let $x,y\in K$ with $\pi_i(x)\le\pi_i(y)$ for $i=1,\dots,e$, and let $\psi\in C_c^r$ have a positive support. Since $X(x)-X(y)$ is $\mathcal F_{\pi_1(y)}^1$ measurable, it is independent of $\xi(\psi_y^\lambda)$ for any $\lambda\in (0,1]$. It follows that

\begin{align*}
\norm{(F_x-F_y)(\psi_y^\lambda)}_{L_2} = \norm{X(z)-X(y)}_{L_2}\norm{\xi(\psi_y^\lambda)}_{L_2} \lesssim \norm{x-y}^{\gamma}\lambda^{-\frac d2}\,.
\end{align*}

It remains to show \eqref{coherenceCov2}. Let $\psi_1,\psi_2$ be two test functions with positive support and let $x_i,y_i,\lambda_i$ be as in Theorem \ref{stochastic_recon_covariance}. Also, without loss of generality, let $\pi_1(y_1)\le \pi_1(y_2)$. Note that the disjoint effective support implies that $(\psi_i)_{y_i}^{\lambda_i}$, $i=1,2$ have disjoint support and thus, $\xi((\psi_i)_{y_i}^{\lambda_i}), i=1,2$ are independent. We decompose $(\psi_1)_{y_1}^{\lambda_1} = (\psi_1)_{y_1}^{\lambda_1}(1_{\pi_1(\cdot)\le\pi_1(y_2)} + 1_{\pi_1(\cdot)>\pi_1(y_2)})$ and calculate

\begin{align*}
E\left[\prod_{i=1}^2 (F_{x_i}-F_{y_i})((\psi_i)_{y_i}^{\epsilon_i})\right] &= E\Bigg[\prod_{i=1}^2 (X(x_i)-X(y_i)) \times \bigg(\xi\left((\psi_1)_{y_1}^{\lambda_1} 1_{\pi_1(\cdot)\le \pi_1(y_2)}\right) \times \underbrace{E\left[\xi\left((\psi_2)_{y_2}^{\lambda_2}\right)~\middle\vert \mathcal F_{\pi_1(y_2)}^{(1)}~\right]}_{=0}\\
&\qquad+ \underbrace{E\left[\xi\left((\psi_1)_{y_1}^{\lambda_1} 1_{\pi_1(\cdot)> \pi_1(y_2)}\right)\xi\left((\psi_2)_{y_2}^{\lambda_2}\right)\middle\vert \mathcal F^{(1)}_{\pi_1(y_2)}\right]}_{= E\left[\xi\left((\psi_1)_{y_1}^{\lambda_1} 1_{\pi_1(\cdot)> \pi_1(y_2)}\right)\xi\left((\psi_2)_{y_2}^{\lambda_2}\right)\right] = 0}\bigg)\Bigg] \\
&= 0\,.
\end{align*}

\noindent This shows the claim.
\end{proof}

\noindent We call the reconstruction of $(F_z)$ $f$, as usual. Our next goal is to show, that $f$ is indeed the classical integral against white noise,  which gets constructed the same way as in the martingale measure case: One first defines it on the set of simple functions given by

\begin{equation*}
X(t,x) = \sum_{i=1}^m X_i 1_{(a_i,b_i]}1_{A_i}(x),
\end{equation*}

\noindent
where $a_i<b_i$, $X_i$ is a bounded, $\mathcal F^{(1)}_{a_i}$ measurable random variable and $A_i\in\mathcal B(\R^{d-1})$ is a Borel measurable set. Setting $\mathcal S$ to be the set of simple functions and $\mathcal P$ to be the $\sigma$-field generated by $\mathcal S$, it is then well known that $\mathcal S$ is dense in $L_2(\R^d\times\Omega,\mathcal P)$ (\cite{KPZ}, Lemma 2.1).

For a simple process $X(t,x)$, we can directly define

\begin{equation*}
\int X(z)\xi(dz) := \sum_{i=1}^m X_i\int 1_{(a_i,b_i)}(t)1_{A_i}(x) \xi(dt,dx)\,.
\end{equation*}

\noindent
One can show that 

\begin{align}
\int \cdot ~\xi(dz) : \mathcal S&\to L_2(\Omega)\nonumber\\
X& \mapsto \int X(z)\xi(dz)\label{classicalWNInt}
\end{align} 

\noindent is linear and continuous and thus one can extend the map by an approximation argument to $L_2(\R^d\times\Omega,\mathcal P)$.

\begin{remark}
While this construction is the same as for martingale measures, it does not hold that any martingale measure with the regularity of white noise admits such an integral, so one has to show the existence of the white noise integral separately.
\end{remark}

\noindent For an $X\in L_2(\R^d\times\Omega,\mathcal P)$, we denote the integral against white noise with $X\cdot\xi$ and for a test function $\psi\in C_c^r$, we set

\begin{equation}\label{eq:WNInt}
(X\cdot\xi)(\psi) := ((X\psi)\cdot\xi) = \int_{\R^d}X(z)\psi(z)\xi(dz)\,.
\end{equation}

\noindent Our goal is to show that for all $\psi\in C_c^r$, it holds that 

\begin{equation*}
(X\cdot\xi)(\psi) = f(\psi)
\end{equation*}

\noindent a.s., where $f$ is the reconstructed random distribution attained in Theorem \ref{theoWN}. Before we show this, let us make the following observations: For a $\mathcal F^{(1)}_{t}$-measurable random variable $X\in L_2(\Omega)$ and a test function $\psi$ with support in $[t,\infty)\times\R^{d-1}$, it holds that

\begin{equation}\label{stepwise}
\int_{z\in\R^d}X\psi(z)\xi(dz) = X\int_{z\in\R^d}\psi(z)\xi(dz) = F_y(\psi),
\end{equation}

\noindent if $X(y) = X$. This simply follows from approximating $\psi$ with piecewise constant functions and applying the definition of our integral. We further recall the Itô-isometry: For all $X\in L_2(\R^d\times\Omega,\mathcal P)$,

\begin{equation}\label{WienerIto}
\norm{\int_{z\in\R^d}X(z)\xi(dz)}_{L_2(\Omega)}^2 = \int_{\R^d}\norm{X(z)}_{L_2(\Omega)}^2 dz,
\end{equation}

\noindent which can be found as equation (126) in \cite{KPZ}. With this in mind, we can show:

\begin{theorem}
Let $X$ be as in Theorem \ref{theoWN}. Then, for any $\psi\in C_c^r$, $r = \lfloor d/2\rfloor +1$, it holds that almost surely

\begin{equation*}
(X\cdot\xi)(\psi) = f(\psi)\,,
\end{equation*}

\noindent where $f$ is the reconstruction of the germ $F_z(\psi) = X(z)\xi(\psi)$.
\end{theorem}

\begin{proof}
Since $X$ is $\mathcal F^{(1)}$ adapted and continuous, it follows that it is $\mathcal P$ measurable and due to $\psi$ being compactly supported, $X\psi$ is square-integrable. Thus, $(X\cdot\xi)(\psi)$ is well-defined. We show that $(X\cdot\xi)(\psi)$ fulfills \eqref{uniqueCov1} and \eqref{uniqueCov2}, the claim then follows from the uniqueness of the reconstruction. To do this, let $\psi$ have strictly positive support. Then \eqref{stepwise} implies, that $F_z(\psi_z^\lambda) = \int_{y\in\R^d}X(z)\psi_z^\lambda(y)\xi(dy)$, and we can use \eqref{WienerIto} to calculate

\begin{align*}
\norm{\int_{y\in\R^d}X(y)\psi_z^\lambda(y)\xi(dy)-F_z(\psi_z^\lambda)}_{L_2} &= \norm{\int_{y\in\R^d}(X(y)-X(z))\psi_z^\lambda(y)\xi(dy)}_{L_2} \\
&=\left(\int_{\R^d}\norm{X(y)-X(z)}_{L_2(\Omega)}^2\abs{\psi_z^\lambda(y)}^2 dy\right)^{\frac 12} \\
&\le \sup_{y\in \supp(\psi_z^\lambda)}\norm{X(y)-X(z)}_{L_2(\Omega)}\norm{\psi_y^\lambda}_{L_2(\R^d)} \\
&\lesssim \lambda^{\alpha}\lambda^{-\frac d2}\,,
\end{align*}

\noindent Which shows \eqref{uniqueCov1}. It further holds that for any test function $\psi$ with positive support, $\lambda>0$ and $x_1,x_2$ such that the stochastic effective supports do not overlap, that

\begin{align*}
E\left[\prod_{i=1}^2\left((X\cdot\xi)-F_{x_i}\right)(\psi_{x_i}^\lambda)\right] &= E\left[\prod_{i=1}^2 \int_{\R^d}(X(z)-X(x_i))\psi_{x_i}^\lambda(z) \xi(dz)\right]\\
&= 0\,,
\end{align*}

\noindent since the support of $\psi_{x_i}^\lambda$, $i=1,2$ are disjoint. Thus, \eqref{uniqueCov2} holds, which shows the claim.
\end{proof}

\appendix
\section*{Appendix}
\section{Wavelet Bases}

This section is dedicated to the introduction of the wavelet techniques used in Section \ref{sec:proofs}, and mainly follows Section 3.1 of \cite{Hairer}. If one is interested in a more detailed discussion of the topic, it can be found in \cite{meyer}.

Intuitively, a wavelet basis is given by a smooth, compactly supported function $\phi$, such that any test function $\psi$ can be decomposed in the form

\begin{equation}\label{wavelet}
\psi = \lim_{n\to\infty} \sum_{k\in\Delta_n} \scalar{\psi,\phi_k^n}\phi_k^n,
\end{equation}

\noindent where the limit holds in $L_2(\R^d)$ and $\phi_k^n$ is a certain localized version of $\phi$ around $k$, $\Delta_n = 2^{-n}\mathbb Z$ is a mesh, which gets increasingly fine as $n\to\infty$. This allows us to analyze arbitrary test functions $\psi$, using only the local properties of $\phi_k^n$. Such a wavelet can be constructed using the following theorem of Daubechies, which can be found in \cite{debauchies}:

\begin{theorem}[Daubechies]\label{theoWave1}

For any $r\in\mathbb N$, there exists a function $\phi\in C_c^r(\R)$, such that

\begin{enumerate}[i)]
\item For all $z\in\mathbb Z$, it holds that

\begin{equation*}
\scalar{\phi,\phi(\cdot-z)} = \delta_{0,z}.
\end{equation*}

\item $\phi$ is self-replicable, i.e. there exist constants $(a_k)_{k\in\mathbb Z}$, such that

\begin{equation*}
\sqrt{\frac 12}\phi(x/2) = \sum_{k\in\mathbb Z} a_k\phi(x-k).
\end{equation*}
\end{enumerate}
\end{theorem}

\noindent Due to the compact support of $\phi$, it holds that $a_k = 0$ for all but finitely many $k$. In Daubechies' setting, it does not matter where the support of $\phi$ lies, as long as it is compact. However, as we have seen several times in this paper, the stochastic theory requires us to separate the number line $\R$ into the ``future'' $[0,\infty)$ and the ``past'' $(-\infty,0]$. In particular, we are only interested in functions $\phi$ with compact support in the future, i.e. $\supp(\phi)\subset[0,R]$ for some $R>0$. Such a wavelet can be easily constructed by replacing $\phi$ with $\phi(\cdot-x)$ for a large enough $x>0$. Further observe that by choosing $x$ large enough, we can also assure that $a_k=0$ for all $k<0$ in property ii), and that $\phi$ has support in $[C,R]$ for some constants $0<C<R$.

Property $i)$ ensures that the recentered wavelets $\phi_z$, $z\in\mathbb Z$ form an orthonormal system in $L_2(\R)$. For $n\in\mathbb N, k\in\Delta_n$, we set

\begin{equation*}
\phi_k^n(x) := 2^{\frac n2}\phi(2^{n}(x-k)) = 2^{-\frac n2}\phi_k^{2^{-n}}(x).
\end{equation*}

\noindent to be the $L_2$-normalization of $\phi$. It is shown in \cite{debauchies}, that $(\phi_k^n)_{k\in\Delta_n}$ almost form an $L_2(\R)$ basis in the sense that \eqref{wavelet} holds in $L_2(\R)$. Since $\phi$ is self-replicable, it follows that 

\begin{equation*}
\phi_k^{n} = \sum_{l\in\Delta_{n+1}} a_{l}^{n+1} \phi_{k+l}^{n+1}\,,
\end{equation*}

\noindent is also self-replicable with $a_l^{n+1} := a_{2^{n+1}l}$\,.

The second ingredient of a wavelet basis is a function $\hat\phi$, which is able to give one a description of the ``missing part'' $\psi-\sum_{k\in\Delta_n}\scalar{\psi,\phi_k^n}\phi_k^n$ in the sense that

\begin{equation}\label{eq:hatPhi}
\psi = \sum_{k\in\Delta_n} \scalar{\psi,\phi_k^n}\phi_k^n + \sum_{m\ge n}\sum_{k\in\Delta_m} \scalar{\psi,\hat\phi_k^m}\hat\phi^m_k
\end{equation}

\noindent
holds for any $n\in\mathbb N$ in $L_2(\R)$, where $\hat\phi_k^n := 2^{\frac n2}\hat\phi(2^n(x-k))$ as before. The existence of such a function is shown through the following theorem from \cite{meyer}:

\begin{theorem}[Meyer]\label{theoWave2}

\noindent
Let 

\begin{equation*}
V_n := \mySpan\{\phi_k^n~\vert~k\in\Delta_n\}\,.
\end{equation*}

\noindent
Then there is a function $\hat\phi_k^n\in C^r_c(\R)$, such that

\begin{enumerate}[i)]

\item $\hat V_n := \mySpan\{\hat\phi_k^n~\vert~k\in\Delta_n\}$ is the orthorgonal complement of $V_{n}$ in $V_{n+1}$, i.e. $\hat V_n = V_{n}^{\perp V_{n+1}}$.

\item $(\hat \phi_k^n)_{n\in\mathbb N,k\in\Delta_n}$ is a family of orthonormal functions:

\begin{equation*}
\scalar{\hat\phi_k^n,\hat\phi_l^m} = \delta_{n,m}\delta_{k,l}.
\end{equation*}

\item $\hat\phi$ eliminates all polynomials of degree less than $r$: For any monomial $x^m$, $m\le r$, it holds that

\begin{equation*}
\scalar{x^m,\hat\phi} = 0.
\end{equation*}
\end{enumerate}
\end{theorem}

\noindent
Note that $V_{n+1}\subset V_n$ by self-replicability. \eqref{wavelet} implies that $L_2(\R) = \overline{\bigcup_{n\in\mathbb N} V_n}$, which together with property i) automatically gives us $L_2(\R) = \overline{V_n \cup \left(\bigcup_{m\ge n}\hat V_m\right)}$ and thus shows \eqref{eq:hatPhi}. Property iii) causes $\scalar{\psi,\hat\phi_k^m}$ to decay faster than one would expect as $m\to\infty$, which is formalized in Lemma \ref{lemma1}.

Let us generalize this setting to $\R^d$ equipped with the scaling $S$: Define

\begin{equation*}
\Delta_n^S = \left\{\sum_{i=1}^d 2^{-ns_i}l_i\vec e_i ~\bigg\vert~ l_i\in\mathbb Z\right\},
\end{equation*}

\noindent to be the rescaled mesh on $\R^d$, where $\vec e_i = (0,\dots,0,1,0,\dots,0)$ is the $i$-th unit vector of $\R^d$. We usually drop the index $S$ on $\Delta_n^S$ and only write $\Delta_n$, since the scaling is fixed within each section of this paper. The wavelet function in $\R^d$ is then simply given by

\begin{equation*}
\phi(y) = \prod_{i=1}^d \phi(\pi_i(y))\,,
\end{equation*}

\noindent with its $L_2$ localizations $\phi_k^n(x) := 2^{n\frac d2}\phi(2^n(x-k))$ as before. We define

\begin{align}
V_n &:= \mySpan\{\phi_k^n~\vert~ k\in\Delta_n\} \label{Vl}\\
\hat V_n &:= V_n^{\perp V_{n+1}} \label{hatVl}
\end{align}

\noindent as before. As it turns out, $\hat V_n$ no longer gets generated by a single function $\hat\phi$, but by a finite set of functions $\hat\phi\in\Phi$, which have the property that $\hat V_n = \mySpan\{\hat\phi_k^n~\vert~\hat\phi\in\Phi,k\in\Delta_n\}$. To be more precise, $\Phi$ is given by all possible functions of the form $\hat\phi(x_1,\dots,x_n) = \prod_{i=1}^n \tilde \phi_i(x_i)$, where all $\tilde\phi_i(x_i)$ are either $\phi(x_i)$ or $\hat\phi(x_i)$ and at least one $\tilde\phi_i(x_i)$ is given by $\hat\phi(x_i)$ (for reference, see \cite{meyer}). This construction immediately gives us that $\phi, \hat\phi$ have support in $[C,R]^d$ for some $0<C<R$ and that $\phi$ is self-replicable with $(a_k)_{k\in\mathbb Z^d}$ such that $a_k\neq 0$ only if $\pi_1(k),\dots,\pi_d(k)\ge 0$.

Since the one-dimensional wavelet $\hat\phi$ eliminates polynomials of degree less or equal to $r$, this implies that for all $\hat\phi\in\Phi$, $\scalar{\hat\phi,p} = 0$ for all multivariate polynomials $p$ of degree less or equal to $r$.

We use the following notations for the projections onto $V_n$ and $\hat V_n$ in $L_2(\R^d)$:

\begin{align*}
P_n \psi &:= \sum_{k\in\Delta_n} \scalar{\psi,\phi_k^n}\phi_k^n \\
\hat P_n \psi &:= \sum_{\hat\phi\in\Phi}\sum_{k\in\Delta_n} \scalar{\psi,\hat\phi_k^n}\hat\phi_k^n \,.
\end{align*}

\noindent
We will need to quantify the speed of convergence of $\scalar{\psi,\phi_k^n}, \scalar{\psi,\hat\phi_k^n}$ for $n\to\infty$. This gets achieved by the following lemma:

\begin{lemma}\label{lemma1}
Let $\psi\in C_c^r$. Then, for all $\hat\phi\in\Phi$, $n\in\mathbb N, y\in\Delta_n,z\in\R^d$ and $2^{-n}\le\lambda\le 1$, it holds that

\begin{align}
\abs{\scalar{\phi_y^n,\psi_z^\lambda}} &\lesssim 2^{-n\frac {\abs{S}}2}\lambda^{-\abs{S}}\norm{\psi}_{C_c^r} \label{LemIneq1}\\
\abs{\scalar{\hat\phi_y^n,\psi_z^\lambda}} &\lesssim 2^{-n\frac {\abs{S}}2-n\tilde r}\lambda^{-\abs{S}-\tilde r}\norm{\psi}_{C_c^r}.\label{LemIneq2}
\end{align}

\noindent where $\tilde r = r\cdot\min_{i=1,\dots,d}(s_i)$.
\end{lemma}

\begin{proof}
\noindent
The proof consists mainly of straight-forward calculations. To show \eqref{LemIneq1}, observe that

\begin{align*}
\abs{\scalar{\phi_y^n,\psi_z^\lambda}} &\le 2^{n\frac {\abs{S}}2}\lambda^{-\abs{S}} \int \abs{\phi\left(S^{2^{n}}(x-y)\right)\psi\left(S^{\frac 1\lambda}(x-z)\right)}dx \\
&= 2^{-n\frac {\abs{S}}2} \lambda^{-\abs{S}} \int\abs{\phi(u-S^{2^{n}}y)\psi\left(S{^\frac 1{2^n\lambda}}u - S^{\frac 1\lambda}z\right)}du \\&\le 2^{-n\frac {\abs{S}}2} \lambda^{-\abs{S}}\sup_{u\in\R^d}\abs{\psi(u)}\norm{\phi}_{L_1} \\&\lesssim 2^{-n\frac {\abs S}2} \lambda^{-\abs S} \norm{\psi}_{C_c^r}.
\end{align*}

\noindent To see \eqref{LemIneq2} we will use that $\hat\phi$ eliminates polynomials of order less or equal to $r$. Let $T^{r-1}_{x_0}$ be the $(r-1)$-th Taylor approximation of $\psi$ at the point $x_0$. Recall that

\begin{equation*}
\abs{\psi(x_0+u)-T^{r-1}_{x_0}(x_0+u)}\lesssim \norm{u}_1^r\norm{\psi^{(r)}}_{\infty},
\end{equation*}

\noindent where $\norm{u}_1 = \abs{u_1}+\dots+\abs{u_d}$ is the 1-norm of the vector $u$. Let $\tilde s = \min_{i=1,\dots,d}(s_i)$ and observe that $\frac{1}{2^n\lambda}\le 1$ implies

\begin{align*}
 \norm{S^{\frac 1{2^n\lambda}}u}_1 &= \abs{\frac{1}{2^{ns_1}\lambda^{s_1}}u_1}+\dots+\abs{\frac{1}{2^{ns_d}\lambda^{s_d}}u_d}\\
 & = \left(\frac{1}{2^n\lambda}\right)^{\tilde s}\left(\abs{\frac{1}{2^{ns_1-\tilde s}\lambda^{s_1-\tilde s}}u_1}+\dots+\abs{\frac{1}{2^{ns_d-\tilde s}\lambda^{s_d-\tilde s}}u_d}\right)\\
 &\le \left(\frac{1}{2^n\lambda}\right)^{\tilde s}\norm{u}_1.
\end{align*}

\noindent With this in mind, we can use $\int \hat\phi(u) T^{r-1}_{x_0}(u) du = 0$ to get

\begin{align*}
\abs{\scalar{\hat\phi_y^n,\psi_z^\lambda}} &= 2^{-n\frac {\abs S}2}\lambda^{-\abs S}\abs{\int\hat\phi(u) \psi\Big(S^{\frac 1{2^n\lambda}}u\underbrace{-S^{\frac 1\lambda}z+S^{2^n}y}_{=:x_0}\Big) du}\\
&\le 2^{-n\frac {\abs S}2}\lambda^{-\abs S}\int\abs{\hat\phi(u)}\underbrace{\abs{\psi\left(S^{\frac 1{2^n\lambda}}u+x_0\right)-T_{x_0}^{r-1}\left(S^{\frac 1{2^n\lambda}}u+x_0\right)}}_{\lesssim \norm{S^{\frac 1{2^n\lambda}}u}_1^r\norm{\psi^{(r)}}_{\infty}\lesssim 2^{-n\tilde r}\lambda^{-\tilde r}\norm{u}_1^r\norm{\psi^{(r)}}_{\infty}}du \\&\lesssim 2^{-n\frac {\abs S}2-n\tilde r}\lambda^{-{\abs S}-\tilde r} \int\abs{\hat\phi(u)}\norm{u}_1^r\norm{\psi^{(r)}}_{\infty}du \\&\lesssim 2^{-n\frac {\abs S}2-n\tilde r}\lambda^{-\abs S-\tilde r}\norm{\psi}_{C_c^r}.
\end{align*}
\end{proof}

\noindent While this lemma is mostly technical, let us only remark that \eqref{LemIneq2} shows that $\hat\phi$ eliminating polynomials causes $\scalar{\psi,\hat\phi_k^n}$ to converge faster to zero than one would expect. In fact, by choosing $r$ big enough, we can cause $\scalar{\psi,\hat\phi_k^n}$ to converge at any rate we want.

Throughout this paper, we have used that $\eqref{wavelet}$ and $\eqref{eq:hatPhi}$ hold in more topologies than just $L_2(\R^d)$-convergence.  To show this, we will use a wavelet basis in $C_c^q$ for some $q>r$. Since $\phi,\hat\phi\in C_c^q\subset C_c^r$ holds for all $\hat\phi\in \Phi$, it is easy to show that any $C_c^q$ wavelet basis is automatically a $C_c^r$ wavelet basis. So while we are technically using $C_c^q$ wavelets to prove the main results of this paper, the test functions $\psi$ lie in $C_c^r$ and we treat $(\phi,\Phi)$ as $C_c^r$ wavelets.

\begin{lemma}\label{lem:PnPsiConverges}
Let $q$ be such that $\tilde q>r+\abs S$ and choose a wavelet basis $(\phi,\Phi)$ of $C_c^{q}$. Then for all $\psi\in C_c^{q}$, it holds that 

\begin{equation}\label{PnPsiConverges}
P_n\psi\xrightarrow{n\to\infty} \psi
\end{equation}

\noindent
converges in $C_c^r$.

Furthermore, let $\alpha<0$ such that $\tilde r+\alpha>0$ and consider a distribution $f\in C^\alpha(\R^d)$. Then for any test function $\psi\in C_c^r$, it holds that

\begin{align}
f(\psi) &= \lim_{n\to\infty} f(P_n\psi) \label{eq:convergence1}\\
&= f(P_n(\psi)) +\sum_{m\ge n} f(\hat P_m\psi)\,,\label{eq:convergence2}
\end{align}

\noindent where \eqref{eq:convergence2} holds for any $n\in\mathbb N$. If $f\in C^\alpha(L_p)$ for some $1\le p<\infty$, both identities hold as limits in $L_p$. 

%If $f\in C^\alpha([x,\infty))$ for some $x\in\R^d$, \eqref{eq:convergence1} and \eqref{eq:convergence2} hold for all $\psi$ with support in $[x_1+\epsilon,\infty)\times\dots\times[x_d+\epsilon,\infty)$ for any $\epsilon>0$.
\end{lemma}

\begin{proof}
\noindent We begin by showing \eqref{PnPsiConverges}. Let $k\ge l$ be two natural numbers. Then it holds that

\begin{equation*}
P_k\psi = P_l\psi +\sum_{m=l}^{k-1}\hat P_m\psi\,,
\end{equation*}

\noindent so to show that $P_n\psi$ is Cauchy, it suffices to show $\norm{\hat P_m\psi}_{C_c^r}\lesssim 2^{-m(\tilde q-r-\abs S)}\norm{\psi}_{C_c^q}$. By Remark \ref{rem:allDistAreCalpha}, $\norm{\hat\phi_x^m}_{C_c^r}\lesssim 2^{m\left(r+\frac{\abs S}{2}\right)}$ holds for all $m\in\mathbb N, x\in\Delta_m$. Recall that $\psi\in C_c^q$, so we can use \eqref{LemIneq2} with $q$ to show

\begin{align*}
\norm{\hat P_m\psi}_{C_c^r} &\le \sum_{x\in\Delta_m} \abs{\scalar{\hat\phi_x^m,\psi}}\norm{\hat\phi_x^m}_{C_c^r} \\
&\lesssim \sum_{x\in\Delta_m, \scalar{\hat\phi_x^m,\psi}\neq 0} 2^{-m(\tilde q-r)}\norm{\psi}_{C_c^q}\,.
\end{align*}

\noindent
Since the number of non-zero terms is of order $2^{m\abs S}$, the claim follows. Thus, $P_n\psi$ is Cauchy and thus convergent in $C_c^r$. Since $P_n\psi\to\psi$ in $L_2(\R^d)$ as $n\to\infty$, the limit can only be $\psi$. It follows that $P_n\psi\xrightarrow{n\to\infty}\psi$ holds in $C_c^r$.

For the second part of the lemma, we will only show the deterministic case, for the $L_p$ case one simply replaces the absolute values with the $L_p$ norms. Observe that for any $0\le n\le l$,

%It further suffices to prove the case $f\in C^\alpha([x,\infty))$, since $C^\alpha = \bigcap_{x\in\R^d} C^\alpha([x,\infty))$. Let $x\in\R^d$, $f\in C^\alpha([x,\infty))$ and $\psi\in C_c^r$ with compact support as in the lemma. 

\begin{align*}
\sum_{k\in\Delta_l}f(\phi_k^l)\scalar{\phi_k^l,\psi} &= f(P_l\psi) \\
&= f(P_n\psi)+\sum_{m=n}^l f(\hat P_m\psi)\\
&= \sum_{k\in\Delta_n}f(\phi_k^n)\scalar{\phi_k^n,\psi}+\sum_{m=n}^l\sum_{\hat\phi\in\Phi}\sum_{k\in\Delta_m}f(\hat\phi_k^m)\scalar{\hat\phi_k^m,\psi}\,,
\end{align*}

\noindent which shows that \eqref{eq:convergence1} and \eqref{eq:convergence2} are the same limit, and it suffices to show one of them. Using \eqref{LemIneq2}, we get that

%As mentioned in the remark, the right-hand-side might not be well-defined if $n,l$ are too small, since $\hat\phi_k^m,\phi_k^n$ do not need to have support in $[x,\infty)$. However, for $n$ large enough, such that $2^-{ns_i}\le \epsilon$ for $i=1\dots, d$, it holds that for all terms with $\scalar{\phi_k^n,\psi}\neq 0$ $\phi_k^n$ has support in $[x,\infty)$, so the term is well-defined. It also follows that the $f(\hat\phi_k^m)$ terms are well-defined, since $m\ge n$. 

\begin{equation*}
\sum_{k\in\Delta_m} \abs{f(\hat\phi_k^m)\scalar{\hat\phi_k^m,\psi}}\lesssim 2^{-m(\alpha+\tilde r)}\norm{\psi}_{C_c^r}\,,
\end{equation*}

\noindent where we use that the number of non-zero terms in the sum is of order $2^{m\abs{S}}$. It follows that \eqref{eq:convergence2} converges for $l\to\infty$, and $\abs{f(P_n\psi)}\lesssim 2^{-n\alpha}\norm{\psi}_{C_c^r}$ implies that the limit is a $C_c^r$ distribution.

We call the limit $g$ and it remains to show that $g=f$. By \eqref{PnPsiConverges}, $f(P_n\tilde\psi)\to f(\tilde\psi)$ as $n\to\infty$ holds for $\tilde\psi\in C_c^{q}$ for large enough $q$, which implies that $f(\tilde\psi)=g(\tilde\psi)$. Since $C_c^{q}\subset C_c^r$ is dense, it follows that $f(\psi) = g(\psi)$ for all $\psi\in C_c^r$.
\end{proof}

\section{A BDG-type inequality}

Let us introduce the inequality that allows us to ``borrow regularity'' from the conditional expectation of a process: Let $(\Omega,\mathcal F,P)$ be a complete probability space. Given a process $(Z_i)_{i\in\mathbb N}\subset L_p(\Omega)$ in discreet time that is adapted to $\mathcal F_{i+1}$ for some filtration $\mathcal F_i$, we can use Doobs decomposition \cite{doob} on the sum

\begin{equation*}
\sum_{i=1}^N Z_i = \sum_{i=1}^N E^{\mathcal F_i}Z_i + \sum_{i=1}^N\left(Z_i- E^{\mathcal F_i}Z_i\right)\,.
\end{equation*}

\noindent This allows us to bound $\norm{\sum_{i=1}^N Z_i}_{L_p}$ by $\sum_{i=1}^N\norm{E^{\mathcal{F}_i}Z_i}_{L_p}$ plus $\norm{\sum_{i=1}^N\left(Z_i- E^{\mathcal F_i}Z_i\right)}_{L_p}$. Since the last term is a sum over martingale differences, we can find a better bound than simply pulling the $L_p$-norm into the sum with BDG-type inequalities, see \cite{BDG}. So as long as the first sum (i.e. the conditional expectation of $Z_i$) is well-behaved, we can get a good bound on $\norm{\sum_{i=1}^N}_{L_p}$. Formally, this reads as:

\begin{lemma}
Let $p\ge 2$ and for $i\in\mathbb N$ let $Z_i\in L_p(\Omega)$. Let $(\mathcal F_i)_{i\in\mathbb N}$ be a filtration, such that $Z_i$ is $\mathcal F_{i+1}$-measurable for all $i\in\mathbb N$. Then, the following inequality holds for all $N\ge 1$:

\begin{equation}\label{BDG}
\norm{\sum_{i=1}^N Z_i}_{L_p} \lesssim\sum_{i=1}^N\norm{E^{\mathcal F_i}Z_i}_{L_p} + \left(\sum_{i=1}^N \norm{Z_i-E^{\mathcal F_i}Z_i}_{L_p}^2\right)^{\frac 12}\,.
\end{equation}
\end{lemma}

\noindent
The proof can be found in \cite{le}, at the beginning of Section 2.

\section{Kolmogorov's continuity theorem for negative Hölder regularity}

We would like to present a version of Kolmogorov's continuity theorem for distributions $f\in C^\alpha(L_p)$ with negative $\alpha$. To do so, we need the fattening of a compact set $K\subset\R^d$, which for every $\epsilon>0$ is defined to be

\begin{equation*}
K_\epsilon := \left\{y\in\R^d~\middle\vert~ \inf_{x\in K}\abs{x-y}\le\epsilon\right\}\,.
\end{equation*}

\noindent
With this, we can formulate the following lemma:

\begin{lemma}\label{lem:Kolmogorov}
Let $K\subset \R^d$ be compact and $\alpha<0, 1\le p<\infty$. Assume that $f\in C^\alpha(L_p(\Omega),K_\epsilon)$ for some $\epsilon>0$. Let $\hat\alpha<\alpha-\frac{\abs S}{p}$ and choose $\hat r$ such that $\tilde{\hat r}>-\hat\alpha$, where $\tilde{\hat r}$ is given as in Lemma \ref{lemma1}. Then there is a non-negative random variable $B\in L_p(\Omega)$ such that for each $\psi\in C_c^{\hat r}$ with $\norm{\psi}_{C_c^{\hat r}} = 1$ and $x\in\R^d, \lambda\in(0,1]$ such that $\supp(\psi_x^\lambda)\subset K$ it holds that

\begin{equation}\label{ineq:Kolmogorov}
\abs{f(\psi_x^\lambda,\omega)}\le B(\omega)\lambda^{\hat\alpha}\,.
\end{equation}

\noindent
It follows that $f(\cdot,\omega)\in C^{\hat\alpha}(K)$ a.s.
\end{lemma}

\begin{proof}
We will assume throughout this proof that $\lambda<\tilde\epsilon := \min_{i=1,\dots,d}\left(\frac{\epsilon}{R-C}\right)^{\frac 1{s_i}}$, where $0<C<R$ are given by a wavelet bases. Extending the result to $\lambda\in(0,1]$ simply adds a factor only depending on $\epsilon, C,R$ to all calculations. Let $(\phi,\Phi)$ be $C_c^{\hat r}$ wavelets and set

\begin{align*}
B_n &:= \sup_{\supp(\phi_x^n)\subset K_{\epsilon}} \frac{\abs{f(\phi_x^n)}}{c_n} \\
\hat B_n &:= \sup_{\hat\phi\in\Phi,\supp(\hat\phi_x^n)\subset K_{\epsilon}} \frac{\abs{f(\hat\phi_x^n)}}{c_n}\,,
\end{align*}

\noindent with normalizing factor $c_n = 2^{-n(\frac 12-\frac 1p)\abs S-n\alpha}$. We can calculate the $L_p$ norm of these terms to be bound by the constant

\begin{align*}
\norm{B_n}_{L_p}^p &\le \sum_{\supp(\phi_x^n)\subset K_{\epsilon}} \norm{f(\phi_x^n)}_{L_p}^p c_n^{-p} \\
&\lesssim \norm{f}_{C^{\alpha}(L_p)}^p\,,
\end{align*}

\noindent where we use that the number of $x\in\R^d$, such that $\supp(\phi_x^n)\subset K_\epsilon$ is of order $2^{n\abs{S}}$. Analogously, $\norm{\hat B_n}_{L_p} \lesssim \norm{f}_{C^{\alpha}(L_p)}$ is bound by a constant. 

Let $\lambda\in(0,\tilde\epsilon]$, $\psi\in C_c^{\hat r}$ and $x\in\R^d$ such that $\supp(\psi_x^\lambda)\subset K$. Choose $n$ in such a way, that $2^{-n}\le\lambda<2^{-n+1}$. By our choice of $\tilde\epsilon$, it follows that for all $\phi_y^n, \hat\phi_y^n$ such that their supports overlap with the support of $\psi_x^\lambda$, we have $\supp(\phi_y^n),\supp(\hat\phi_y^n)\subset K_{\epsilon}$. We calculate, that

\begin{align*}
\abs{f(\psi_x^\lambda)} &= \sum_{y\in\Delta_n} \abs{f(\phi_y^n)\scalar{\phi_y^n,\psi_x^\lambda}}+\sum_{\hat\phi\in\Phi}\sum_{m\ge n}\sum_{y\in\Delta_m} \abs{f(\hat\phi_y^m)\scalar{\hat\phi_y^m,\psi_x^\lambda}} \\
&\le B_nc_n\sum_{y\in\Delta_n}\abs{\scalar{\phi_y^n,\psi_x^\lambda}}+\sum_{m\ge n}\hat B_mc_m \sum_{y\in\Delta_m}\abs{\scalar{\hat\phi_y^m,\psi_x^\lambda}}\,.
\end{align*}

\noindent Set $\kappa := \alpha-\frac{\abs{S}}{p}-\hat\alpha > 0$. By using \eqref{LemIneq1} and \eqref{LemIneq2} and carefully counting the number of non-zero terms, we get that

\begin{align*}
\abs{f(\psi_x^\lambda)} &\lesssim \left(B_n +\sum_{m\ge n} \hat B_m 2^{(-m-n)(\alpha-\frac{\abs{S}}p+\tilde{\hat r})}\right)\lambda^{\alpha-\frac{\abs S}p}\\
&\approx \left(B_n +\sum_{m\ge n} \hat B_m 2^{-(m-n)(\alpha-\frac{\abs{S}}p+\tilde{\hat r})}\right)2^{-n\kappa}\lambda^{\hat \alpha}\,,
\end{align*}

\noindent so by choosing

\begin{equation*}
B := D\sup_{n\in\mathbb N}~ 2^{-n\kappa}\left(B_n +\sum_{m\ge n} \hat B_m2^{-(m-n)(\alpha-\frac{\abs S}p+\tilde{\hat r})}\right)
\end{equation*}

\noindent for a smartly chosen constant $D>0$, we get \eqref{ineq:Kolmogorov}. It remains to show that $B\in L_p(\Omega)$, which can be easily seen from the calculation

\begin{align*}
\norm{B}_{L_p} &\le D\sum_{n\in\mathbb N}~2^{-n\kappa}\left(\norm{B_n}_{L_p} +\sum_{m\ge n} \norm{\hat B_m}_{L_p}2^{-(m-n)(\alpha-\frac{\abs S}p+\tilde{\hat r})}\right)\\
&\lesssim D\sum_{n\in\mathbb N} 2^{-n\kappa}\left(1+\sum_{m\ge n}2^{-(m-n)(\alpha-\frac{\abs S}p+\tilde{\hat r})}\right)\norm{f}_{C^\alpha(L_p)}\\
&\lesssim D\norm{f}_{C^\alpha(L_p)}\,.
\end{align*}
\end{proof}

\section*{Acknowledgments}

I want to thank Peter Friz for suggesting this topic to me, and for many helpful discussions throughout the process of writing this paper. I further want to thank Khoa Lê and Philipp Forstner for their helpful remarks, and Carlo Bellingeri for our mathematical discussions. Last but not least, I want to thank the anonymous reviewers whose feedback greatly improved the quality of this paper.

\section*{Funding}
 This project was supported by the IRTG 2544, which is funded by the DFG.

\printbibliography

\end{document}